\documentclass[reqno]{amsart}
\usepackage{graphicx}
\usepackage{amsmath}
\usepackage{cite}
\usepackage[english]{babel}

\usepackage{amsfonts}
\usepackage{amssymb}
\usepackage{bbm}
\usepackage{epstopdf}
\numberwithin{equation}{section}

\newcommand{\cR}{\mathcal R}

\newcommand{\cC}{\mathcal C}

\newcommand{\cT}{\mathcal T}
\newcommand{\cH}{\mathcal H}

\newcommand{\cS}{\mathcal S}

\newcommand{\bbT}{\mathbb{T}}

\renewcommand{\P}{\mathbb{P}}
\newcommand{\E}{\mathbb{E}}
\newcommand{\C}{\mathbb{C}}
\newcommand{\R}{\mathbb{R}}
\newcommand{\cO}{\mathcal{O}}

\newcommand{\tb}{\bullet}
\newcommand{\tw}{\circ}
\newcommand{\frb}{\mathfrak{b}}
\newcommand{\frw}{\mathfrak{w}}

\newcommand{\wc}{u^\tw}
\newcommand{\bc}{u^\tb}
\newcommand{\wpr}{w}
\newcommand{\bpr}{b}

\usepackage{amsthm}
\newtheorem{theorem}{Theorem}[section]

\newtheorem{remark}[theorem]{Remark}
\newtheorem{corollary}[theorem]{Corollary}

\newtheorem{assumption}[theorem]{Assumption}
\newtheorem*{assumption*}{Assumption}

\newtheorem{lemma}[theorem]{Lemma}
\newtheorem{definition}[theorem]{Definition}
\newtheorem{proposition}[theorem]{Proposition}

\renewcommand{\Im}{\operatorname{Im}}
\renewcommand{\Re}{\operatorname{Re}}
\newcommand{\G}{\mathcal{G}}
\newcommand{\rI}{\mathrm{I}}
\newcommand{\rD}{\mathrm{D}}

\DeclareMathOperator{\area}{Area}
\DeclareMathOperator{\var}{Var}

\newcommand\LipKd[2]{{\mbox{\textsc{Lip(}$#1,#2$\textsc{)}}}}
\newcommand\ExpFat[1]{{{\mbox{\textsc{Exp-Fat(}$#1$\textsc{)}}}}}
\newcommand\ExpFatPrime[2]{{{\mbox{\textsc{Exp-Fat(}$#1,#2$\textsc{)}}}}}

\title[Dimer model and holomorphic functions on t-embeddings]{Dimer model and holomorphic functions\\ on t-embeddings of planar graphs}

\author[Dmitry Chelkak]{Dmitry Chelkak$^\mathrm{a,b,c}$}

\author[Beno\^{\i}t Laslier]{Beno\^{\i}t Laslier$^\mathrm{d}$}

\author[Marianna Russkikh]{Marianna Russkikh$^\mathrm{e,f}$}

\thanks{\textsc{${}^\mathrm{A}$ ENS--MHI chair funded by MHI. D\'epartement de math\'ematiques et applications, \'Ecole Normale Sup\'erieure, CNRS, PSL University, 45 rue d'Ulm, 75005 Paris, France.}}

\thanks{{\textsc{${}^\mathrm{B}$} \emph{Current address:} \textsc{Department of Mathematics, University of Michigan,
Ann Arbor, MI 48109-1043, USA}}}

\thanks{\textsc{${}^\mathrm{C}$ \emph{On leave from} St.~Petersburg Dept. of Steklov Mathematical Institute RAS, Fontanka 27, 191023 St.~Petersburg, Russia.}}

\thanks{\textsc{${}^\mathrm{D}$ {Universit\'e Paris Cit\'e}, Sorbonne Universit\'e, CNRS, Laboratoire de Probabilit\'es, Statistiques et Mod\'elisations (LPSM), Paris, France}}

\thanks{\textsc{${}^\mathrm{E}$ Massachusetts Institute of Technology, Department of Mathematics, 77 Massachusetts Avenue, Cambridge, Massachusetts, 02139–4307}}

\thanks{{\textsc{${}^\mathrm{F}$} \emph{Current address:} \textsc{Department of Mathematics, California Institute of Technology, Pasadena,
CA 91125, USA}}}

\thanks{\emph{E-mail:} {\texttt{dchelkak@umich.edu}, \texttt{laslier@lpsm.paris}, \texttt{russkikh@caltech.edu}}}

\begin{document}

\begin{abstract} We introduce the framework of discrete holomorphic functions on t-embeddings of weighted bipartite planar graphs; t-embeddings also appeared under the name Coulomb gauges in a recent {paper}~\cite{KLRR}. We argue that this framework is particularly relevant for the analysis of scaling limits of the height fluctuations in the corresponding dimer models. In particular, it unifies both Kenyon's interpretation of dimer observables as derivatives of harmonic functions on T-graphs and the notion of s-holomorphic functions originated in Smirnov's work on the critical Ising model. We develop an a priori regularity theory for such functions and provide a meta-theorem on convergence of the height fluctuations to the Gaussian Free Field. We also discuss {how several more standard discretizations of complex analysis fit this general framework.}
\end{abstract}

\keywords{dimer model, discrete holomorphicity, Gaussian free field}

\subjclass[2010]{82B20, 30G25}

\maketitle

\newpage

\tableofcontents

\newpage

\section{Introduction}

\subsection{General context} \label{sub:gen-context}
This paper contributes to two subjects: the dimer model on bipartite planar graphs and the discrete complex analysis techniques in probability and statistical physics. Both topics are very rich, we refer {an} interested reader to~\cite{kenyon-lectures,gorin-lectures} and~\cite{smirnov-icm2010} and {references} therein, respectively. Though the two subjects are known to be intimately related, it should be said that many other powerful techniques were successfully applied to studying the dimer model, e.g., see~\cite{petrov,bufetov-gorin,colomo-sportiello,BLR1,aggarwal,giuliani-mastropietro-toninelli,dubedat-SL2,BdTR-via-dimers} and references therein to mention some of the important achievements obtained during the last decade. In particular, in the
last years there was a widespread feeling that discrete complex analysis {ideas} had almost
reached the limit of their capacity to bring new interesting results in the bipartite dimer model context. In this paper and {its} follow-up~\cite{CLR2} we intend to revive the link between the two topics; see also~\cite{chelkak-s-emb} for a companion research project on the planar Ising model.

It is well known that entries of the inverse Kasteleyn matrix (also known as the \emph{coupling function}) of the {homogeneous} dimer model on the square grid satisfy the most straightforward discrete version of the Cauchy-Riemann equation. This observation was used by Kenyon in~\cite{kenyon-gff-a,kenyon-gff-b} to prove the convergence of the height fluctuations to the Gaussian Free Field for the so-called Temperleyan discretizations of planar domains. This classical result was among the very first rigorous proofs of the convergence of lattice model observables, considered in discrete domains on~$\delta\mathbb{Z}^2$ approximating a continuous domain~$\Omega$ as~$\delta\to 0$, to conformally invariant quantities. A few years later, a similar treatment of the critical Ising model on the square grid appeared in the work of Smirnov~\cite{smirnov-annals}. Smirnov's approach, in particular, relied upon a specific reformulation of the discrete Cauchy--Riemann equations on~$\mathbb{Z}^2$. This reformulation is now commonly known as the \mbox{\emph{s-holomorphicity}} property, a term coined in the {paper}~\cite{chelkak-smirnov-universality} devoted to a generalization of Smirnov's results to the Z-invariant critical Ising model on isoradial {grids}. Another, at first sight unrelated, discretization of the Cauchy--Riemann equations was suggested in~\cite{dynnikov-novikov} {which in fact boils down to the linear relations satisfied by the} coupling functions on the honeycomb grid.

However, a naive interpretation of discrete Cauchy--Riemann equations is known to be often misleading even in the context of regular lattices, like the square or the honeycomb ones. Though it works well in several contexts (critical Ising model, dimers in {Temperleyan-type} domains), the dimer model observables are known \emph{not} to have holomorphic scaling limits in other situations, in particular if the Cohn--Kenyon--Propp limit shape surface~\cite{cohn-kenyon-propp} is not horizontal. The intrinsic reason for such a mismatch is that we expect the scaling limit to live in a less trivial complex structure than the one suggested by the naive discretization of Cauchy-Riemann equations. This {effect} manifests itself by the fact that quantities like the {entries of the} inverse Kasteleyn matrix that, {in principle,} could have holomorphic limits {do \emph{not} remain uniformly bounded even locally and, in particular, do not converge as $\delta \to 0$}. Instead, they grow exponentially with the number of steps in a way reminiscent of {generic} discrete harmonic functions.
This raises a question of finding a framework in which, on the one hand, this exponential growth is removed and, on the other hand, the new discrete equations are compatible with a \emph{nontrivial} continuous complex structure arising {in the limit}; see~\cite{kenyon-okounkov} for the description of this complex structure {via} the limit shape surface {for doubly periodic dimer models.}

Developing this idea in~\cite{kenyon-honeycomb}, Kenyon introduced a framework of holomorphic functions on \mbox{\emph{T-graphs}} (combinatorial objects first discussed in~\cite{kenyon-sheffield}) in order to analyze the behavior of dimer model observables in the non-horizontal case. In particular, this paper already contained an idea of embedding a given \emph{abstract} planar graph~$\G^\delta$ (a piece of the honeycomb grid in that case) into the complex plane as a T-graph so that discrete observables approximate holomorphic functions in the metric of these embeddings~$\Omega^\delta\subset\C$. This procedure, in particular, requires a proper choice of the \emph{gauge function}, a transformation of the dimer weights which leave the law of the model unchanged. The gauge function is responsible for the removal of the local exponential growth {of dimer coupling functions,} which varies from point to point in the original metric and should be evened out. {We refer an interested reader to a recent paper~\cite{laslier-21} for an extensive discussion of this approach.}

A more geometric viewpoint on `nice' gauge functions, the so-called \emph{Coulomb gauges}, {was suggested in~\cite{KLRR}. These gauges} have many remarkable algebraic properties (see also~\cite{affolter}) and are also closely related to T-graphs mentioned above. In parallel, a notion of \emph{s-embeddings} of graphs carrying the planar Ising model was suggested in~\cite{chelkak-icm2018}. As explained in~\cite[Section~7]{KLRR}, the latter are a particular case of the former under the combinatorial bosonization correspondence of the two models~\cite{dubedat-bosonization}. {The notion of} \emph{t-embeddings} discussed in our paper {is} fully equivalent to Coulomb gauges of~\cite{KLRR} except that we focus on embeddings of the dual graphs~$(\G^\delta)^*$ from the very beginning. The appearance of another name for the same object is caused by the fact that we were not aware of the research of~\cite{KLRR} at the beginning of this project and arrived at the same concept aiming to generalize results obtained for the Ising model observables on s-embeddings to dimers.

Whilst the work~\cite{KLRR} is focused on algebro-geometric properties of {t-embeddings,} our paper is devoted to the study of discrete holomorphic functions on such graphs, which we will call \emph{t-holomorphic} functions to distinguish from other discretizations of complex analysis. In particular, {our framework generalizes (a part of) the discrete complex analysis techniques recently developed in} \cite{chelkak-icm2018,chelkak-s-emb} in the planar Ising model context. We are particularly interested in the behavior of t-holomorphic functions in the `small mesh size' limit. It is worth noting that 
we do \emph{not} rely upon usual `uniformly bounded angles, degrees or sizes of faces' assumptions. In particular, the notion of the scale~$\delta$ of a t-embedding~$\cT^\delta$ requires a more invariant definition, which is discussed below. Also, note that for the case of convergence of harmonic functions on \emph{circle packings} or, more generally, on orthodiagonal quadrangulations, similar technical assumptions were recently fully dropped in~\cite{gurel-jerison-nachmias}. However, the notion of harmonicity on T-graphs associated {with} t-embeddings is formulated in terms of \emph{directed} random walks (and not via conductances), which significantly changes the perspective. Nevertheless, among other things we prove {the} a priori Lipschitzness of harmonic functions under a mild assumption \ExpFat{\delta}\ formulated below.

One of the long-term motivations to get rid of `technical' assumptions mentioned above is to {develop a discrete complex analysis framework that could be eventually applied to} \emph{random planar maps} weighted by the critical Ising or {by} the bipartite dimer model. We believe that s- and t-embeddings of abstract {weighted} planar graphs are the right tools to attack these questions. This perspective is somehow similar to the idea~\cite{smirnov-boykiy-private} of using \emph{square tilings} to study random planar maps weighted by uniform spanning trees; in this case an alternative option could be to use \emph{Tutte's harmonic embeddings} which are also related to the context of t-embeddings via~\cite[Section~6.2]{KLRR}. {On deterministic graphs,} we believe that the discrete complex analysis viewpoint on Kasteleyn equations provided in our paper is flexible enough to be applied {in} rather general situations, both in terms of the underlying lattice and of the limit shape surfaces; e.g. see~\cite{chelkak-ramassamy} where the case of the classical Aztec diamond is discussed from this perspective. In particular, our paper unifies Smirnov's concept of s-holomorphic functions and Kenyon's interpretation of dimer model observables as derivatives of harmonic functions on T-graphs. {We also refer an interested reader to {Section~\ref{sec:appendix2},} in which several links {between} the t-embeddings framework {developed in our paper and} more standard discretizations of complex analysis are discussed.}

\begin{figure}
\includegraphics[clip,trim=6.2cm 13.1cm 5.7cm 7.8cm, width=0.65\textwidth]{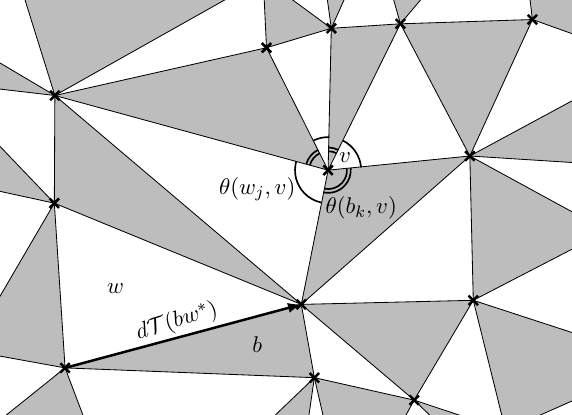}
\caption{\label{Temb} A portion of a t-embedding $\cT$ 
{and the notation $d\cT(bw^*)$.}
The angle condition $ \sum_{{j=1}}^{{n}} \theta( w_j, v) = \sum_{{k=1}}^{{n}} \theta(b_k, v) = \pi$ around a vertex $v$ {of degree~$2n=6$} is highlighted. The {Kasteleyn matrix of the dimer model on \emph{faces} of~$\cT$ is~$K_\cT(b,w):=d\cT(bw^*)$.}}
\end{figure}

\subsection{Basic concepts and assumptions}
We now briefly recall the setup of t-embeddings or Coulomb gauges, see Section~\ref{sec:setup} and~\cite{KLRR} for more details. Let~$\G$ be a weighted bipartite graph carrying the dimer model; the latter is a random choice of a perfect matching of vertices of~$\G$ with a probability proportional to the product of the corresponding positive weights. {We call the two bipartite classes of vertices of~$\G$ `black' and `white' and denote them $B$ and~$W$ in what follows.} Assume that all vertices of~${\G=B\cup W}$ have degree at least three and that the graph~$\G$ is planar. A t-embedding~$\cT$ is a proper embedding of the \emph{dual} graph~$\G^*$ into the complex plane such that
\begin{itemize}
\item all edges {of~$\cT$} are straight segments {and} all faces {of $\cT$} are convex polygons,
\item the geometric weights given by {the} lengths of edges {of~$\cT$} are gauge equivalent to the original {dimer weights,}
\end{itemize}
and the following \emph{angle condition} holds (see Fig.~\ref{Temb}):
\begin{itemize}
\item for each inner vertex~$v$ of~$\cT$, the sum of angles of black faces adjacent to~$v$ {(which correspond to black vertices of~$\G$ adjacent to a given face)} equals~$\pi$.
\end{itemize}

If~$\G$ is a finite {planar} graph {with the sphere topology (or, more accurately, a \emph{planar map,} i.e., a proper embedding of~$\G$ into the sphere considered up to homotopies),} one should be more accurate and first specify {an `outer'} face of~$\G$, {further replacing} the corresponding vertex~$v_\mathrm{out}$ of~$\G^*$ by a cycle of length~$\deg v_\mathrm{out}$. The graph thus obtained is called the \emph{augmented dual} in~\cite{KLRR}, we still denote it by~$\G^*$. A finite t-embedding is an embedding of this augmented dual graph~$\G^*$ and the angle condition is dropped at boundary vertices of~$\G^*$; see Fig.~\ref{finTembTriangl}.

\begin{figure}
\centering
\begin{minipage}{0.4\textwidth}
\includegraphics[clip, trim=4.4cm 10.4cm 6.3cm 4cm,width=\textwidth]{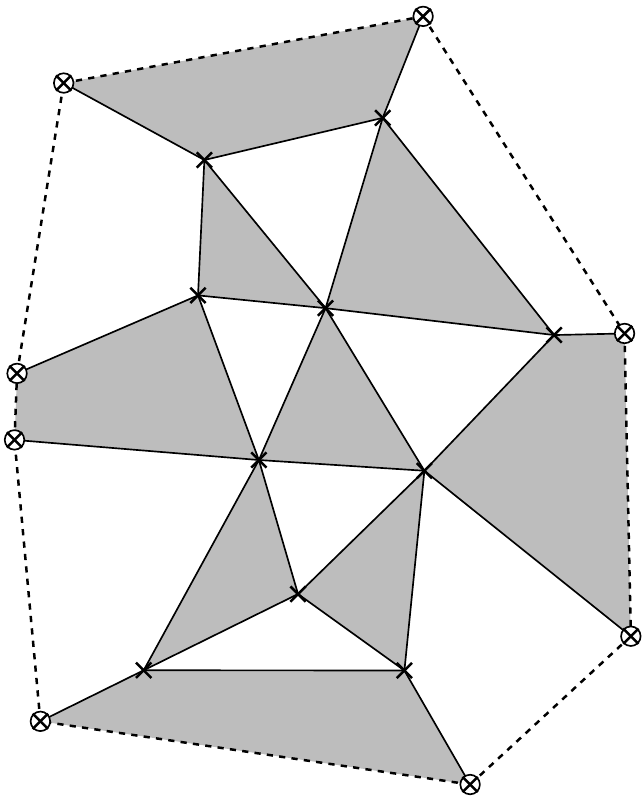}
\end{minipage}
\hskip 0.08\textwidth
\begin{minipage}{0.34\textwidth}
\includegraphics[clip, trim=4.4cm 9.7cm 7.4cm 4.1cm,width=\textwidth]{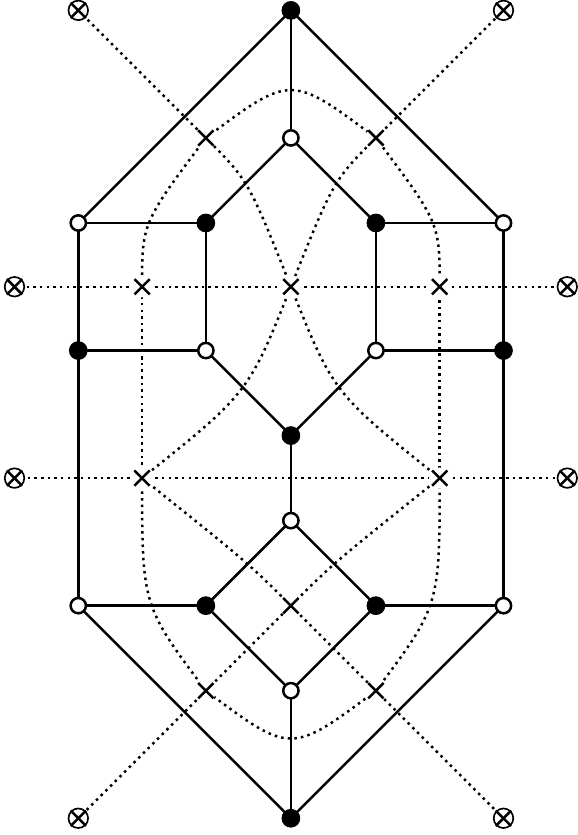}
\end{minipage}
\caption{\textsc{Left:} A finite t-embedding, the boundary edges are dashed.
\textsc{Right:} {A bipartite graph~$\G$ and its augmented dual} $\G^*$ (dotted). We call $\cT(\G^*)$ a finite triangulation if all vertices of $\G$ have degree three. {In this} {case interior faces of~$\cT$ are triangles whilst boundary ones are quadrilaterals.}}\label{finTembTriangl}
\end{figure}

In this paper, we do \emph{not} discuss the existence of t-embeddings of a given abstract planar graph~$\G$ carrying the bipartite dimer model. This question was addressed in~\cite{KLRR}, we quote some of these results below, see Theorem~\ref{thm:KLRR-existence}. Overall, {we believe that all finite planar} bipartite graphs admit many t-embeddings if no constraints are imposed at {the} \emph{boundary} vertices of~$\cT$. The interested reader is also referred to our follow-up paper~\cite{CLR2} for a notion of \emph{`perfect'} t-embeddings {of finite graphs,} which specifies additional constraints on the boundary of~$\cT$ in a way that {potentially} provides both the existence and the uniqueness {of such an embedding} up to a few natural isomorphisms. In particular, we consider {`perfect'} t-embeddings from~\cite{CLR2} as a very important application of the framework developed in this paper; see also remarks after Theorem~\ref{thm:main-GFF}.

To summarize the preceding discussion, in {what follows} we view a t-embed\-ding $\cT$ as an object given \emph{in advance}, and then study the dimer model on faces of~$\cT$ with weights given by edge lengths. The angle condition easily implies (see~\cite{KLRR} or Section~\ref{sec:setup} below) that the matrix~$K(b,w):=d\cT(bw^*)$ is a Kasteleyn matrix for this dimer model, see Fig.~\ref{Temb} for the notation.
Loosely speaking, \emph{t-holomorphic functions} on~$\cT$ are just functions satisfying the Kasteleyn relations locally. Note however that this down-to-earth interpretation is not the best possible one, we refer the reader to Section~\ref{sec:holomorphicity} for precise definitions and a discussion.

The central concept for our analysis is the \emph{origami map}~$\cO$ associated to a t-embedding~$\cT$. {(The name `origami' for this map is motivated by~\cite{KLRR}, the reason is that tilings of the plane satisfying the angle condition coincide with the crease patterns of origami that are locally flat-foldable, see~\cite{hull-origami}.)} Let~$z$ be the complex coordinate in the plane in which the graph~$\cT$ is drawn. Informally speaking, to construct the mapping~$z\mapsto\cO(z)$ out of~$\cT$, one folds this plane along each of the edges of~$\cT$. The angle condition guarantees that this folding procedure is locally and hence globally consistent; we refer the reader to Section~\ref{sec:setup} for an accurate definition. Note that~$\cO$ is defined up to translations, rotations and a possible reflection; in our convention the white faces preserve their orientation in~$\cO$ while the black ones change it. If one starts with the square lattice~$\delta\mathbb{Z}^2$, then the image of the origami map is just a single square of size~$\delta$. Similarly, if one starts with the regular triangular lattice (which corresponds to the dimer model on the honeycomb grid), then the image of~$\cO$ is just a single {equilateral} triangle. However, already for a skewed triangular lattice the map~$\cO$ becomes less trivial though its image is still bounded. Surprisingly enough, the origami map of these triangular lattices also appeared in the dynamical systems context recently~\cite{romaskevich}. The origami map also gives a link between t-embeddings and \emph{T-graphs}: the latter are just the images of the t-embedding under the mappings~$z\mapsto z+\alpha^2\cO(z)$ or~$z\mapsto z+\overline{\alpha^2\cO(z)}$, $\alpha\in\mathbb{T}$, where~$\mathbb{T}:=\{\alpha\in\C:|\alpha|=1\}$. 
We use the notation~$\cT+\alpha^2\cO$ and~$\cT+\overline{\alpha^2\cO}$ for these T-graphs, see Section~\ref{sub:t-graphs} for details.

Clearly, {the mapping $z\mapsto\cO(z)$} does not increase {Euclidean} distances {in the complex plane,} i.e., is a~$1$-Lipschitz function. The main assumption for our analysis is that {this mapping} has a slightly better Lipschitz constant at large scales:

\begin{assumption}[\LipKd{\kappa}{\delta}] \label{assump:LipKd}
Given two {constants}~$\kappa<1$ and~$\delta>0$ we say that a t-embedding~$\cT$ satisfies assumption~{\LipKd{\kappa}{\delta}} in a region~$U\subset\C$ {covered by~$\cT$} if
\[
|\cO(z')-\cO(z)|\ \le\ \kappa\cdot|z'-z|\ \ \text{for all $z,z'\in U$\ such that\ $|z-z'|\ge\delta$.}
\]
\end{assumption}
Since we do not fold any face of~$\cT$ to get~$\cO$, assumption~{\LipKd{\kappa}{\delta}} clearly implies that all faces have diameter less than~$\delta$. We think of~$\delta$ as the `mesh size' of a t-embedding and sometimes explicitly include it into the notation by writing~$\cT^\delta$ and~$\cO^\delta$ instead of~$\cT$ and~$\cO$. Still, let us emphasize that the actual size of faces {can be} much \emph{smaller} than~$\delta$.

A good part of the a priori regularity theory developed in this paper holds just under assumption~{\LipKd{\kappa}{\delta}}. {Notably,} this assumption is enough to prove a uniform ellipticity estimate for random walks on {the} associated T-graphs {on scales greater than~$\delta$,} to prove an a priori H\"older-type estimate for t-holomorphic functions, and to describe their possible subsequential limits, see Section~\ref{sec:regularity} for details. However, to derive the a priori \emph{Lipschitz}-type estimate for harmonic functions on \mbox{T-graphs} and to deduce meaningful results for the dimer model on~$\cT$ we need slightly more: {see Assumption~\ExpFat{\delta}\ below. For shortness, we now formulate this additional assumption} only in the case when~$\cT^\delta$ are triangulations; {the general case is discussed in Section~\ref{sec:non_triangulation}.}

\smallskip

Given~$\rho>0$, let us say that a face of~$\cT$ is~\emph{$\rho$-fat} if it contains a disc of radius~$\rho$.

\begin{assumption}[\ExpFat{\delta}, triangulations] \label{assump:ExpFat-triang}
We say that a sequence $\cT^\delta$ of \mbox{t-embeddings} with triangular faces satisfies assumption~{\ExpFat{\delta}} {(or, more accurately, \ExpFatPrime{\delta}{\delta'}) in} a region~$U\subset\C$ {covered by~$\cT$} (or, more generally, {in} regions $U^\delta\subset\C$ {covered by~$\cT^\delta$ and} depending on~$\delta$) as~$\delta\to 0$ if {there exist auxiliary scales~$\delta'=\delta'(\delta)$ such that~$\delta'\to 0$ as~$\delta\to 0$ and the following holds:}
\begin{center}
 if one removes all ${\delta\exp(-\delta'\delta^{-1})}$-fat triangles from~$\cT^\delta$, then {each of the \\
 remaining vertex-connected components of~$\cT^\delta$ has diameter at most~$\delta'$.}
\end{center}
\end{assumption}
{As a warm-up example, let us assume that all edges of a t-embedding~$\cT$ with triangular faces are uniformly comparable to~$\ell$ and that all angles of its faces are uniformly bounded away from~$0$. In this case it is not hard to check that there exist~$\kappa<1$ and~$C>1$ such that the assumption~$\LipKd{\kappa}{\delta}$ holds with~$\delta=C\ell$. In its turn, the assumption~$\ExpFat{\delta}$ holds with~$\delta'=C'\delta$ provided that $C'>1$ is big enough as in this case \emph{all} triangles are~$e^{-C'}\delta$-fat. Thus, in this setup both~$\delta$ and~$\delta'$ are just multiples of the `true' mesh size of the tiling. A more interesting example arises when we know that~$\cT=\cT^\delta$ satisfies the assumption~\LipKd{\kappa}{\delta}\ and that all its faces except maybe isolated ones are, say, $\delta^{100}$-fat. In this case the assumption~\ExpFat{\delta}\ still holds with a huge margin since one can take \mbox{$\delta':=99\delta|\log\delta|$}. In full generality, one can replace~$\delta^{100}$ by any function decaying, as~$\delta\to 0$, slower than exponentially (in~$\delta^{-1}$) and admit not only isolated `exponentially non-fat' triangles in~$\cT^\delta$ but also arbitrary clusters formed by them, with the only requirement that the maximal \emph{Euclidean} diameter of these clusters tends to zero as~$\delta\to 0$.

Let us emphasize that -- contrary to~\LipKd{\kappa}{\delta} -- we regard the second assumption~\ExpFat{\delta}\ as `technical': loosely speaking (see Section~\ref{sub:Lipschitz} for details), we use it to exclude a hypothetical pathological scenario in which the gradients of uniformly bounded harmonic functions on T-graphs obtained from~$\cT^\delta$ grow exponentially (in~$\delta^{-1}$) fast as~$\delta\to 0$. Certainly, it would be mich nicer to rule out this pathological scenario using~\LipKd{\kappa}{\delta}\ only. However, it seems plausible to believe that the very mild assumption~\ExpFat{\delta}\ still holds in potential applications.}

\subsection{Regularity of harmonic functions on T-graphs} Though studying harmonic functions on T-graphs is not the primary motivation of our work, it is nevertheless one of its important ingredients. We now formulate our main a priori regularity result in this direction. For an open set~$V\subset\C$, let~$W^{1,\infty}(V)$ be the Sobolev space of functions whose derivatives are bounded on compact subsets of~$V$.

{Recall that the origami map~$\cO^\delta$ associated with a t-embedding~$\cT^\delta$ is defined up to translations and rotations. Therefore, when considering a sequence of, e.g., T-graphs~$\cT^\delta+(\alpha^\delta)^2\cO^\delta$ associated with given~$\cT^\delta$ we can assume that~$\alpha^\delta=1$ for all~$\delta$ without loss of generality.} In a special case of T-graphs obtained from skewed triangular lattices~$\cT^\delta$, {the following theorem yields \cite[Lemma~3.6]{kenyon-honeycomb};} recall {however} that our aim is to develop the regularity theory for \emph{general} t-embeddings~$\cT^\delta$.

\begin{theorem} \label{thm:Lip-intro}
Let~$\cT^\delta$, $\delta\to 0$ be a sequence of t-embeddings satisfying both assumption~{\LipKd{\kappa}{\delta}} (with a common constant~$\kappa<1$) and assumption~{\ExpFat{\delta}}. Let~$H^\delta$ be a sequence of {(real-valued)} harmonic functions defined on T-graphs ${\cT^\delta+\cO^\delta}$ associated to~$\cT^\delta$. If the functions~$H^\delta$ are uniformly bounded in a region $V\subset\C$, then {these functions are also uniformly Lipschitz on each compact subset of~$V$. Moreover,} the family~$\{H^\delta\}$ is pre-compact in the space~$W^{1,\infty}(V)$.
\end{theorem}
\begin{proof} See Section~\ref{sub:Lipschitz}, notably Corollary~\ref{cor:H-sub-limits}.
\end{proof}

{Given Theorem~\ref{thm:Lip-intro},} one can ask about properties of subsequential limits of bounded harmonic functions on T-graphs ${\cT^\delta+\cO^\delta}$. To this end, let us assume that
the t-embeddings~$\cT^\delta$ cover a common region~$U\subset\C$. As~$\cO^\delta$ are $1$-Lipschitz functions on~$U$, one can always find a subsequence such that
\begin{equation}
\label{eq:O-conv-theta}
\cO^\delta(z)\ \to\ \vartheta(z)\quad \text{uniformly on compact subsets,}
\end{equation}
for a Lipschitz function~$\vartheta:U\to\C$. As above, assume that t-embeddings~$\cT^\delta$ satisfy both assumptions~{\LipKd{\kappa}{\delta}} and~{\ExpFat{\delta}}\ {in~$U$.} In Section~\ref{sub:Lipschitz} we also show that the gradients~$2\partial h:=\partial_x h-i\partial_y h$ of all subsequential limits~$h{:(\mathrm{id}+\vartheta)(U)\to\R}$ from Theorem~\ref{thm:Lip-intro} admit the following representation:
\begin{align*}
2\partial h=f\circ(\mathrm{id}+\vartheta)\quad &\text{{with a H\"older-continuous function}}\ f:U\to \C\\[-2pt]
&\text{such that the form}\ f(z)dz+\overline{f(z)}d\overline{\vartheta(z)}\ \text{is closed}.
\end{align*}
In a special situation~$\vartheta(z)\equiv 0$, which we call the \emph{`small origami' case} below, one sees that the functions~$h$ are just harmonic in~${U}$. In general, $h$ satisfies a second order PDE whose coefficients can be recovered from~$\vartheta$.
Though we do not go into such an analysis here, let us nevertheless mention that there also exists a very particular generalization of the case~$\vartheta(z)\equiv 0$. Namely, if we assume that~$(z,\vartheta(z))$ is a space-like {\emph{maximal}} surface in {the Minkowski space~$\mathbb \R^{2,2}$,} then all {subsequential limits}~$h$ are harmonic in the conformal metric of this surface, see~\cite{CLR2} for details.

\subsection{Convergence framework for the dimer model on t-embeddings} We begin with recalling the definition of the Thurston \emph{height function}~\cite{thurston-height} for the dimer model on a bipartite graph~$\G$. Given a perfect matching~$\mathrm{P}$ of vertices of~$\G$, let~$\mathrm{P}^*$ be a flow on edges of the dual graph~$\G^*$ constructed as follows: one assigns the value~$1$ to edges~$bw^*$ crossing those edges~$bw$ of~$\G$ which are used in~$\mathrm{P}$ (with the plus sign if~$b$ is on the right), and $0$ to all other edges. If~$\mathrm{P}_0$ is an (arbitrarily chosen) reference perfect matching, then the primitive of the flow~${\mathrm{P}^*\!-\mathrm{P}_0^*}$ is well defined (up to an additive constant) and is called the \emph{height function} of the perfect matching~$\mathrm{P}$. Given a t-embedding~${\cT}$ carrying the dimer model, we denote by~${h_\cT}$ the random height function obtained from a random perfect matching of faces of~${\cT}$; {note that~$h_\cT$ is defined on vertices of~$\cT$. Further,} let~${\hbar_\cT}:= {h_\cT}-\E({h_\cT})$ be the \emph{fluctuations} of~${h_\cT}$. {It is not hard to see that, even though the definition of the function~${h_\cT}$ involves a choice of the reference matching $\mathrm{P}_0$, the fluctuations~$\hbar_\cT$ are independent of this choice. For} a collection of vertices~{$v_1,\ldots,v_n$ of~$\cT$, denote by}
\begin{equation}
\label{eq:def-Hn-discrete}
{H_{\cT,n}}(v_1,\ldots,v_n)\ :=\ \mathbb E({\hbar_\cT}(v_1)\ldots{\hbar_\cT}(v_n))
\end{equation}
the \emph{correlation functions} of the height fluctuations at {these vertices.}

{Let us now assume that we are given a sequence of finite t-embeddings~$\cT_m$ and that the corresponding discrete domains~$\Omega_{\cT_m}$, defined as the unions of faces of~$\cT_m$, approximate} a bounded simply connected domain~$\Omega\subset\C$ as~${m\to\infty}$ (say, in the Hausdorff sense for simplicity though in fact one can also work with weaker notions of convergence {of~$\Omega_{\cT_m}$} to $\Omega$). For~$v_1,\ldots,v_{{2k}}\in\Omega$, let
\begin{align}
G_{\Omega,{2k}}&(v_1,\ldots,v_{{2k}}) \notag \\
&\ :=\ \sum\nolimits_{\substack{\operatorname{pairings} {\varpi} \operatorname{of} 1,\ldots,{2k}}}G_\Omega(v_{{\varpi}(1)},v_{{\varpi}(2)})\ldots G_\Omega(v_{{\varpi}({2k}-1)},v_{{\varpi}({2k})})
\label{eq:Green-correlations}
\end{align}
be the correlation functions of the Gaussian Free Field {(GFF)} in~$\Omega$ with Dirichlet boundary conditions {(e.g., see~\cite{sheffield-GFF} for background),} where the normalization of the Green function is chosen so that
$G_\Omega(z,z')=-\tfrac{1}{2\pi}\log|z'-z|+O(1)$ as $z'\to z$;
we also set~$G_{\Omega,{2k}+1}:=0$. The following theorem provides a general framework to study the limit of the dimer model on t-embeddings.

\begin{theorem}\label{thm:main-GFF} Let {the} t-embeddings~${\cT_m}$ approximate a bounded simply connected domain~$\Omega\subset\C$ {as~$m\to\infty$. Assume that for each compact subset~$K\subset\Omega$ there exist a constant $\kappa=\kappa(K)<1$ and scales~$\delta_m=\delta_m(K)\to 0$, $\delta'_m=\delta'_m(K)\to 0$ as~$m\to\infty$ such that~$\cT_m$ satisfies the assumptions~\LipKd{\kappa}{\delta_m}\ and \ExpFatPrime{\delta_m}{\delta'_m}\ on~$K$ for all sufficiently large~$m$. Assume also that}
{\renewcommand{\theenumi}{\Roman{enumi}}
\begin{enumerate}
\item we are in the `small origami' case: ${\cO_m}(z)\to \vartheta(z)\equiv 0$ as ${m\to\infty}$;
\item {the coupling functions $K^{-1}_{\cT_m}$ are uniformly bounded on compact sets: for each~$\rho>0$ there exists $C(\rho)>0$ such that~$|K^{-1}_{\cT_m}(w,b)|\le C(\rho)$ provided that $m$ is big enough (depending only on~$\rho$) and the faces~$w,b$ of~$\cT_m$ stay \mbox{$\rho$-away}} from each other and from the boundary of $\Omega$;
\item the correlations~\eqref{eq:def-Hn-discrete} are uniformly small near the boundary of~$\Omega$: {for each~$\varepsilon>0$ and for each $\rho>0$ there exists~$d(\varepsilon,\rho)>0$ such that
\[
\qquad |\,H_{\cT_m,n}(v_1^{(m)},\ldots,v_n^{(m)})|\ \le\ \varepsilon\ \ \text{if}\ \ \mathrm{dist}(v_n^{{(m)}},\partial\Omega)\le d(\varepsilon,\rho),
\]
$m$ is big enough (depending only on~$\varepsilon$, $\rho$ and~$n$) and provided that the other vertices~$v_1^{{(m)}},\ldots,v_{n-1}^{{(m)}}$ of~$\cT_m$ stay $\rho$-away} from each other and from $\partial\Omega$.
\end{enumerate}}
Then, the height function correlations~\eqref{eq:def-Hn-discrete} converge to those of the GFF in~$\Omega$: for all~$n\ge 2$ and all collections of pairwise distinct points~$v_1,\ldots,v_n\in\Omega$, we have
\[
H_{{\cT_m,n}}(v_1^{{(m)}},\ldots,v_n^{{(m)}})\ \to\ \pi^{-n/2}G_{\Omega,n}(v_1,\ldots,v_n)\quad \text{if}\ \ v_k^{{(m)}}\to v_k\ \ \text{as}\ \ {m\to\infty}.
\]
Moreover, this convergence is uniform provided that~$v_1,\ldots,v_n$ remain at a definite distance from each other and from the boundary of~$\Omega$.
\end{theorem}
Before discussing assumptions (I)--(III) in more detail, let us emphasize {that in Theorem~\ref{thm:main-GFF} we neither assume nor prove the existence of scaling limits of the coupling functions~$K^{-1}_{{\cT_m}}$ themselves.} Such limits, when they do exist, are known to be highly sensitive to the microscopic details of the boundary, see Section~\ref{sub:discussion} for a discussion.
In particular, {in many setups} one should \emph{not} expect the convergence of the full sequence of the coupling functions though
subsequential limits of them still exist under assumption (II) due to compactness arguments.

We emphasize that {in this paper} we do \emph{not} discuss how {one can check the assumptions (I)--(III) in any concrete setup,} this is why we call our result a framework or a meta-theorem.
Still, it is worth mentioning that {almost all} known examples of applications of discrete complex analysis techniques to the bipartite dimer model {fit} the framework of Theorem~\ref{thm:main-GFF}; {see also Section~\ref{sec:appendix2}. An important exception is the work~\cite{kenyon-honeycomb} of Kenyon on the convergence of height correlations to the GFF in a `non-flat' metric; see also a recent development~\cite{laslier-21} that fixes several details of this approach, which combines a `local' study of the dimer coupling function by means of discrete complex analysis on skewed triangular lattices with other ideas. Though we do not know whether it is possible to construct an appropriate global t-embedding and to apply Theorem~\ref{thm:main-GFF} in the setup of \cite{kenyon-honeycomb,laslier-21}, we believe that our paper, in particular, provides a natural development of the ideas originated in~\cite{kenyon-honeycomb}.

Let us also mention that we use} essentially the same approach to the convergence of height fluctuations in our follow-up paper~\cite{CLR2}. To conclude, we briefly discuss each of {the assumptions (I)--(III),} in particular {in order} to make precise the {links} between {the setup of Theorem~\ref{thm:main-GFF} and that of~\cite{CLR2}.}
{\renewcommand{\theenumi}{\Roman{enumi}}
\begin{enumerate}
\item This assumption cannot be dropped completely. Still, there exists a striking case when the proof of Theorem~\ref{thm:main-GFF} goes through just by the cost of more involved computations. Namely, if~$(z,\vartheta(z))$ is a space-like {\emph{maximal}} surface in {the Minkowski space~$\R^{2,2}$,} then the correlations~\eqref{eq:def-Hn-discrete} converge to those of the GFF in the conformal metric of {this} surface. {For simplicity, we do not consider this more general setup here and discuss it in~\cite{CLR2}.}
\smallskip
\item This assumption seems natural from the discrete complex analysis perspective: if the coupling functions~$K^{-1}_{{\cT_m}}$ are not bounded on compacts, they typically contain local exponential growing factors as mentioned in Section~\ref{sub:gen-context}. In such a situation, one should not expect that the `discrete conformal structure' {provided by the t-embedding~$\cT_m$} captures the behaviour of~$K^{-1}_{{\cT_m}}$ correctly. In {previously} known examples, this assumption is typically verified along with finding the \emph{scaling limit} of~$K^{-1}_{{\cT_m}}$ {as~$m\to\infty$}, which is exactly the route that we want to \emph{avoid} by formulating Theorem~\ref{thm:main-GFF}. {In particular,} in~\cite{CLR2} we introduce a special class of t-embeddings {of finite graphs,} so-called {`perfect' ones,} for which we are able to derive the required uniform boundedness of {the coupling functions~$K^{-1}_{\cT_m}$} from general estimates, not {identifying} their possible scaling limits.
\smallskip
\item This assumption is quite natural from the dimer model perspective: it {simply} says that
fluctuations vanish near the boundary {of~$\Omega$.} However, it should be said that even in the well-known cases {this fact} is typically derived \emph{a posteriori} from the identification of the scaling limit of~$K^{-1}_{{\cT_m}}$. Nevertheless, in some situations one could hope to verify it by probabilistic tools. Let us also mention that {for `perfect' t-embeddings} introduced in~\cite{CLR2} we are in fact able to prove the uniform boundedness of the functions ${K^{-1}_{\cT_m}(w,b)}$ not only on compact subsets of~$\Omega$ but also in a situation when \emph{one} of {$w$ and $b$ is allowed to approach} the boundary of~${\cT_m}$. Though this estimate near~$\partial\Omega$ is not strong enough to control the boundary values of the limits of~$H_{{\cT_m,n}}$, it nevertheless implies the convergence of the \emph{gradients} of these correlation functions to those of the GFF; see~\cite{CLR2} for details.
\end{enumerate}}

The paper is organized as follows. We overview the setup of t-embeddings~$\cT$ in Section~\ref{sec:setup}. The notion of t-holomorphicity on~$\cT$, in the special case when~$\cT$ is a triangulation, is introduced in Section~\ref{sec:holomorphicity}. In Section~\ref{sec:Tgraph} we discuss the links between t-holomorphic functions on t-embeddings and harmonic {functions} on T-graphs; note that Section~\ref{sub:harmonicity-1} contains a new material {as} compared to, say, the paper~\cite{kenyon-honeycomb} due to Kenyon. Section~\ref{sec:non_triangulation} is devoted to generalizations of all these notions to the case of general t-embeddings (not triangulations). Section~\ref{sec:regularity} is at the heart of our paper, we develop the a priori regularity theory for t-holomorphic and harmonic functions there. Two particularly important {ingredients} are the uniform ellipticity estimate for random walks on T-graphs obtained in Section~\ref{sub:ellipticity} under {the} assumption~{\LipKd{\kappa}{\delta}} and the a priori Lipschitzness of harmonic functions discussed in Section~\ref{sub:Lipschitz} under {the} additional assumption~{\ExpFat{\delta}}. We prove Theorem~\ref{thm:main-GFF} in Section~\ref{sec:convergence}. {Finally, in Section~\ref{sec:appendix2} we discuss the} links between \mbox{t-holomorphic} functions on t-embeddings and more standard discretizations of complex analysis.

\addtocontents{toc}{\protect\setcounter{tocdepth}{1}}
\subsection*{Acknowledgements} D.C. is grateful to {Mikhail Basok,} Alexander Logunov, Eugenia Malinnikova {and R\'emy Mahfouf} for helpful discussions. M.R. would like to thank Alexei Borodin for useful discussions. We also would like to thank Nathana\"el Berestycki, Richard Kenyon and Stanislav Smirnov for their interest {and the referees for providing a useful feedback on the first version of this paper.}

D.C. is the holder of the ENS--MHI chair funded by MHI. The research of D.C. and B.L. was partially supported by the ANR-18-CE40-0033 project DIMERS. The research of M.R. is supported by the Swiss NSF grants P400P2-194429 and P2GEP2-184555 and also partially supported by the NSF Grant DMS-1664619.
\addtocontents{toc}{\protect\setcounter{tocdepth}{2}}

\section{The setup of t-embeddings}\label{sec:setup}

\subsection{Definitions}
In this section we introduce 
 t-embeddings and give several related definitions.

\begin{definition}
A t-embedding in the whole plane is an embedded {locally finite} planar {graph} with the following properties:
\begin{itemize}
\item {Properness:} The edges are non-degenerate straight segments, the faces are convex, do not overlap and cover the whole plane.
\item Bipartite dual: The dual graph is bipartite, we call the bipartite classes black and white, and denote them $B$ and $W$, respectively. {(In other words, we assume that the \emph{faces} of the corresponding tiling of the plane by convex polygons are colored black and white in a chessboard fashion.)}
\item Angle condition: For every vertex $v$ one has
\[
 \sum_{{b\in B:\,b\sim v}} \theta(b, v)\ = \sum_{{w\in W:\,w\sim v}} \theta(w, v)\ =\ \pi,
\]
where $\theta({u}, v)$ {is} the angle of a face ${u}$ at {a} neighbouring vertex $v$, see Fig.~\ref{Temb}.
\end{itemize}
\end{definition}
Given an {infinite} t-embedding, let $\G^*$ be the associated planar graph seen as an abstract combinatorial object {(i.e., as a \emph{planar map}: a planar graph embedded into the plane and considered up to homotopies in the space of proper embeddings)} and let $\G$ be its planar bipartite dual, also seen abstractly.

The above definition can be extended to finite {bipartite planar graphs~$\G$} {with the topology of the sphere. To this end, we remove one marked vertex $v_\mathrm{out}$ from {the dual graph $\G^*$ which is to be embedded,} and replace it by a cycle of length $\deg v_\mathrm{out}$ so that $\deg v_\mathrm{out}$ edges adjacent to $v_\mathrm{out}$ become adjacent to corresponding vertices of the cycle. {Following~\cite{KLRR},} we call this procedure an \emph{augmentation} at $v_\mathrm{out}$.

\begin{definition}
{A finite t-embedding of a planar graph with the topology of the sphere and a marked vertex $v_\mathrm{out}$ is an embedding of its augmentation at $v_\mathrm{out}$} with the following properties:
\begin{itemize}
\item {Properness:} The edges are non-degenerate straight segments, the faces are convex and do not overlap, the outer face corresponds to the cycle replacing $v_\mathrm{out}$ in the augmented graph.
\item Bipartite dual: The dual graph of the augmented map becomes bipartite once the outer face is removed, we call the bipartite classes black and white, and denote them $B$ and $W$.
\item Angle condition: For every interior vertex $v$ one has
\[
 \sum_{{b\in B:\,b\sim v}} \theta(b, v)\ = \sum_{{w\in W:\,w\sim v}} \theta(w, v)\ =\ \pi,
\]
where we call $v$ interior if it is not adjacent to the outer face, see Fig.~\ref{finTembTriangl}.
\end{itemize}
We call the union of the closed faces {(except the outer one)} of a finite t-embedding  the \emph{discrete domain} associated {with} this t-embedding.
\end{definition}

{In what follows,} we \emph{exclude} the outer face from $\G$ and the \emph{boundary edges} (i.e., those adjacent to the outer face) from $\G^*$; see Fig.~\ref{finTembTriangl}. Recall that ${V(\G)=B\cup W}$, where $B$ and $W$ denote the sets of black and white faces of $\G^*$, respectively. Below we denote typical faces of $\G^*$ either $b$ or $w$ depending on their color and also use the notation $u^\tb$, $u^\tw$ for the same purpose. The vertices of $\G^*$ are typically denoted as $v,v'$ etc.}
We say that a face of the graph~$\G^*$ of a finite t-embedding is a \emph{boundary face}
if it is adjacent to at least one boundary edge. Other faces are called \emph{interior}. Let $\partial B$ and $\partial W$ be the sets of boundary black and boundary white faces, respectively.

Given an oriented edge $(bw)$ of $\G$, denote by $(bw)^*$ (or $bw^*$ for brevity) the oriented edge of~$\G^*$ which has the first face (here $b$) to its right. Denote by $wb^*$ the same edge of $\G^*$ oriented in the opposite direction. Let~$\cT$ denote the map from $\G^*$ to $\mathbb{C}$ giving the position of any vertex in the embedding.
Given an oriented edge~${e^* = (vv')}$ of $\G^*$, let ${d\cT( e^* ) := \cT( v') - \cT(v)}$, see Fig.~\ref{Temb}.
For a given face $b$ or $w$ of $\G^*$, we write $\cT(b)$ or $\cT(w)$ to denote the corresponding polygon in the embedding.



Let us now {briefly} describe how to construct a \emph{realisation}
of~$\G$ given a t-embedding ${\cT = \cT(\G^*)}$. The resulting realisations with an embedded dual have been introduced and studied in~\cite{affolter,KLRR} {in the so-called `circle patterns' context, so we refer the reader to these papers for more details.}

\begin{lemma}
\label{lem:def-C}
{The following definition of a mapping $\cC{:V(\G)\to\C}$ constructed from a t-embedding~$\cT(\G^*)$ is consistent:} fix an arbitrary white vertex $w_0\in {V(\G)}$ and choose $\cC(w_0)$ arbitrarily, then define $\cC$ at neighbours of~$w_0$ and iteratively everywhere {on $V(\G)$} by saying that {points $\cC(b)$ and $\cC(w)$ are symmetric with respect to the line $\cT(bw^*)$ for each pair of neighboring $b$ and $w$.}
\end{lemma}

\begin{proof}
It is enough to check {the consistency} 
around a single vertex of~$\G^*$. Let $w_1, b_1, \ldots,$ $w_k, b_k$ be faces of~$\G^*$ around~$v$, {labeled in the counterclockwise order,} 
and assume without loss of generality and for ease of notation that $\cT(v) = 0$. Let
\[d_1 = {d}\cT(w_1b_1^*)/|{d}\cT( w_1b_1^*) |,\quad d_2 = {d}\cT( b_1w_2^*)/|{d}\cT( b_1w_2^*)|, \quad\ldots\] be the directions of the edges of $\cT(\G^*)$ around $\cT(v)$, pointing away from~$\cT(v)$.

It is easy to {see} that the reflection symmetry condition
{gives} \mbox{$\cC( b_1) = d_1^2\cdot \overline{\cC( w_1)}$} and therefore $\cC( w_2 ) = (d_2 \overline{d_1})^2\cdot \cC( w_1)$. Note that $\arg(d_2 \overline{d_1})$ is the angle of the face~$\cT(b_1)$ at~$\cT(v)$, 
so the {angle condition is equivalent to the consistency of the definition of $\cC$ around $v$.}
\end{proof}

The construction {described in {Lemma~\ref{lem:def-C}} produces a two-dimensional family of realisations of $\G$ parametrized by the position of $\cC(w_0)$. {In general,} it is not clear whether $\cC$ is a proper embedding {of $\G$.} However, for each face $v$ of $\G$, all points $\cC(u)$, where $v\sim u\in V(\G)$, lie on a single circle and  each point $\cC(u)$ is an intersection of ${\deg(u)}$ such circles. {This justifies the name \emph{circle pattern realisations} for such embeddings of bipartite planar graphs,} see~\cite{affolter,KLRR}.

Informally speaking, the above construction of $\cC$ can be equivalently described as follows: fold the plane along all the edges of $\cT$ (where the angle condition {guarantees} that this operation makes sense), then pierce the folded plane at an arbitrary point. Finally, unfold the plane:  the realisation $\cC$ is given by the positions of all the punctures (provided that all points $\cC(u)$ lie inside corresponding faces $\cT(u)$ of the t-embedding, which is certainly not true in general).

\subsection{The origami map} The goal of this section is to introduce a formal definition of the folding procedure described above, which we call the origami map and which plays a crucial role in our analysis.  

\begin{definition}\label{def:eta}
A function $\eta : {V(\G)=B\cup W} \to \bbT$ is said to be an origami square root function if it satisfies {the identity}
\begin{equation}
\label{eq:def-eta}
\overline{\eta}_b\overline{\eta}_w = \frac{d \cT ( bw^*)}{| d \cT( bw^*) |}
\end{equation}
for all pairs $(b,w)$ of white and black neighbouring faces of $\G^*$.
\end{definition}

\begin{remark}\label{rem:eta-def}
Since the sum of external angles of each convex polygon equals~$2\pi$, equation~\eqref{eq:def-eta} gives a consistent definition of~$\eta$ around each face of~$\cT$. However, it can be (slightly) inconsistent around its vertices: due to the angle condition, {definition~\eqref{eq:def-eta} implies that the total increment of~$\eta_w$ (and similarly for~$\eta_b$) around a vertex~$v$ of a t-embedding equals~$(1\!+\!\frac12\deg v)\cdot\pi$. Therefore,} the function $\eta^2$ is always well-defined but the function $\eta$ itself has to branch over every {vertex ${v}$ of~$\G^*$ (i.e., a face of~$\G$) such that $\deg{v}\in 4\mathbb{Z}$.} (In other words, $\eta$ has to be defined on an appropriate double cover of $V(\G)$, with the values on the two sheets being opposite of each other.) By an abuse of notation we will consider $\eta$ being defined up to {the} sign.
We also define the values $\phi_u\in (-\pi/2, \pi/2]$ as follows:
\[\phi_u := \arg \eta_u \mod \pi,\qquad u\in B\cup W.\]
\end{remark}

{It is clear that two origami square root functions $\eta$ and $\eta'$ differ only by a global factor: more precisely, there exists ${\alpha}\in\bbT$ such that $\eta'_w={\alpha}\cdot \eta_w$ for all $w\in W$ and $\eta'_b=\overline{{\alpha}}\cdot \eta_b$ for all $b\in B$. In general, there is no canonical way to choose the global prefactor ${\alpha}$.}
Let us now 
{comment on how the angles $\phi$ are related to the geometry of a t-embedding.}

\begin{lemma}\label{lem:increment_phi}
{Let~$v$ be an inner vertex of~$\G^*$ and $b_1,b_2\in B$ and~$w\in W$ be three consecutive faces adjacent to~$v$.} If $b_1, w, b_2$ are in the counterclockwise order around $v$, then, for any origami square root function and associated $\phi$,
\[
 \phi_{b_2} - \phi_{b_1} =  -{\theta(w, v)} \mod \pi,
\]
where $\theta(w, v)$ is the angle of the white face $w$ at the vertex $v$, {computed in the} positive direction.

Similarly, 
if $w_1$, $b$, $w_2$  are in the counterclockwise order around their common vertex $v$, then
\[
\phi_{w_2} - \phi_{w_1} =  - {\theta(b, v)} \mod \pi.
\]
\end{lemma}

\begin{proof}
Let $v_{j}\neq v$ be {the endpoints of} the edges $\cT( wb^*_{j})$ for ${j}\in\{1,2\}$.
{Then,}
\[
{\eta_{b_2} \overline{\eta}{}_{b_1}}  = \frac{{\overline{d \cT ( b_2 w^*)}}}{| d \cT( b_2 w^*) |}\cdot \frac{{d\cT( b_1 w^*)}}{| d \cT( b_1 w^*) |}
  = \frac{{\overline{v - v_2}}}{| v - v_2|}\cdot  \frac{{v_1 - v}}{| v_1 - v|}
  = - e^{-i \theta(w, v)} .
\]
The computation for the second case is identical. 
\end{proof}

We now formally define the folding of the plane along the edges of $\cT$ {using the (square of the) function $\eta$ introduced above.}


\begin{definition}\label{def:O}
The origami differential form associated to $\eta$ is defined as
\[
d \cO (z) := \begin{cases}
 \eta_w^2 \,dz & \text{if $z$ {belongs to a} white face $\cT(w)$,}\\
 \overline{\eta}_b^2 \,d \bar{z} & \text{if $z$ {belongs to a} black face $\cT(b)$.}
\end{cases}
\]
\end{definition}
{Let us emphasize that we view $d\cO (z)$ as a piecewise constant differential form defined in the whole \emph{complex plane} (or inside the discrete domain associated to a finite t-embedding). However, it is worth noting that the above definition also allows one to view $d\cO$ as a well-defined $1$-form on \emph{edges of $\cT$} by setting}
\begin{equation}
\label{eq:dO-on-T}
d \cO ( bw^*)\ :=\ \eta_w^2 \,d\cT( bw^*)\ {=\ \overline{\eta}_b\eta_w\,|d\cT( bw^*)|\ } =\ \overline{\eta}_b^2 \,\overline{d \cT}( bw^*).
\end{equation}

\begin{lemma}
\label{lem:dO-closed}
The origami differential form $d\cO$ is a closed form (inside the {associated discrete domain} in the finite case).
We denote its primitive by $\cO$, which we call the origami map.
\end{lemma}

\begin{proof} Let $\gamma$ be a closed contour running in the domain of a t-embedding. If $\gamma$ lies inside a single face of $\cT$, then $\oint_\gamma d\cO = 0$ since $d\cO$ is proportional either to $dz$ or to $d\overline{z}$. For general~$\gamma$, since $|d\cO(z)|\le|dz|$, one can always write $\oint_\gamma d\cO$ as a sum of integrals over smaller loops~$\gamma_u$, each of which belongs to a \emph{closed} face $\cT(u)$. As pointed out in \eqref{eq:dO-on-T}, on an edge $bw^*$ of the t-embedding, the two definitions of $d\cO$ (coming from the right face $b$ and the left face $w$) agree. Hence, all contour integrals over such loops $\gamma_u$ make sense and vanish, thus $\oint_\gamma dO$ also vanishes.
\end{proof}

Note that $\eta^2_{w}$ {is} the local rotation angle of the origami map $\cO$ on the white face $\cT(w)$. Recall that {the} origami differential form $d\cO$ {is} defined up to a global prefactor ${\alpha}^2\in \bbT$ only, {which means that the origami map~$\cO$ itself is defined up to rotations and translations.} {If $\cO|_{w_0}=\mathrm{Id}$ on some white face $w_0\in W$} (which one can always assume by choosing ${\alpha}$ {and the integration constant} properly), then} it is easy to check that $\cO$ maps $z$ to its position after the folding procedure {(started from the face $w_0$ so as it is kept fixed), which was described in} the construction of a circle pattern realisation $\cC$. In particular, if $\cC(u)\in \cT(u)$ for all $u\in V(\G)$ (recall that this is not true in general), then~$\{\cC(u),~u\in V(\G)\}=\cO^{-1}(\cC(w_0))$.

With a slight abuse of notation ({similar to that} in the definition of the origami differential form), below we also allow ourselves to see $\cO$ as a map from $\G^*$ to $\C$.


\subsection{Dimers and t-embeddings} In this section we describe how to define Kasteleyn weights on a bipartite graph $\G$ in a natural geometric way given a t-embedding of {its dual graph} $\G^*$. Let \mbox{$\chi( e):=| d\cT (e^*) |$} be positive weights of edges $e$ of~$\G$. Recall that a \emph{Kasteleyn matrix} $K$ is a weighted, complex-signed adjacency matrix whose rows index the black vertices and columns index the white vertices, and the signs ($\tau_{bw}\in\mathbb{C}, |\tau_{bw}|=1$) are chosen to satisfy the following condition: 
around a face of $\G$ of degree $2k$ the alternating product of signs over the edges of this face is~$(-1)^{k+1}$. Signs satisfying this condition around each face are called Kasteleyn signs.

\begin{proposition}\label{Kast}
For {$b, w\in\G$, let} $K(b, w) := d \cT(bw^*)$ if $b$ and $w$ are neighbours and ${K(b,w):=}0$ otherwise. Then, $K$ is a Kasteleyn matrix for the weights $\chi$.
\end{proposition}
\begin{proof}
 Fix a face $v$ of $\G$ and let $b_1, w_1, \ldots, b_k, w_k$ be its {neighboring} vertices {listed counterclockwise.} Let $v_1, \ldots, v_{2k}$ be its neighbouring faces {listed counterclockwise} so that $v_1$ is between $b_1$ and $w_1$. {It is easy to see that}
 \[
 \prod_{i=1}^{k} \frac{K(b_i, w_{i})}{K( b_{i+1}, w_{i}) }\  =\ \prod_{i=1}^{k} \frac{d \cT( v v_{2i -1 })}{d \cT( v_{2i} v ) }\ =\
  {X_v}\cdot \prod_{i=1}^{k} (- e^{- i \theta( w_i,v)})\ =\
  {X_v}\cdot(-1)^{k+1}
 \]
 where ${X_v:= \prod_{i=1}^{k}} \frac{\chi(b_iw_{i})}{\chi( b_{i+1}w_{i}) }$ is a positive constant; in the last equality we use that the white angles {adjacent to $v$} sum {up} to $\pi$. This is exactly the sign condition in a Kasteleyn matrix.
\end{proof}

{Given an abstract planar weighted bipartite graph $(\G,x)$, one can wonder about the existence of a t-embedding of $\G^*$ into the plane such that {the} given edges weights $x(e)$ are \emph{gauge equivalent} to the geometrical weights $\chi(e)=|d\cT(e^*)|$ introduced above. (The gauge equivalence means that $\chi(bw)=g(b)x(bw)g(w)$ for all $b\in B$, $w\in W$ and some function $g:V(\G)\to\R_+$; such a transform preserves the law of the dimer model on $\G$.) In this case, we say that $(\G,x)$ admits a t-embedding of the (augmented) dual graph $\G^*$.}
We are now in the position to state one of the main results of~\cite{KLRR}, we refer the interested reader to this paper for more details.
\begin{theorem}[{\cite[Theorem~2, Theorem~8]{KLRR}}]\label{thm:KLRR-existence}
T-embeddings of the (augmented) dual graph $\G^*$ exist at least in the following cases:

\smallskip

\noindent (i) $(\G,x)$ is a nondegenerate bipartite finite weighted graph {admitting a dimer cover and} with outer face of degree~$4$. 

\smallskip

\noindent (ii) $(\G,x)$ is a doubly periodic weighted bipartite graph, equipped with an equivalence class of doubly periodic edge weights, which corresponds to a liquid phase. {In this situation we can also require that the t-embedding $\cT$ is {doubly periodic} and that the origami map $\cO$ is bounded. Moreover, such {doubly periodic} t-embeddings of $\G^*$, considered up to scalings, rotations, translations and {reflections,} are in bijection with the interior of the amoeba of the dimer spectral curve.}
\end{theorem}

\section{T-holomorphicity}\label{sec:holomorphicity}

In this section, we introduce the notion of \emph{t-holomorphic functions} defined on faces of a \mbox{t-embedding} and give some basic facts about such functions. Let us already remark that this theory {has a simpler and more invariant form} in the case of triangulations so we restrict ourselves to this case for now. The modifications required in the general case will be given in Section~\ref{sec:non_triangulation}.

Below, we work with a fixed t-embedding $\cT$ of a finite or infinite \emph{triangulation} and a fixed origami square root function $\eta$. In the finite case, we call a t-embedding $\cT(\G^*)$ a {triangulation} if all its \emph{interior} faces are triangles (equivalently, if all vertices of the corresponding dual bipartite graph $\G$ have degree three),
see Fig.~\ref{finTembTriangl} for an example. A t-holomorphic function $F$ will be defined on both black and white faces but $B$ and $W$ do \emph{not} play the same role. We denote by $F^\tb$ the restriction of a function $F$ to black faces and $F^\tw$ the restriction to white faces, 
the t-holomorphicity condition links the values of $F^\tb$ and $F^\tw$ to each other. We will use a subscript $\frb$ or $\frw$ to indicate whether we `primarily' consider a function on black or white faces, note that all four combinations $F_\frb^\tw, F_\frb^\tb, F_\frw^\tw, F_\frw^\tb$ are used below.

\subsection{Definition of t-holomorphic functions}
We begin with a preliminary lemma. Let $K$ be the Kasteleyn matrix defined in Proposition~\ref{Kast}.
Given a discrete path $\gamma = (e_1, \ldots, e_n)$ on~$\G^*$ and a function $F$ on {\emph{unoriented}} edges of $\G^*$, define ${\int_\gamma F\, d \cT := \sum F(e_i) \,d\cT (e_i)}$. 
With a slight abuse of notation, one can extend this definition to functions defined {\emph{either}} on black {or on} white faces {of~$\G^*$ by setting} $F^\tw( bw^*) := F^\tw(w)$ for~$w\in W$ and similarly for $b\in B$.

Given a face $u$ of $\G^*$, let $\partial u$ be its boundary, viewed as a path on $\cT(\G^*)$ {and oriented in the positive (i.e., counterclockwise) direction.}
\begin{lemma} \label{lem:KF=oint}
Let $F^\tw$ be a complex-valued function defined on (a subset of) W. 
Then, for each interior black face $b$ one has
\[
(K F^\tw) (b) =  - \oint_{\partial b} F^\tw\, d\cT. 
\]
 Similarly, for a function $F^\tb$ defined on (a subset of) $B$ and an interior white face $w$, one has
\[
(F^\tb K) (w) = \oint_{\partial w} F^\tb \,d \cT. 
\]
\end{lemma}
\begin{proof}
By the definition of $K$, we have
$
(K F^\tw) (b)= \sum_{w:w \sim b} d \cT (bw^*)\cdot F^\tw(w).
$
{According to our conventions,} this is the definition of the contour integral {$-\oint_{\partial b} F^\tw d \cT$}; see Fig.~\ref{Temb}. The proof for white faces is similar.
\end{proof}

Note that, for any function $F^\tw$ and any white face $w$, the equality $\oint_{\partial w} F^\tw \,d \cT=0$ holds since $\cT$ is well defined. Therefore, the condition $K F^\tw(b) = 0$ for all $b$ in a simply connected region~$U$ of the t-embedding is equivalent to $\oint_\gamma F^\tw\, d \cT = 0$ for all closed contours~$\gamma$ in this region. A similar statement holds for a function $F^\tb$ defined on the black faces of $U$.
{(In this section we think of~$\gamma$ as being a path composed of edges of~$\cT$. However, let us note that below we adopt a more general viewpoint and think of $\gamma$ as a general rectifiable curve in the complex plane, in which~$\cT$ stands for the complex coordinate; see Lemma~\ref{lem:FdT-in-C} for a formal statement.)}

For a region $U$ of the t-embedding $\cT$, {let~$B_U$ and~$W_U$ denote the sets of black and white faces, respectively, that are contained in~$U$. Further, given a subset~$\mathfrak{p}$ of faces of~$\cT$, thought of as `punctures', denote~$U_\mathfrak{p}:=U\smallsetminus\mathfrak{p}$.} The {forthcoming Definition~\ref{def:t-hol}} is one of the central concepts of this paper, see Remark~\ref{rem:motivation-thol} for the motivation. {We use the notation
\[
\Pr(F,\eta\R)\ :=\ \tfrac12(F+\eta^2\overline{F})
\]
for the orthogonal projection of a complex number~$F$ onto the line~$\eta\R$, where~$|\eta|=1$.}

\begin{definition}\label{def:t-hol}
Given a subregion $U$ of a t-embedding $\cT$ {with triangular faces} and an origami square root function~$\eta$, a function  $F_\frw: U \to \mathbb{C}$, is said to be \emph{t-white-holomorphic at $\wc\in {W_U}$} if $\wc$ is an inner face of $U$ and
\begin{equation}\label{eq:def-twhol}
\begin{cases}
F_\frw^{\tb}(\bpr) \in {\eta}_b \R,\\ 
\Pr( F_\frw^{\tw}( \wc) ,{\eta}_b \R)= F_\frw^{\tb} (\bpr) 
\end{cases} \quad \text{for~all}~\bpr\in B~\text{such~that}~\bpr\sim\wc.
\end{equation}
A function $F_\frw$ is t-white-holomorphic in a region $U$ or, more generally, in {a punctured region} $U_\mathfrak{p}$ if it is t-white-holomorphic at all inner white faces of the region.

Similarly, we say that $F_\frb$ is \emph{t-black-holomorphic} at an inner face ${\bc\in B_U}$ if
\begin{equation}\label{eq:def-tbhol}
\begin{cases}
F_\frb^{\tw}( \wpr) \in \eta_w \R,\\
\Pr( F_\frb^{\tb}( \bc) , \eta_w\R )= F_\frb^{\tw} (\wpr) 
\end{cases}
\quad \text{for~all}~\wpr\in W~\text{such~that}~\wpr\sim\bc.
\end{equation}
{If no confusion arises, we simply say that a function is} t-holomorphic if it is \emph{either} \mbox{t-white}-holomorphic or t-black-holomorphic; {note that these properties never hold together as $B$ and~$W$ play different roles in each of the definitions~\eqref{eq:def-twhol},~\eqref{eq:def-tbhol}.}
\end{definition}

\begin{remark} \label{rem:motivation-thol}
A typical example of a t-white-holomorphic function is {given by}
\[
F_w^\tb(\bpr)\ {:=}\ \overline{\eta}_w\cdot K^{-1}(w, \bpr),\ \ \text{{where $w$ is a fixed white face of~$\cT$.}}
\]
Indeed, {the first condition in~\eqref{eq:def-twhol} holds since the matrix~$(\eta_bK(b,w)\eta_w)_{b\in B,w\in W}$ is real-valued due to~\eqref{eq:def-eta} and hence so is its inverse $(\overline{\eta}_wK^{-1}(w,b)\overline{\eta}_b)_{w\in W,b\in B}$.}
In Lemma \ref{lem:kernel_to_holomorphy} given below we show that the existence of a value $F^\tw(\wc)$ such that the second condition in~\eqref{eq:def-twhol} holds true given the first is equivalent to the identity $\oint_{\partial \wc}F^\tb{d\cT}=0$. This is true for $F_w^\tb$ due to Lemma~\ref{lem:KF=oint} {unless $\wc=b$; in other words, the function~$F_w$ is t-white-holomorphic in the punctured region~$U_w$.} Similarly, a typical example of a t-black-holomorphic function is given by $F_b^\tw(\wpr)= \overline{\eta}_b\cdot K^{-1}(\wpr,b)$, for a fixed black face~$b$. In~particular, the notation $F_\frw$, $F_\frb$ is designed so that in the future these functions can be {easily} replaced by $F_w$ and $F_b$ for actual faces~$w$ and $b$.
\end{remark}

\begin{lemma}\label{lem:kernel_to_holomorphy}
Given a t-embedding $\cT$ of a triangulation and an origami square root function~$\eta$, let $F_\frw^{\tb}$ be a function on black faces of some region $U$ such that \mbox{$F_\frw^{\tb} (\bpr) \in {\eta}_\bpr \R$} for all~{$\bpr\in B_U$. The function $F_\frw^{\tb}$ can be extended to a} t-white-holomorphic function $F_\frw$ in $U_\mathfrak{p}$, $\mathfrak{p}\subset {W_U}$, if and only if $\oint_{\partial \wc} F_\frw^\tb \,d  \cT = 0$ for all inner faces $\wc\in {W_U\smallsetminus \mathfrak{p}}$.

Similarly, a function $F_\frb^\tw$ defined on white faces of~$U$ such that $F_\frb^\tw(\wpr) \in \eta_\wpr \R$ {for all~$\wpr\in W_U$} admits an extension to a t-black-holomorphic function {in $U_\mathfrak{p}$, $\mathfrak{p}\subset {B_U}$,} if and only if $\oint_{\partial \bc} F_\frb^\tw \,d  \cT = 0$ for all inner faces $\bc\in {B_U\smallsetminus\mathfrak{p}}$.
\end{lemma}
\begin{proof}
{Consider} an inner white face $\wc$ of $U_\mathfrak{p}$ and let $\bpr_1$, $\bpr_2$, $\bpr_3$ be its adjacent black faces. The function $F^\tb_\frw$ can be extended to $\wc$ {as} a t-white-holomorphic function if and only if the three lines perpendicular to {$\eta_{\bpr_{k}}\R$, $k=1,2,3$,} and passing through the points $F^\tb_\frw(\bpr_{k})$, {respectively,} intersect. This corresponds to the {following equations} with unknown ${F_\frw^\tw}(\wc)$:
\[
  {F_\frw^\tw}(\wc) +  {\eta}_{\bpr_{k}}^2 \overline{{F_\frw^\tw}(\wc)} = 2\cdot F^\tb_\frw(\bpr_{k}) \quad \text{for~all}\ {k}\in\{1,2,3\}.
\]
These equations can be viewed as a system of three \emph{real} equations on two real unknowns. Since $\eta_b^2d\cT(bw^*)=d\overline{\cO(bw^*)}$, multiplying each equation by $d\cT({(b_k\wc)^*})$ and adding them together one easily gets a necessary solvability condition:
\begin{align*}
2\oint_{\partial\wc}\!F^\tb_\frw(\bpr_{k})d\cT\ &=\ \oint_{\partial\wc}\!\big({F_\frw^\tw}(\wc) +  {\eta}_{\bpr_{k}}^2 \overline{{ F_\frw^\tw}(\wc)}\big)d\cT\ \\
&=\ F_\frw^\tw(\wc)\oint_{\partial\wc}\!d\cT+\,\overline{{F_\frw^\tw}(\wc)}\oint_{\partial\wc}\!d\overline{\cO}\ =\ 0.
\end{align*}
This condition is also sufficient as $\wc$ has degree three and the directions $\eta_{b_k}$ are never collinear.

Similarly, for a function {$F_\frb$} we want
\[
 {F_\frb^\tb}(\bc) + \eta_{\wpr_{k}}^2 \overline{{F_\frb^\tb}(\bc)} = 2\cdot F^\tw_\frb (\wpr_{k}) \quad \text{for~all}~{k}\in\{1,2,3\}.
\]
which gives the desired solution {using the identity $\eta_\wpr^2 \,d\cT(bw^*) = d\cO(bw^*)$.}
\end{proof}

\begin{remark}\label{rq:involution}
Since {$K(b,w)\in \overline{\eta}_b\overline{\eta}_w\mathbb{R}$} {for all~$b,w$,}
the mapping $F^{\tb}\to{\eta}^2\cdot\overline{F^\tb}$ defines an involution on the kernel of $K$, which is naturally split into the invariant and anti-invariant components. The first condition in~\eqref{eq:def-twhol} says that we {consider} the invariant component only.
Let us emphasise that t-holomorphic functions form a real-linear space but not a complex-linear one.
\end{remark}

{In the finite case,} let us glue to each boundary edge an outer face with a color different from the color of the incident boundary face, see Fig.~\ref{stan_bc}. Denote the sets of {thus obtained} black and white faces by $\partial_{\operatorname{out}}B$ and $\partial_{\operatorname{out}}W$, respectively.
Denote $\overline B:=B\cup \partial_{\operatorname{out}}B$, $\overline W:=W\cup \partial_{\operatorname{out}}W$ and let $U$ be a (sub)region of the t-embedding; {we also use a notation~$\overline{B}_U$, $(\partial_\mathrm{out}B)_U$ etc for intersections of these sets with~$U$.}

\begin{figure}
\centering
\includegraphics[clip, trim=4.4cm 11.2cm 1.5cm 4.1cm ,width=0.7\textwidth]{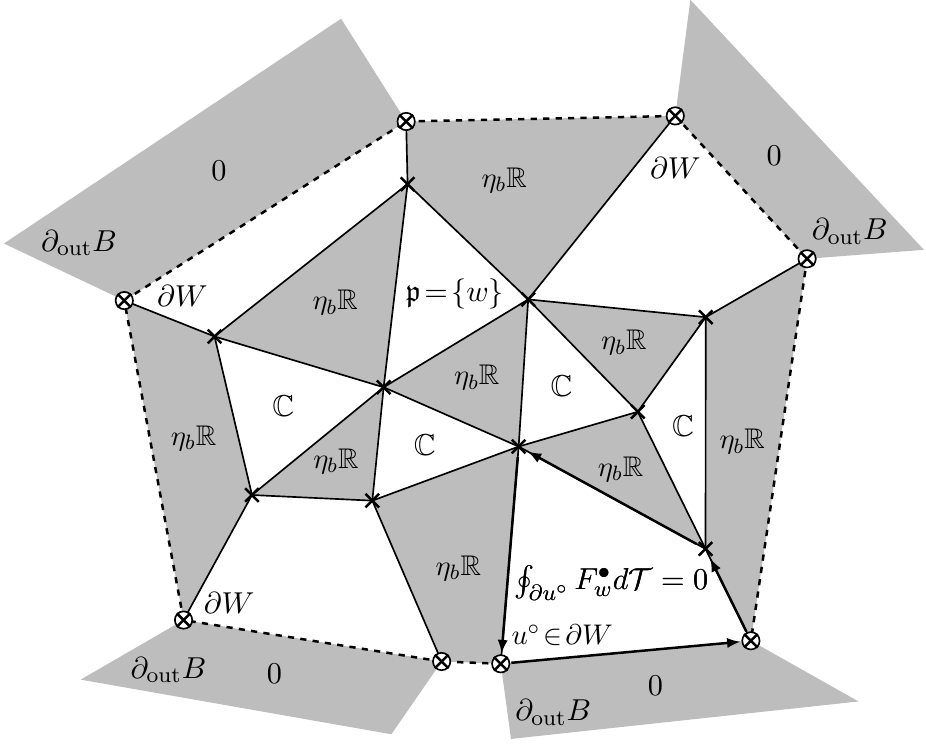}
\caption{Standard boundary conditions for a t-white-holo\-morphic function {$F_w$, $w\in W$; see Remark~\ref{rem:motivation-thol} and Definition~\ref{def:standard-bc}.}}\label{stan_bc}
\end{figure}

\begin{definition}\label{def:standard-bc}
We say that a t-white-holomorphic function $F_\frw$ defined on {a set $\overline{B}_U\cup(W_U\smallsetminus((\partial W)_U\cup\mathfrak{p}))$, $\mathfrak{p}\subset W_U$,} satisfies \emph{standard boundary conditions} if
\[
\begin{cases}
F_\frw^\tb(\bpr)=0 & \text{for all $\bpr \in {(\partial_{\operatorname{out}}B)_U}$},\\
{\oint_{\partial \wc} F^\tb_\frw\,d\cT = 0} & {\text{for all $\wc\in {(\partial W)_U\smallsetminus\mathfrak{p}}$.}}
\end{cases}
\]
{(Recall that~$\partial W\subset W$ is the set of white `inner' boundary faces of~$\cT$.)} The {standard boundary conditions} for t-black-holomorphic functions are defined similarly.
\end{definition}

Note that a t-white-holomorphic function ${F_w^\bullet}:\bpr \mapsto \overline{\eta}_w\cdot K^{-1}(w, \bpr)$ {with~$w\in W$} satisfies standard boundary conditions in the region ${U_w}$ {provided we set $F_w^\tb(b):=0$ for $b\in\partial_{\operatorname{out}} B$, see Fig.~\ref{stan_bc}. Indeed, $(F_w^\tb K)(\wc)=0$ for all $\wc\ne w$ including the boundary faces $\wc\in{\partial W_U\smallsetminus\{w\}}$. As we set $F_w^\tb(b)=0$ at the nearby outer black face, this sum coincides with the contour integral along $\partial\wc$ as before.}

\subsection{Closed forms associated to t-holomorphic functions} Let us first
summarize basic properties of t-holomorphic functions discussed in the previous section.

\begin{proposition}\label{prop:integral_simple} Let $U$ be a simply connected region in the domain of a t-embedding and
 $F_\frw$ be a t-white-holomorphic function on {a punctured} region $U_\mathfrak{p}$, {$\mathfrak{p}\subset {W}$}. Then,
 on edges not adjacent to boundary white faces {and/or} to faces of $\mathfrak{p}$,
 \begin{equation}\label{eq:Fb_dT=}
  {2}F_\frw^\tb \,d\cT =  F_\frw^\tw \,d \cT + \overline{F}{}_\frw^\tw \,d \overline{\cO}
 \end{equation}
 and {$F_\frw^\tb \,d\cT$} is a closed form  in $U_\mathfrak{p}$ away from the boundary (i.e., the integral over any closed contour $\gamma$ running over interior edges
 and not surrounding faces from~$\mathfrak{p}$ vanishes). Moreover, if $F_\frw$ satisfies standard boundary conditions, then the left-hand side of~\eqref{eq:Fb_dT=} also defines a closed form up to the boundary (i.e., $\gamma$ can then contain boundary edges too).

Similarly, if $F_\frb$ is a t-black-holomorphic function {in} $U_\mathfrak{p}$, {$\mathfrak{p}\subset {B}$}, then, on edges not adjacent to boundary black faces {and/or} to faces of $\mathfrak{p}$,
\begin{equation}\label{eq:Fw_dT=}
 {2}F^\tw_{\frb} \,d \cT =F_\frb^\tb\, d \cT + \overline{F}{}_\frb^\tb\, d \cO
\end{equation}
 and {$F^\tw_{\frb} \,d \cT$} is a closed form in $U_\mathfrak{p}$ away from the boundary. {Again, if $F_\frb$ satisfies the standard boundary conditions, then the left-hand side of~\eqref{eq:Fw_dT=} defines a closed form up to the boundary.}
\end{proposition}
\begin{proof} {See the proof of Lemma~\ref{lem:kernel_to_holomorphy}: the equalities~{\eqref{eq:Fb_dT=}, \eqref{eq:Fw_dT=}} follow from the definition of t-holomorphic functions and the identities $\eta_{b}^2d\cT(bw^*)=d\overline{\cO(bw^*)}$ and $\eta_{w}^2d\cT(bw^*)=d\cO(bw^*)$. The fact that the form $F_\frb^\tb d\cT$  (respectively, $F_\frw^\tw d\cT$) is closed is trivial around black (resp., white) faces and is equivalent to the definition of t-holomorphicity at white (resp., black) ones. The extension up to the boundary is nothing but the definition of standard boundary conditions.}
\end{proof}

{In what follows, we `primarily' think about t-holomorphic functions~$F_\frw$ and~$F_\frb$ as of~$F_\frw^\tw$ and~$F_\frb^\tb$, respectively; {let us emphasize once again that the two colors play non-symmetric roles in the definition of t-holomorphicity, so~$F_\frw^\tb$ and~$F_\frb^\tw$ are functions of a different kind whose values have complex signs~$\eta_u$ prescribed in advance, contrary to~$F_\frw^\tw$ and~$F_\frb^\tb$.} Note that the differential forms $F_\frw^\tw d\cT$ and $F_\frb^\tb d\cT$ are \emph{not} closed:} the contour integrals $\oint_{\partial \bc} F_\frw^\tw \,d \cT$ and $\oint_{\partial \wc} F_\frb^\tb \,d \cT$ do not vanish.

\begin{lemma} \label{lem:FdT-in-C}
Similarly to the definition of the origami differential form $d\cO$, one can view~\eqref{eq:Fb_dT=} and~\eqref{eq:Fw_dT=} as closed piecewise constant differential forms
\[
F_\frw^{\tw}{(z)d z} + \overline{F_\frw^{\tw}{(z)}}d \overline{\cO{(z)}}
\qquad\text{and}\qquad
F_\frb^{\tb}{(z)dz} + \overline{F_\frb^{\tb}{(z)}} d \cO{(z)}
\]
defined in the {plane} (and not just on edges of the t-embedding), where we set $F_\frw^\tw(z):=F_\frw^\tw(\wc)$ if $z\in\cT(\wc)$ and $F_\frb^\tb(z):=F_\frb^\tb(\bc)$ if $z\in\cT(\bc)$, respectively. To define the former form for $z$ inside an {interior} black face $\cT(b)$ (respectively, the latter for {$z\in\cT(w)$}) one can use any of the three values $F_\frw^\tw(\wc)$ at the adjacent white faces $\wc\sim b$ {(respectively, any of the three values~$F_\frb^\tb(\bc)$, $\bc\sim w$):} all thus obtained expressions coincide.
\end{lemma}

\begin{proof} Let us consider the form~\eqref{eq:Fb_dT=}. Its extension inside white faces is a triviality. Moreover, one can also extend this form inside a black face as ${2}F_\frw^\tb(b)dz$, \mbox{$z\in\cT(b)$}: similarly to the definition of the origami differential form~${d\cO}$, this procedure is consistent since the two sides of~\eqref{eq:Fb_dT=} match along the edge $(b\wc)^*$. Finally, note that
\[
{2}F_\frw^\tb(b)dz\ =\ 
F_\frw^\tw(\wc){dz}+\eta_b^2\overline{F_\frw^\tw(\wc)}dz\ =\ 
F_\frw^\tw(\wc)dz+\overline{F_\frw^\tw(\wc)}d\overline{\cO(z)}
\]
as $d\cO(z)=\overline{\eta}_b^2d\overline{z}$ for $z\in\cT(b)$. The other case is identical.
\end{proof}

\begin{remark} \label{rem:ext-F(z)dz}
Though Lemma~\ref{lem:FdT-in-C} does not apply to faces of higher degrees literally (and, in particular, does not apply to boundary faces of a finite triangulation; see Fig.~\ref{finTembTriangl}) it can be nevertheless extended to the full generality {by splitting faces of higher degree into triangles.} We refer the reader to Section~\ref{sec:non_triangulation} for more details.
\end{remark}

{The next proposition provides a key identity for the analysis of dimer correlation functions in Section~\ref{sec:convergence}. Since $F_\frb^\tb\equiv \mathrm{cst}$ (resp., $F_\frw^\tw\equiv \mathrm{cst}$) is a trivial example of a t-holomorphic function, it can be also viewed as a generalization of~\eqref{eq:Fb_dT=} and~\eqref{eq:Fw_dT=}.}

\begin{proposition}\label{prop:integral_product}
If $F_\frb$ and $F_\frw$ are respectively a t-black- and a t-white-holo\-morphic functions on some region $U_\mathfrak{p}$, then, on edges not adjacent to boundary faces and to faces of $\mathfrak{p}$, {the identity}
\begin{equation}
\label{eq:FFdT=}
F_\frw^\tb F_\frb^\tw\, d\cT = \tfrac{1}{2}\Re \left( F_\frw^\tw F_\frb^\tb \,d \cT  +  F^\tw_\frw \overline{F}{}^\tb_\frb\, d \cO \right)
\end{equation}
{holds and the form} $F_\frw^\tb F_\frb^\tw\, d\cT$ is closed in $U_\mathfrak{p}$ away from the boundary.

Moreover, if $F_\frw$ and $F_\frb$ satisfy standard boundary conditions, then the form $F_\frw^\tb F_\frb^\tw\, d\cT$ is closed up to the boundary (provided we set $F_\frw^\tb(\bpr)F_\frb^\tw(\wpr):=0$ for boundary edges $bw^*$).
\end{proposition}
\begin{proof}
The definition of t-holomorphicity implies that
\[
F_\frw^\tb(b)F_\frb^\tw(w)d\cT(bw^*)\ =\ \tfrac{1}{4}\big(F_\frw^\tw(w)+\eta_b^2\overline{F_\frw^\tw(w)}\big) \big(F_\frb^\tb(b)+\eta_w^2\overline{F_\frb^\tb(b)}\big)d\cT(bw^*),
\]
which gives the result since $\eta_b^2d\cT=d\overline{\cO}$, $\eta_w^2d\cT=d\cO$ and $\eta_b^2\eta_w^2d\cT=d\overline{\cT}$ on $bw^*$. {As}
\[
\textstyle \oint_{\partial\wc}F_\frw^\tb F_\frb^\tw\, d\cT = F_\frb^\tw(\wc)\oint_{\partial\wc}F_\frw^\tb\, d\cT=0,
\]
and
\[
\textstyle \oint_{\partial\bc}F_\frw^\tb F_\frb^\tw\, d\cT = F_\frw^\tb(\bc)\oint_{\partial\bc}F_\frb^\tw\, d\cT=0
\]
for $\wc\in {W_U\smallsetminus\mathfrak{p}}$ and $\bc\in {B_U\smallsetminus\mathfrak{p}}$, respectively, the expression~\eqref{eq:FFdT=} defines a closed form {on edges of the t-embedding.}
\end{proof}

\begin{remark} \label{rem:FFdT-in-C}
{Similarly to Lemma~\ref{lem:FdT-in-C}}, the form~\eqref{eq:FFdT=} can be {extended from edges of $\cT$ to} a closed {piecewise constant} differential form
\[
{\tfrac{1}{2}}\Re \big( F_\frw^\tw{(z) F_\frb^\tb{(z)}dz}+ F_\frw^\tw{(z)}\overline{F}{}^\tb_\frb{(z)}\, d \cO{(z)} \big)
\]
defined in the {complex} plane. For $z\in \cT(b)$ (and similarly for $z\in\cT(w)$), we set $F_\frb^\tb(z):=F_\frb^\tb(b)$ and use an arbitrary adjacent white face $\wc\sim b$ to define the value $F_\frw^\tw(z):=F_\frw^\tw(\wc)$. Thus obtained differential form does not depend on the choices of $\wc$ (and similar choices made for $z\in\cT(w)$). {Similarly to Remark~\ref{rem:ext-F(z)dz}, this definition does not literally apply to faces of degree more than three (including boundary ones) but can be extended to the full generality; see Section~\ref{sec:non_triangulation}.}
\end{remark}

\newcommand\F[2]{F^{\scriptscriptstyle [#1#2]}}

\subsection{Dimer coupling function as a linear combination of t-holomor\-phic ones} \label{sub:Fpmpm-def}
Let $w\in W$ and $b\in B$. As discussed in Remark~\ref{rem:motivation-thol}, the functions $F_w^\tb(\cdot):=\overline{\eta}_wK^{-1}(w,\cdot)$ and $F_b^\tw(\cdot):=\overline{\eta}_bK^{-1}(\cdot,b)$ are t-holomorphic and, in particular, admit extensions $F_w^\tw$ and $F_b^\tb$
to the inner faces of the opposite color (except $w$ and $b$, respectively) such that the conditions~\eqref{eq:def-twhol} and~\eqref{eq:def-tbhol} are fulfilled. If~$\wc_b\in W$ and $\bc_w\in B$ satisfy~$b\sim\wc_b\ne w$ and $w\sim\bc_w\ne b$, this reads as
\begin{equation}
\label{eq:K-1viaFw}
K^{-1}(w,b)\ =\ \eta_w\cdot\tfrac{1}{2}\big(F_w^\tw(\wc_b) +  \eta_b^2\overline{F_w^\tw(\wc_b)}\,\big)\ =\ \eta_b\cdot\tfrac{1}{2}\big(F_b^\tb(\bc_w) +  \eta_w^2\overline{F_b^\tb(\bc_w)}\,\big).
\end{equation}
The next proposition provides a more symmetric representation of the dimer coupling function~$K^{-1}$,
which will be particularly useful in Section~\ref{sec:convergence}.
\begin{proposition} \label{prop:Fpmpm-def}
There exist four complex-valued functions~$\F\pm\pm$,
defined on pairs $(\bc,\wc)$ of inner faces $\bc\in B\smallsetminus\partial B$ and $\wc\in W\smallsetminus\partial W$, such that

\smallskip

\noindent (i) one has $\F--(\bc,\wc)=\overline{\F++(\bc,\wc)}$ and $\F+-(\bc,\wc)=\overline{\F-+(\bc,\wc)}$;

\smallskip

\noindent (ii) the following identities hold if $w\sim\bc\ne b$ and $b\sim \wc\ne w$:
\[
F_w^\tw(\,\cdot\,)\,=\,\tfrac{1}{2}\big(\overline{\eta}_w\F++ +\eta_w\F-+\big){(\bc,\,\cdot\,),}\quad
F_b^\tb(\,\cdot\,)\,=\,\tfrac{1}{2}\big(\overline{\eta}_b\F++ +\eta_b\F+-\big){(\,\cdot\,,\wc)};
\]
moreover, for such~$w\sim\bc$ and $b\sim\wc$ one has
\[
K^{-1}(w,b)\ =\ \tfrac{1}{4}\big(\F++ +\eta_b^2\F+- +\eta_w^2\F-++\eta_w^2\eta_b^2\F--\big){(\bc,\wc)};
\]
\noindent (iii) for each $\eta\in\C$, the function $\frac{1}{2}(\overline{\eta}\F++ +\eta\F-+){(\bc,\cdot)}$ is
t-white-holo\-morphic away from~$\bc$ and 
$\frac{1}{2}(\overline{\eta}\F++ +\eta\F+-){(\cdot,\wc)}$ is t-black-holomorphic away from $\wc$.
\end{proposition}

\newcommand\cc[1]{c^{\scriptscriptstyle[#1]}}

\begin{proof} Given an inner white face $\wc$, let $\cc{+}_{\wc b}\in\eta_b\R$, $b\sim \wc$, be the (uniquely defined) triple of numbers satisfying the identities
\[
\sum_{{b:\,b\sim \wc}}\cc{+}_{\wc b}\overline{\eta}_b=2,\qquad \sum_{{b:\,b\sim \wc}}\cc{+}_{\wc b}\eta_b=0,
\]
and let $\cc{-}_{\wc b}\in\overline{\eta}_b\R$ be the complex conjugate of $\cc{+}_{\wc b}$. Note that the following identities are fulfilled for each t-white-holomorphic function~$F_\frw$:
\[
F_\frw^\tw(\wc)\,=\, \sum_{{b:\,b\sim\wc}}\cc{+}_{\wc b}\cdot\overline{\eta}_bF_\frw^\tb(b),\qquad \overline{F_\frw^\tw(\wc)}\,=\, \sum_{{ b:\,b\sim\wc}}\cc{-}_{\wc b}\cdot\overline{\eta}_bF_\frw^\tb(b),
\]
since $\overline{\eta}_bF_\frw^\tb(b)=\tfrac{1}{2}(\overline{\eta}_bF_\frw^\tw(\wc)+\eta_b\overline{F_\frw^\tw(\wc)})$. 
In particular, for $\wc\ne w$ one has
\[
F_w^\tw(\wc)\, = \, \sum_{{b:\,b\sim\wc}}\cc{+}_{\wc b}\cdot\overline{\eta}_bF_w^\tb(b) \, = \, \sum_{{b:\,b\sim\wc}}\cc{+}_{\wc b}\cdot \overline{\eta}_b\overline{\eta}_wK^{-1}(w,b)
\]
and similarly for the conjugate, with the coefficients $\cc{+}_{\wc b}$ replaced by $\cc{-}_{\wc b}$.

Given an inner black face $\bc$, let $\cc{+}_{\bc w}\in\eta_w\R$, $w\sim \bc$, be defined by the identities
\[
\sum_{{w:\,w\sim \bc}}\cc{+}_{\bc w}\overline{\eta}_w=2,\qquad \sum_{{w:\,w\sim \bc}}\cc{+}_{\bc w}\eta_w=0,
\]
and let $\cc{-}_{\bc w}\in\overline{\eta}_w\R$ be their complex conjugate. For $\bc\ne b$, the t-holomorphi\-city of $F_b$ implies
\[
F_b^\tb(\bc)\, = \, \sum_{{w:\,w\sim\bc}}\cc{+}_{\bc w}\cdot \overline{\eta}_wF_b^\tw(w) \, = \, \sum_{{w:\,w\sim\bc}}\cc{+}_{\bc w}\cdot \overline{\eta}_b\overline{\eta}_wK^{-1}(w,b)
\]
and similarly for the conjugate, with the coefficients $\cc{+}_{\bc w}$ replaced by $\cc{-}_{\bc w}$.

Now, for inner faces $\bc\in B\smallsetminus\partial B$ and $\wc\in W\smallsetminus \partial W$, define
\[
 \F\pm\pm(\bc,\wc)\ :=\ \sum_{{w:\,w\sim\bc}}\,\sum_{{b:\,b\sim \wc}}\cc{\pm}_{\bc w}\cc{\pm}_{\wc b}\cdot\overline{\eta}_b\overline{\eta}_wK^{-1}(w,b),
\]
where the superscript of $\cc{\pm}_{\bc w}$ corresponds to the first superscript of $\F\pm\pm$ and that of $\cc{\pm}_{\wc w}$ to the second one. Since $\overline{\eta}_b\overline{\eta}_wK^{-1}(w,b)\in\R$, the property (i) holds automatically.

Let us now prove the identities (ii). If $w\sim \bc\ne b$ and $\wc\ne w$, then
\begin{align*}
\overline{\eta}_w\F++({\bc},\wc)&+\eta_w\F-+({\bc},\wc)\\ & =\ \sum_{{w':\,w'\sim\bc}}\, \sum_{{b:\,b\sim \wc}}(\overline{\eta}_w\cc{+}_{\bc w'}+\eta_w\cc{-}_{\bc w'})\cc{+}_{\wc b}\cdot\overline{\eta}_b\overline{\eta}_{w'}K^{-1}(w',b)\\
&=\ \sum_{{b:\,b\sim\wc}}\cc{+}_{\wc b}\big(\,\overline{\eta}_wF_b^\tb(\bc)+\eta_w\overline{F_b^\tb(\bc)}\,)\\
&=\ \sum_{{b:\,b\sim\wc}}\cc{+}_{\wc b}\cdot 2\overline{\eta}_wF_b^\tw(w)\ =\ \sum_{{b:\,b\sim\wc}}\cc{+}_{\wc b}\cdot 2\overline{\eta}_bF_w^\tb(b)\ = \ 2F_w^\tw(\wc).
\end{align*}
A similar identity for the function~$F_b^\tb(\bc)$ follows from the same arguments and the formula for $K^{-1}(w,b)$ follows, e.g., from~\eqref{eq:K-1viaFw}.

Finally, note that (iii) holds if $\eta=\eta_{w}$, $w\sim \bc$ (or $\eta=\eta_b$, $b\sim \wc$, respectively). The result for all $\eta\in\C$ follows from the fact that t-holomorphic functions form a real-linear vector space.
\end{proof}

\section{T-embeddings and T-graphs}\label{sec:Tgraph}
We still assume that $\cT(\G^*)$ is a triangulation in this section.
Our approach to the properties (in particular the regularity) of t-holomorphic functions will be to link them to harmonic functions on related graphs called {\emph{T-graphs}, which were first introduced in~\cite{kenyon-sheffield}.} {We recall the definition of T-graphs and discuss basic properties of random walks on them in Section~\ref{sub:t-graphs}. The link (similar to~\cite[Lemma~2.4]{kenyon-honeycomb}) between t-holomorphic functions and harmonic functions on T-graphs is discussed in Section~\ref{sub:harmonicity}. Section~\ref{sub:harmonicity-1} contains a new material: another link between t-holomorphic functions and \emph{time-reversed} random walks on T-graphs.}

\subsection{T-graphs and their random walks}\label{sub:t-graphs}
In this section, we consider the image of $\G^*$ {under the mapping} $\cT + \cO$ and relate it to the geometry of $\cT$. We allow ourselves a similar abuse of the notation for $\cO$ and $\cT$ by {viewing} them {both} as complex-valued functions {defined on} an abstract graph $\G^*$ and as functions {defined in (a subset of)} $\C$. 
Note that in the latter case $\cT$ is just the identity {mapping}.

\begin{definition}\label{def:Tgraph}
A (non-degenerate) \emph{T-graph} in the whole plane is a closed {path-connected} subset of~$\C$ which can be written as the disjoint, locally finite, union of a countable number of open segments.

A finite {(non-degenerate)} T-graph is a closed {path-connected} subset of $\C$ which can be written as the disjoint union of a finite number of open segments and a finite {number} of single points, named \emph{`boundary vertices'}, {each of which is adjacent either to a single open segment or to a pair of those lying on the same line; see Fig.~\ref{fig:T-graph}.}

{We say that a finite T-graph has the \emph{topology of the disc} if all its `boundary vertices' are adjacent to the unbounded connected component of its complement.}
\end{definition}

\begin{figure}
\centering
\begin{minipage}{0.35\textwidth}
\includegraphics[clip, trim=4.4cm 10.4cm 6.25cm 4.1cm,width=\textwidth]{t-emb_finite.pdf}
\end{minipage}
\hskip 0.05\textwidth
\begin{minipage}{0.59\textwidth}
\includegraphics[clip, trim=4.45cm 15.8cm 6.2cm 4.1cm ,width=\textwidth]{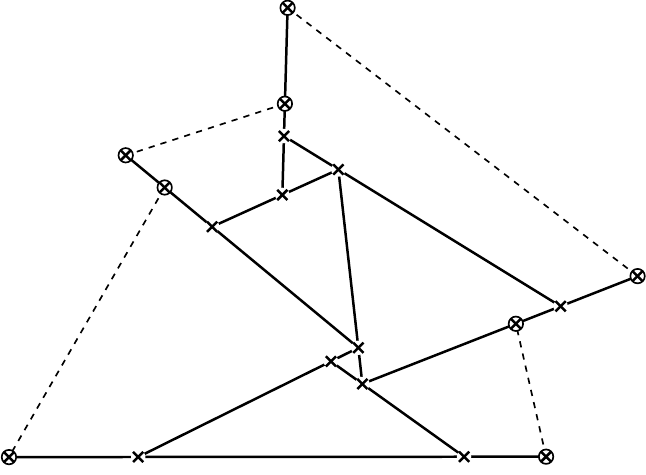}
\end{minipage}
\caption{{A (rescaled) non-degenerate T-graph $\cT+\alpha^2\cO$ (right) obtained from a finite t-embedding~$\cT$ (left, see also Fig.~\ref{finTembTriangl}): faces $w\in W$ of~$\cT$ correspond to faces of the T-graph whilst $b\in B$ are flattened to the segments (and vice versa in T-graphs $\cT+\overline{\alpha^2\cO}$).} {When~$\alpha$ varies, some of the faces degenerate; see Fig.~\ref{fig:t-graph-degen} below.}}
\label{fig:T-graph}
\end{figure}

Note that since the union of {open} segments {is required to} form a closed set, the endpoints of {each} segment have to lie {either} inside another segment or at a boundary point. Furthermore, this is the only way two segments can meet so the name refers to the fact that {each} vertex of a \emph{non-degenerate} T-graph {typically} looks like a {\fontfamily{ptm}\selectfont T} or a {\fontfamily{ptm}\selectfont K} {(or, in more involved situations, like an~{\fontfamily{ptm}\selectfont X} with one of the segments split into two at the intersection point etc)} but not like a {\fontfamily{ptm}\selectfont Y}; {see Fig.~\ref{fig:T-graph} and Fig.~\ref{fig:t-graph-degen}.}

\begin{definition}\label{def:Tgraph_degenerate}
A T-graph with possibly degenerate faces {in} the whole plane is a {disjoint, locally finite union of open segments and single points called \emph{degenerate faces} such that the following conditions hold (see also Fig.~\ref{fig:t-graph-degen}):
\begin{itemize}
\item each of the endpoints of an open segment either lies inside another segment as in the non-degenerate case, or coincides with a degenerate face;
\item each degenerate face is the endpoint of~$n+m$ open segments, among which $n\ge 3$ are called \emph{outgoing} and $m\ge 0$ \emph{incoming}, with a restriction that the directions of outgoing segments are not contained in a half-plane;
\item in the latter case we say that this degenerate face has degree~$n$ and assign to it an `infinitesimal' convex \mbox{$n$-gon} (i.e., an equivalence class of polygons considered up to homotheties) with sides parallel to the outgoing segments (note that for~$n=3$ no additional data are actually required as the directions of the sides define such an `infinitesimal' triangle uniquely).
\end{itemize}}
\noindent The definition in the finite case is similar. 
\end{definition}

\begin{figure}
\centering
\begin{minipage}{0.32\textwidth}
\includegraphics[clip, trim=4.4cm 18.6cm 11.1cm 4.1cm, width=\textwidth]{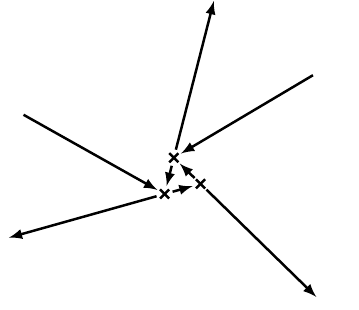}
\end{minipage}
\hskip 0.01\textwidth
\begin{minipage}{0.32\textwidth}
\includegraphics[clip, trim=4.4cm 12.3cm 11.1cm 10.4cm, width=\textwidth]{t-graph-degen.pdf}
\end{minipage}
\hskip 0.01\textwidth
\begin{minipage}{0.32\textwidth}
\includegraphics[clip, trim=4.4cm 5.8cm 11.1cm 16.9cm, width=\textwidth]{t-graph-degen.pdf}
\end{minipage}
\caption{{An example of a triangular face in the T-graph~$\cT+\alpha^2\cO$ for three consecutive values of~$\alpha\in\mathbb{T}$: the face degenerates in the central picture (this corresponds to~$n=3$ and $m=2$ in Definition~\ref{def:Tgraph_degenerate}). The arrows indicate possible transitions for the random} {walks on these T-graphs; see Definitions~\ref{def:chain}, \ref{def:chain-deg}} {and Remark~\ref{rq:degenerate_face}.}
\label{fig:t-graph-degen}}
\end{figure}

\begin{proposition}\label{prop:geomT}
For {each} $\alpha\in\C$ with $|\alpha| = 1$, the {image of~$\G^*$ under the mapping} $\cT + \alpha^2\cO$ is a T-graph, possibly with degenerate faces. In this T-graph:

(i) for {each $w \in W$,} the image of $w$ is a translate of $(1\!+\!\alpha^2\eta_w^2)\cT(w)$;

(ii) for {each} $b \in B$, 
the image of $b$ is a translate of $2\Pr( \cT(b), {\alpha}\overline{\eta}_b \R)$.

\noindent For a generic choice of $\alpha$, no face {of $\cT + \alpha^2\cO$} is degenerate.
\end{proposition}
\begin{proof}
Let us start by identifying the image of faces of $\G^*$.
On a white face $w$, one has ${d (\cT + \alpha^2\cO) = ( 1 + \alpha^2\eta_w^2) dz}$ which proves the first item. The second item is identical, so we just need to show that $\cT + \alpha^2\cO$ is a T-graph. The angle property of a t-embedding together with the fact that all white faces preserve the orientation imply that the end of each segment either lies on some other segment or belongs to a degenerate face. Therefore $\cT + \alpha^2\cO$ is a union of segments which satisfies Definition~\ref{def:Tgraph_degenerate} except the fact that the segments are disjoint.

Let us show that there are no overlaps. Suppose that the images of two white faces $w$,~$w'$ overlap. {Choose} a point ${z \in (\cT+\cO)( w) \cap (\cT + \cO)( w')}$ such that $z$ is not on any segment of the T-graph. Recall that  $\cO$ and $\cT$ can be seen as functions from $\C$ to $\C$ and in this case, $\cT$ is just the identity.
Let us orient edges of $\cT$ in the counterclockwise direction around each white face. Let~$\gamma$ be an oriented closed edge path surrounding both $\cT(w)$ and $\cT(w')$. Note that $\cO$ {is} defined up to an additive constant, so we can assume that $\gamma$ {surrounds the} point~$z$. Since the orientation of all white faces of $\cT + \cO$ is the same as the orientation of white faces of $\cT$ the winding of $(\cT+\cO)( \gamma)$ around $z$ is at least~$4\pi$. On the other hand, we have clearly $|\cO( z') - \cO(z)| \leq |\cT(z') - \cT(z)|$ for all $z'\in \gamma$ so by the Rouch\'e theorem (or ``dog on a leash'' lemma) the winding of $(\cT+\cO)( \gamma)$ around $z$ is the same as the winding of $\gamma$ around $z$, which is~$2\pi$. This is a contradiction.
\end{proof}

Note that by Definition~\ref{def:eta} and Definition~\ref{def:O}, $\alpha^2\cO$ is just the origami map corresponding to the 
origami square root function ${\alpha}\eta$. Also note that for any white face $w$, its image is degenerate exactly for $\alpha^2 = - \overline{\eta}_w^2$. {In what follows we focus our attention on the T-graph} $\cT+\cO$ without loss of generality.

\begin{definition}\label{def:chain}
The (continuous time) random walk on a {whole plane} T-graph with no degenerate faces is the Markov chain with the following transition rates. For any interior vertex~$v$, there exists a unique segment $(v^-,v^+)$ such that $v \in (v^-,v^+)$. 
We set
\[
{q}( v \to v^\pm)\ :=\ \frac{1}{|v^\pm - v|\cdot|v^+ - v^-|}\,
\]
{and} all other transitions have probability zero.

In the finite triangulation case, for each edge of the T-graph corresponding to a \emph{boundary} black face of~$\cT$ {(recall that such faces have degree four)} we make a choice between two options to split this face into two triangles and define the possible transitions accordingly; see Fig.~\ref{boundary_t_graph}. Boundary vertices act as sinks for the chain.
\end{definition}

\begin{figure}
\center{\includegraphics[width=\textwidth]{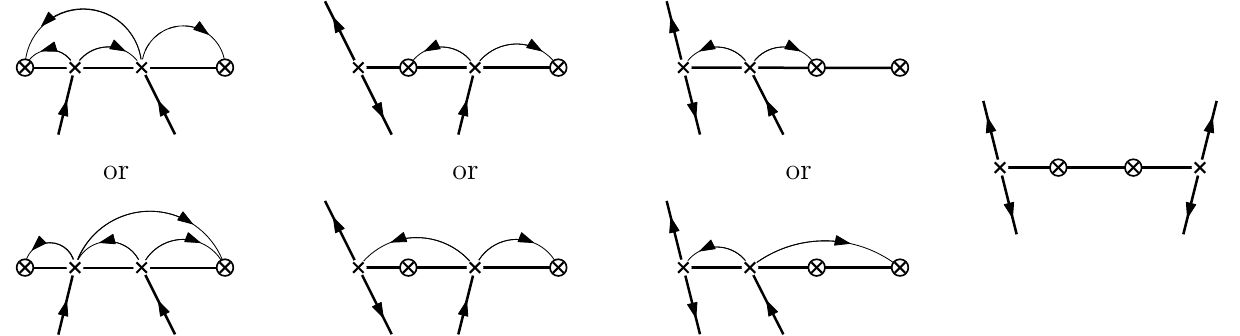}}
\caption{Possible configurations of a T-graph near the boundary of a finite triangulation. The circled crosses are boundary vertices (sinks) and the arrows indicate possible transitions for the random walk; see also Section~\ref{sec:non_triangulation} {for a more general discussion.}}\label{boundary_t_graph}
\end{figure}

\begin{remark}
The Markov chain $X_t$ defined above is a martingale. The choice of transition rates is made so that it fits the expected time for a Brownian motion started at $v$ {and} moving along the segment $[v^-, v^+]$ to hit the endpoints. In particular, {in the whole plane case one has} $\operatorname{Tr}(\operatorname{Var}( X_t ) )= t$ for all $t\ge 0$.
\end{remark}

For T-graphs associated to t-embeddings {with triangular faces,} Definition~\ref{def:chain} can be naturally extended to degenerate faces as follows.

\begin{definition}\label{def:chain-deg}
Consider a T-graph of the form $\cT +  \cO$ and suppose $v =( \cT +  \cO)(w)$ is a degenerate {triangular} face. Let $b_1, b_2, b_3$ be the adjacent {faces} of $w$ in ${\G^*}$ and let $v_1, v_2, v_3$ be the endpoints of the corresponding segments in $\cT +\cO$. We define transition rates for the random walk from~$v$~as
\[
{q}( v \to v_k ) := \frac{{m_k}}{|v_k - v|^2},\qquad 
{m_k} := \frac{|d\cT( b_kw^* )|\cdot |v_k - v|}{\sum_{j=1}^3 |d\cT( b_jw^*) |\cdot |v_j - v|}.
\]
\end{definition}

\begin{remark}\label{rq:degenerate_face}
One can understand {these} transition probabilities as follows. The {degenerate} vertex~$v$ corresponds to three vertices of non-degenerate T-graphs \mbox{$\cT + {\alpha^2}\cO$}, with~$\alpha\to 1$. For~$\alpha=1$, these vertices form a face of diameter $0$ but $v$ still contains the information on the aspect ratio of~$\cT(w)$. In particular, each of these three collapsed vertices now have a possible transition to one of the~$v_k$'s with the rate~$|v_k-v|^{-2}$ and a transition to other vertex with infinite rate; {see Fig.~\ref{fig:t-graph-degen}.} These infinite rates still depend on the geometry {of~$\cT(w)$} and have invariant measure ${m_k}$. {Clearly,} this invariant measure just multiplies the rates of the long jumps.
\end{remark}

It is not hard to see that the law of the (continuous time) random walk on $\cT + \alpha^2 \cO$ is continuous in $\alpha$, including those producing degenerate faces, cf.~Remark~\ref{rq:degenerate_face}.

We now make the transition probabilities more explicit in terms of the geometry of the \mbox{t-embedding} itself. For this recall Lemma~\ref{lem:increment_phi} and
note that it implies that the {values} $\phi_w$ around a black face of $\G^*$ are {monotone} with a single jump of $\pi$.
If $v$ is a non-degenerate vertex of~$\cT + \cO$, denote by $b(v)$ the unique black face such that $v$ is an interior point of {the segment} $(\cT+\cO)({b(v)})$. If~${v = (\cT + \cO)(w)}$ is a degenerate face, define by $b_1(v), b_2(v), b_3(v)$ the three black faces adjacent to~$w$.  Finally, denote the area of a triangle $\cT(b)$ by $S_b$.

\begin{lemma}\label{lem:proj}
Let $b$ be an inner black face of $\G^*$, 
let $A$, $B$, $C$ be vertices of the triangle $\cT(b)$ {listed counterclockwise}, let $w_A$, $w_B$, $w_C$ be the {opposite} white faces {adjacent to~$b$}, and let ${v_A = (\cT + \cO)(A)}$ 
{and similarly for~$v_B$ and $v_C$; see Fig.~\ref{fig:ABC-transitions}.} 
Then, the following holds:

\noindent (i) the vertex $v_A$ {lies in} the interior of the segment $(\cT + \cO)(b)$ (i.e., $b = b(v_A)$) if and only if
\[{-\pi/2 <}~\phi_{w_C} < \phi_{w_A} < \phi_{w_B} < \pi/2\,;\]
\noindent (ii)  if $b = b(v_A)$, then
\[
{q}( v_A \to v_B ) =
\frac{\tan(\phi_{w_A}) - \tan(\phi_{w_C})}{8 S_{b}}\ \ \text{and}\ \
{q}( v_A \to v_C ) = \frac{\tan(\phi_{w_B}) - \tan(\phi_{w_A})}{8 S_{b}}\,;
\]
\noindent (iii) vertices $v_A$ and $v_B$ coincide if and only if 
$\phi_{w_C} = \pi/2$. In this case
\[
{q}( v \to v_C ) = \frac{\tan(\phi_{w_B}) - \tan(\phi_{w_A})}{8(S_{b_1(v)}+S_{b_2(v)}+S_{b_3(v)})}\,,\qquad {v=v_A=v_B\,.}
\]
\end{lemma}

\begin{proof}
Note that $\Pr( \cT(b) ,\overline{\eta}_b{\R}) = \overline{\eta}_b \Pr( {\eta}_b \cT(b), \R)$ and that by definition of the origami square root function $\eta$,
{the sides of the triangle~$\eta_b\cT(b)$ are parallel to the lines~$\overline{\eta}_{w_A}\R, \overline{\eta}_{w_B}\R, \overline{\eta}_{w_C}\R$. Assuming that these lines are not vertical,}
 it is clear that the angle equality {from Lemma~\ref{lem:increment_phi}} holds without an additional $\pi$ at the rightmost {and the leftmost vertices of the triangle} $ \eta_b \cT(b)$. 
Therefore, the interior {vertex} of ${\Pr(}\eta_b \cT(b),\R)$ corresponds to the vertex of $\cT(b)$ where the angle equality holds with an additional $\pi$. {This proves (i).}

\begin{figure}
\center{\includegraphics[width=\textwidth]{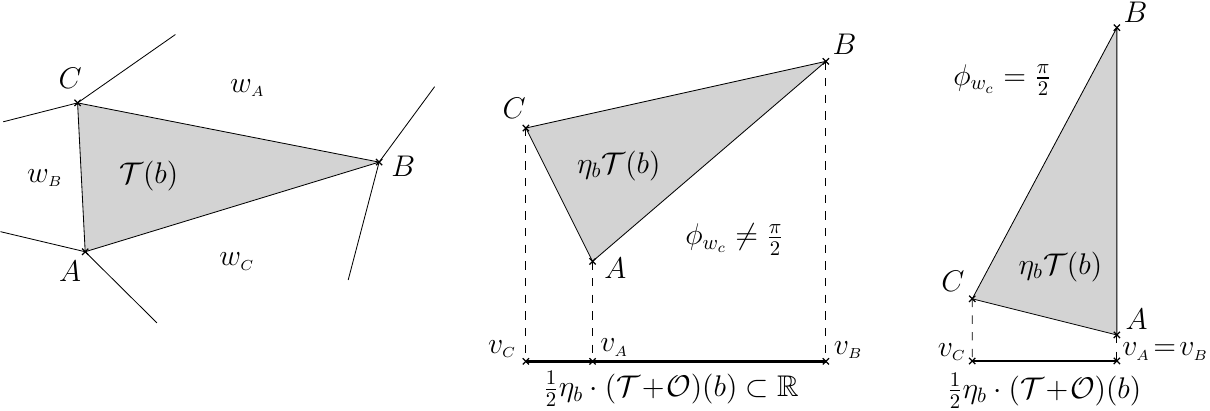}}
\caption{Notation used in Lemma~\ref{lem:proj}: a triangle~$\cT(b)$ in a t-embedding {(left) and} the corresponding edge of the T-graph~$\cT+\cO$ in the non-degenerate {(center) and degenerate (right) cases.}\label{fig:ABC-transitions}}
\end{figure}

To compute the transition rates {from the vertex~$v_A$ such that $b\!=\!b(A)$}, note that
\begin{align*}
 |v_B-v_A|\ &=\  2 \cos( \phi_{w_C} ) |AB|\,,  \\
 |v_A-v_C|\ &=\  2 \cos( \phi_{w_B} ) |AC|\,, \\
 |v_B-v_C|\ &=\ 2\cos( \phi_{w_A }) |BC|\,.
\end{align*}
This gives the transition rate
\[
{q}( v_A \to v_B )\ =\ \frac{1}{2 \cos( \phi_{w_C} ) |AB|\cdot 2\cos( \phi_{w_A} ) |BC|}\ =\ \frac{\sin( \theta(b, B))}{8\cos( \phi_{w_C} )\cos( \phi_{w_A} ) S_b },
\]
where ${\theta(b,B)=\phi_{w_A} - \phi_{w_C}}$ is the angle {of} the triangle $\cT(b)$ at the {vertex} $B$.
{This implies~(ii).}

{In} the degenerate case {(iii),} it is obvious that $v_A = v_B$ if and only if 
$\phi_{w_C} = \pi/2$ and that in this case $v_C$ is the other endpoint of the segment $ (\cT\!+\!\cO)(b)$. First, we note that
\[
|\cT( bw_C^*) | \cdot |v_B - v_C|\ =\ |AB| \cdot {2}\cos (\phi_{w_A})|BC|\ =\ {4} S_b\cdot \frac{\cos \phi_{w_A}}{\sin \theta(b,B)}\ =\ {4S_b}
\]
since $\phi_{w_C} = \pi/2$ and hence $\phi_{w_A} = \theta(b, B) - \pi/2$. This shows
that
\[
{m_k}{} = {S_{b_k}}/({S_{b_1} + S_{b_2} + S_{b_3}}).
\]
The rest of the proof is a simple computation similar to the non-degenerate case.
\end{proof}

We now give a simple geometric expression for the invariant measure of the random walk on the T-graph $\cT+\cO$. {It is worth noting that in fact we have the \emph{same} measure for each of the random walks on~$\cT+\alpha^2\cO$, even though these T-graphs are quite different for different~$\alpha\in\mathbb{T}$.}

\begin{corollary}\label{cor:invariant_measure}
For a whole plane T-graph~$\cT+\cO$, define an infinite measure on its vertices by $\mu(v) := S_{b(v)}$ if $v$ is not a degenerate vertex of $\cT+ \cO$ and ${\mu(v) := S_{b_1(v)} + S_{b_2(v)} +S_{b_3(v)}}$ if $v$  is a degenerate one. 
The measure $\mu$ is invariant for the random walk on $\cT+ \cO$ {defined above.}
\end{corollary}
\begin{proof}
First, assume that there are no degenerate faces. {Consider} a vertex~$v$ of~$\cT+\cO$ and let~$2n$ be its degree in~$\G^*$.
The {consecutive} values $\phi_{w_j}$ for the white faces adjacent to $v$ differ by either $ \theta(b_j,v)$ or ${\theta(b_j,v)}-\pi$ (where $b_j$ are the black vertices around~$v$). 
{Moreover,} it is easy to see that there is exactly one increment $\theta(b_j,v)-\pi$ 
and without loss of generality {we can assume that this is} the increment from $ \phi_{w_n}$ to $ \phi_{w_1}$. By Lemma~\ref{lem:proj}, the two outgoing edges from $v$ {lead to} the two other vertices of $b_n$ and the total outgoing rate is ${\frac18}(\tan \phi_{w_n} - \tan \phi_{w_1})$.

For the incoming rate, we see from Lemma~\ref{lem:proj} that the incoming rate through the edge {of $\cT+\cO$} corresponding to~$b_j$ is ${\frac18}(\tan \phi_{w_{j+1}} - \tan \phi_{w_j})$, independently of which vertex it comes from. Therefore the total incoming rate is also ${\frac18}(\tan \phi_{w_n} - \tan \phi_{w_1})$, which concludes the proof.

Finally, it is {not hard} to check that, if $v=(\cT+\cO)(w)$ is a degenerate face, {then} the above arguments still hold, one only needs to consider more possible transitions. {An alternative -- and more conceptual -- argument is to use the continuity of the random walks on~$\cT+\alpha^2\cO$ with respect to~$\alpha$; see~Remark~\ref{rq:degenerate_face}.}
\end{proof}

{\begin{remark} In the finite case, one clearly cannot define a true invariant measure in presence of the absorbing {boundary}. Nevertheless, let us note that the definition of~$\mu$ given above still makes sense. More precisely, recall that the definition of~$X_t$ on a segment~$(\cT+\cO)(b)$ obtained from a boundary quad~$b$ requires a choice of a decomposition of~$b$ into two triangles; see Fig.~\ref{boundary_t_graph}. If~$v$ is an inner vertex of~$\cT+\cO$ on such a segment, then we set~$\mu(v)$ to be the area of the corresponding triangle and not that of~$b$. It is easy to see that thus defined measure~$\mu$ is subinvariant.
\end{remark}}

Clearly, one can exchange the roles of black and white faces replacing the origami map~$\cO$ by its conjugate~$\overline{\cO}$. {Below we list} properties of thus obtained T-graphs with flattened white faces.

\begin{proposition}\label{prop:white_Tgraph}
For each $\alpha\in\mathbb{T}$, the mapping $\cT + {\overline{\alpha^2\cO}}$ defines a T-graph {and}

\smallskip 
\noindent (i) for {each} ${w\in W}$, the edge $(\cT + {\overline{\alpha^2\cO}})(w)$ is a translate of $2\Pr( \cT(w), {\overline{\alpha}}\overline{\eta}_w \R)$;

\smallskip 
\noindent (ii) for {each} ${b\in B}$, the face $(\cT + {\overline{\alpha^2\cO}})(b)$ is a translate of $(1 + {\overline{\alpha}^2}\eta_b^2 ) \cT(b)$.

\smallskip 

\noindent (iii) Let~$\alpha=1$ {and $w={\scriptstyle \triangle} ABC$ be an interior white face} of $\cT$. The vertex \mbox{$v_A:=(\cT+\overline{\cO})(A)$} lies in the interior of the {corresponding} segment $(\cT + \overline{\cO} )(w)$ {of the T-graph} if and only if
$-\pi/2<\phi_{b_C}<\phi_{b_A}<\phi_{b_B}<\pi/2$.

\smallskip 
\noindent (iv) In the above case, the transition rates of the random walk are
\[
{q}( v_A \to v_B ) = \frac{\tan \phi_{b_A} - \tan \phi_{b_C}}{8 S_w}
\quad {\text{and}}\quad
{q}( v_A \to v_C) = \frac{\tan \phi_{b_B} - \tan \phi_{b_A}}{8 S_w}.
\]


\noindent (v) Vertices $v_A$ and $v_B$ coincide if and only if 
$\phi_{b_C} = \pi/2$. In this case,
\[
{q}( v\to v_C ) = \frac{\tan(\phi_{b_B}) - \tan(\phi_{b_A})}{8(S_{w_1(v)}+S_{w_2(v)}+S_{w_3(v)})}\,,\qquad {v=v_A=v_B\,}.
\]


{\noindent (vi) The invariant measure~$\mu$ for the random walk discussed above is given by $\mu(v):=S_{w(v)}$ if $v$ is non-degenerate and~$\mu(v):=S_{w_1(v)}+S_{w_2(v)}+S_{w_3(v)}$ otherwise.}
\end{proposition}
\begin{proof}
The proof mimics the case of T-graphs~$\cT+\alpha^2\cO$ with flattened black faces.\!\!
\end{proof}


\subsection{T-holomorphic functions as derivatives of harmonic functions on \mbox{T-graphs}}\label{sub:harmonicity}
{In this section we present a relation between t-holomorphic functions on t-embeddings and harmonic ones on \mbox{T-graphs}, similar to~\cite[Lemma~2.4]{kenyon-honeycomb}. The harmonic functions are understood in the usual sense: $H$ is harmonic on~$\cT+\alpha^2\cO$ if~$H(X_t)$ is a martingale for the corresponding random walk~$X_t$. In the whole plane case, such a function can be naturally} extended {onto} segments of the T-graph {in a linear way. Moreover,} if all faces of the T-graph are triangles, then it can also be extended as a function on $\C$ which is affine on each face. Conversely, it is easy to see that any such piecewise affine function restricts to a harmonic function on vertices of the T-graph. In the finite case, a similar correspondence holds for piecewise affine functions defined on the union of \emph{interior} faces only, recall that boundary faces are not triangles; see Fig.~\ref{finTembTriangl}.

\begin{definition}\label{def:D}
On a non-degenerate T-graph~{$\cT+\alpha^2\cO$ with flattened black faces}, we define a derivative operator $\rD$, acting on real-valued harmonic functions~$H$, by specifying that
\begin{equation}
\label{eq:rD-def}
\begin{array}{ll}
dH = \rD[H](b)dz &\text{along {each} segment ${(\cT\!+\!\alpha^2\cO)}(b)$,}\\
dH = \Re(\rD[H](w)dz) & \text{inside each {inner} face~${(\cT\!+\!\alpha^2\cO)}(w)$.}
\end{array}
\end{equation}
If $v = (\cT+\alpha^2\cO)(w)$ is a degenerate face, calling {$b_k=b_k(v)$, $k=1,2,3$,} the neighbouring black faces as before, we define $\rD[H]( w)$ as the unique complex number such that
\[
\rD[H](b_k)=\Pr(\rD[H](w),{\overline{\alpha}\eta_{b_k}\R})\quad \text{for~all}~k\in\{1,2,3\}.
\]
Note that~${\overline{\alpha}\eta_{b_k}\R}$ is the {conjugated} direction of the segment $(\cT+\alpha^2\cO)(b_k)$, and that the above relation also holds around non-degenerate faces~$w$ due to~\eqref{eq:rD-def}.
\end{definition}

We need to check that the definition {of~$\rD[H](w)$ for degenerate faces} makes sense. {Denote \mbox{$b_k=b_k(v)$} for shortness.} By harmonicity, we have {the identity}
\[
\sum_{k=1}^3 \frac{|d\cT(b_kw^*)| |v_k - v|}{|v_k - v|^2}\cdot (v_k - v)\rD[H]( b_k)\ =\ 0,
\]
which simplifies into
\[
\sum_{k=1}^3 {\alpha}\rD[H](b_k)\cdot\overline{\eta}_{b_k} |d\cT(b_kw^*)|\ =\ \eta_w \sum_{k=1}^3 {\alpha}\rD[H]( b_k) \cdot d\cT(b_kw^*)\ =\ 0.
\]
{This is exactly} the condition from Lemma~\ref{lem:kernel_to_holomorphy} which ensures {the existence of a complex number~$\alpha\rD[H](w)$ with prescribed projections~$\alpha\rD[H](b_k)$ onto the lines~$\eta_{b_k}\R$.}

\begin{remark}\label{rem:def-D} {Definition~\ref{def:D}} extends 
to complex multiples of {real-valued harmonic functions by linearity (note however that one cannot extend it to all complex-valued~$H$ as the definition of~$\rD[H](w)$ is not complex-linear). For what follows, a particularly important case is when~$H$ {is~$\alpha\R$-valued.} For such functions, we have $\rD[H](b)=\Pr(\rD[H](w),\eta_b\R)$ if~$b\sim w$.}
\end{remark}

In other words, if~$H$ is an $\alpha\R$-valued harmonic function on~$\cT+\alpha^2\cO$, then its derivative~$\rD[H]$ satisfies the t-holomorphicity condition. The next definition provides the inverse operation.
\begin{definition} \label{def:rI} Let a function~$F_\frw$ be t-white-holomorphic on (a subset of) the t-embedding~$\cT$. We denote by~$\rI_\C[F_{\frw}]$ {a} primitive of the form~\eqref{eq:Fb_dT=}. Similarly, for a t-black-holomorphic function~$F_\frb$, let~$\rI_\C[F_{\frb}]$ be {a} primitive of~\eqref{eq:Fw_dT=}. Further, let~$\rI_{\alpha\R}[F_{{\frw}}]:=\Pr(\rI_\C[F_{{\frw}}];\alpha\R)$ for~$\alpha\in\mathbb{T}$ {and similarly for~$F_\frb$.}
\end{definition}

\begin{proposition}\label{prop:representation_derivative}
Let $F_\frw$ be a t-white-holomorphic function {and $\alpha$ be} in the unit circle. {The function $\rI_{\alpha\R}[F_\frw]$} 
is harmonic on {the T-graph} $\cT + \alpha^2 \cO$, except possibly on segments containing boundary vertices. Furthermore, $F_\frw =  \rD [\rI_{\alpha\R}[F_\frw]]$ away from the boundary. If $F_\frw$ satisfies standard boundary {conditions,} then the function $ \rI_{\alpha\R}[F_\frw]$ is harmonic up to the boundary.


{The same statements hold for t-black-holomorphic functions: $\rI_{\overline{\alpha}\R}[F_\frb]$
is harmonic on {the T-graph} $\cT + \overline{\alpha^2 \cO}$, up to the boundary if $F_\frb$ satisfies standard boundary conditions, and $F_\frb=\rD[\rI_{\overline{\alpha}\R}[F_\frb]]$.}
\end{proposition}
\begin{proof}
{Consider a t-white-holomorphic function~$F_\frw$,} two vertices $v$, $v'$ of $\G^*$, and let $b$, $w$ be the black and the white faces of $\G^*$ adjacent to the edge~$(vv')$. Assume first that they are not boundary faces. Let us check that all relations are consistent for increments between $v$ and $v'$.
{Due to~\eqref{eq:Fb_dT=},} one has
\begin{align*}
\Pr\big( \rI_\C[F_\frw](v') - \rI_\C[F_\frw](v), \alpha \R \big)\ & =\ F_\frw^\tb(b) \,d\cT (vv') + \alpha^2 \overline {F_\frw^\tb(b)\, d \cT (vv')} \\
 & =\  F_\frw^\tb(b)\cdot\big(d\cT (vv') + \alpha^2 \overline{\eta}_b^2\, \overline{d \cT(vv')}\big)\\
  & =\ F_\frw^\tb(b)\cdot (d\cT + \alpha^2d\cO)(vv').
\end{align*}
In particular the increments of $\rI_{\alpha\R}[F_\frw]$ are linear along {each segment} of $\cT + \alpha^2 \cO$, hence $\rI_{\alpha\R}[F_\frw]$ is harmonic and for all $b$ one has $\rD[{\rI_{\alpha\R}[F_\frw]}](b) = F_\frw^{{\tb}}(b)$. For white faces, the following holds:
\begin{align*}
\Pr\big( \rI_\C[F_\frw](v') &- \rI_\C[F_\frw](v), \alpha \R \big)\\
&=\ {(F_\frw^\tw(w)d\cT + \overline{F}{}_\frw^\tw(w)d\overline{\cO}+\alpha^2\overline{F}{}_\frw^\tw(w)d\overline{\cT}+\alpha^2F_\frw^\tw(w)d\cO)(vv')}\\
& =\ \Pr\big( F_w^\tw(w)\cdot(d\cT + \alpha^2\,d\cO)(vv')\,,\,\alpha \R \big),
\end{align*}
which shows that $\rD[ {\rI_{\alpha\R}[F_\frw]}](w) = F_\frw^{{\tw}}(w)$ {according to the definition of the derivative~$\rD[H]$ for~$\alpha\R$-valued harmonic functions.}


Note that the proof {given above works up to} the boundary as long as the primitive $\rI_\C$ is well defined, {which is the case if~$F_\frw$ satisfies standard boundary conditions.}
The case of t-black-holomorphic functions is similar.
\end{proof}

\begin{remark}
Proposition~\ref{prop:representation_derivative} explains why, for a t-white-holomorphic function, its values on white vertices {have} a better behaviour than those on black vertices. Indeed, in the above representation the values $F_\frw^\tw$ encode the whole derivative of $H$
while $F_\frw^\tb$ only gives the derivative in a specific direction. Finally, if $H$ is regular, $F_\frw^\tw$ inherits its regularity while $F_\frw^\tb$ does not.
\end{remark}

\subsection{{T-holomorphic functions and reversed random walks on T-graphs}}\label{sub:harmonicity-1} This section is devoted to another link between t-holomorphic functions and T-graphs, which was not discussed in the earlier literature.
Namely, we show that projecting the values~$F_\frw^\tw$ (similarly,~$F_\frb^\tb$) onto a given direction, one obtains a harmonic function with respect to the \emph{reversed} random walk on an appropriate \mbox{T-graph}.

\begin{proposition}\label{prop:backward_harmonic}
Let $F_\frw$ be a t-white-holomorphic function. For {each} $\alpha$ in the unit circle, {the function} $\Pr(F_\frw^\tw, \alpha \R )$ is a martingale for the time reversal of the continuous time random walk on {the T-graph}
$\cT - \overline {\alpha^2 \cO}$ (with respect to the invariant measure {given in Proposition~\ref{prop:white_Tgraph}(vi)}).

Similarly, if $F_\frb$ is t-black-holomorphic, then $\Pr( F_\frb^\tb, {\overline{\alpha}} \R )$ is harmonic for the time reversal of the random walk on {the T-graph} $\cT - \alpha^2 \cO$. {Both claims hold true under proper identifications of the white (resp., black) faces of a t-embedding~$\cT$ with its vertices. Such an identification depends on~$\alpha$ and is described in Lemma~\ref{lem:bijection}.}
\end{proposition}

{\begin{remark}\label{rem:Im-on-deg-T}
Let~$v=(\cT\!-\!\overline{\alpha^2\cO})(b)$ be a degenerate face of the T-graph \mbox{$\cT\!-\!\overline{\alpha^2\cO}$}, note that this means~$\eta_b=\pm \alpha$. In Lemma~\ref{lem:bijection}, all the three white faces~$w_k$ adjacent to~$b$ are identified with~$v$. However, if~$F_\frw$ is t-holomorphic, the three values~$\Pr(F_\frw^\tw(w_k),\alpha\R)=F_\frw^\tb(b)$ match. Therefore, even in presence of degenerate faces, it makes sense to view the function~$\Pr(F_\frw^\tw;\alpha\R)$ as being defined on vertices of the T-graph~$\cT-\overline{\alpha^2\cO}$ via Lemma~\ref{lem:bijection}.
\end{remark}}

The proof of Proposition~\ref{prop:backward_harmonic} {goes} through a sequence of lemmas. Let us focus on the case of $\Im( F_\frw^{{\tw}} )$ for simplicity and without true loss of generality. We first assume that {the T-graph} $\cT\! -\! \overline {\alpha^2  \cO}{{}=\cT\!+\!\overline{\cO}}$ has no degenerate faces, which is equivalent to saying that $\Re \eta_{b} \neq 0$ for all $b$.

\begin{lemma}\label{lem:backward_identity}
Let $F_\frw$ be a t-white-holomorphic function.
Let $w_1, b_1, w_2, \ldots$ be faces adjacent to an interior vertex $v$ {of~$\cT$, listed counterclockwise,} and assume that $w_k \in W\smallsetminus\partial W$ for all $k$.
Then,
\[
\sum_k \Im(F_\frw^{{\tw}}(w_k))\cdot (\tan (\phi_{b_{k-1}}) - \tan(\phi_{b_k})) = 0,
\]
where we use a cyclical indexing of vertices.
\end{lemma}
\begin{proof}
{Due to the definition of t-white-holomorphic functions, the values} $F_\frw^{{\tw}}(w_k)$ and $F_\frw^{{\tw}}(w_{k+1})$ have the same projection on {the direction} $ \eta_{b_k}\R$. Therefore, for all~$k$, {we have the identity}
\[
\overline{\eta}_{b_k} \cdot F_\frw^{{\tw}}(w_k) +  \eta_{b_k}\cdot \overline {F_\frw^{{\tw}}(w_k) }= \overline{\eta}_{b_k} \cdot F_\frw^{{\tw}}(w_{k+1}) +  \eta_{b_k} \cdot\overline {F_\frw^{{\tw}}(w_{k+1})},
\]
which can be rewritten as
\[
\Re( F_\frw^{{\tw}}(w_k)) - \Re( F_\frw^{{\tw}}(w_{k+1})) + \frac{\Im (\eta_{b_k})}{\Re(\eta_{b_k} )}\cdot \big(  \Im(  F_\frw^{{\tw}}(w_k))  - \Im(  F_\frw^{{\tw}}(w_{k+1}))\big) = 0.
\]
Summing over $k$ and re-indexing, we obtain
\[
\sum_k \Im(  F_\frw^{{\tw}}(w_k))\cdot\left( \frac{\Im (\eta_{b_k})}{\Re(\eta_{b_k} )}  - \frac{\Im (\eta_{b_{k-1}})}{\Re(\eta_{b_{k-1}} )}\right) = 0,
\]
which is the desired statement written in terms of $\eta$.
\end{proof}
\begin{remark}One can also interpret the identity of Lemma~\ref{lem:backward_identity} geometrically: successive values of $F_\frw^\tw$ have prescribed projections on {the lines} $ \eta_{b_k}\R$ so they must form a closed polygonal chain with edges with directions $ i \eta_{b_k}$. The identity expresses {the fact} that this chain is closed.
\end{remark}

Using Lemma~\ref{lem:increment_phi}, it is easy to see that all the {coefficients} $(\tan(\phi_{b_{{k-1}}}) - \tan(\phi_{b_{k}}))$ are {positive} except for a single one. As a consequence, {for each vertex~$v$ of~$\G^*$ lying away from the boundary,} we can  {specify a white face} $w(v)=w_{k(v)}$ {corresponding to this negative coefficient,} and rewrite the equations as
\begin{equation}\label{eq:backward-harm}
\Im\big( F_\frw^\tw(w(v) )\big)\ = \sum_{w \neq w(v){:w\sim v}}  {\widetilde p}(w(v), w) \cdot\Im( F_\frw^\tw(w ))
\end{equation}
where
\begin{equation}\label{transit}
{\widetilde p}(w(v), w_k)\ :=\ \big({\tan(\phi_{b_{k-1}}) - \tan(\phi_{b_k})}\big)\big/\big({\tan(\phi_{b_{k(v)}}) - \tan(\phi_{b_{k(v)-1}})}\big)\,.
\end{equation}
{Note that $\widetilde p(w(v), w)$ are} positive and sum {up} to $1$. We want to see these as equations for a discrete harmonic function for a random walk with transition probabilities given {by~\eqref{transit}.}
Let us write explicitly that $w(v)=w_{k(v)}$, where~$k(v)$ is the unique index such that $\phi_{b_{k(v)-1}} <\phi_{b_{k(v)}}$ and that this condition is equivalent to {the equality} $\phi_{b_{k(v)}} - \phi_{b_{k(v) - 1}} =\pi-{\theta(v,w(v))}$; 
{see Lemma~\ref{lem:increment_phi}.}

\begin{lemma}\label{lem:bijection}
Suppose first that $\cT + \overline \cO$ has no degenerate faces. In the whole plane case, the map $v \mapsto w(v)$ {constructed above is a bijection.} {Its inverse can be described as follows: $ (\cT+\overline{\cO})(v)$ is the interior vertex of the segment~$(\cT+\overline{\cO})(w(v))$ in~$\cT+\overline{\cO}$ (more generally, in~$\cT-\overline{\alpha^2\cO}$).}

In the finite case, there exists a subset $V'$ of the set~$V$ of vertices of~$\G^*$ such that {$v \mapsto w(v):V'\to W\smallsetminus\partial W$ is a bijection.} Moreover, $V'$ differs from $V$ only at the boundary in the sense that {all vertices} in $V \smallsetminus  V'$ are adjacent to boundary faces. 

Finally, the {bijection} $ v \mapsto w(v)$ is well defined {on~$\G^*$} even if {the T-graph} $\cT + \overline \cO$ has degenerate faces: {in this case, each degenerate vertex}~$v=(\cT+\overline{\cO})(w)$ {corresponds to three vertices of~$\G^*$ and, further, to three white {faces} adjacent to $b$; cf. Remark~\ref{rq:degenerate_face}.}
\end{lemma}

\begin{proof}
Given {an inner white face} $w$ of~$\cT$, let $A$, $B$, $C$ be the adjacent vertices of~$\G^*$ {listed counterclockwise} and let $b_A$, $b_B$, $b_C$ be the black faces opposite to ${A, B, C}$, respectively. Let~$v(w)\in\G^*$ be the vertex mapped into the inner vertex of the segment~$(\cT+\overline{\cO})(w)$ of the T-graph. {Due to Proposition~\ref{prop:white_Tgraph}(iii),}~$A=v(w)$ if and only if $ -\pi/2<\phi_{b_C} < \phi_{b_A}<\phi_{b_B}<\pi/2$, which is also equivalent to {say that} ${\phi_{b_B} - \phi_{b_C}} = \pi- \theta{(A,w)}$.

Let us check that the composition $ v\mapsto w(v)\mapsto v(w(v))$ gives the identity {in the whole plane case}. {Given a vertex $v$}, let $b_{ k(v)-1}, b_{k(v)}$ be the two common black neighbours of $v$ and $w(v)= w_{k(v)}$.
By definition {of the mapping $v\mapsto w(v)$,} we have $\phi_{b_{k(v)}} - \phi_{b_{k(v)-1}}=\pi- \theta(v,w(v))$. Therefore, $v(w(v))=v$. The proof of~$w(v(w))=w$ is identical, thus~$v\mapsto w(v)$ is a bijection.

In the finite case, we just notice that {the mapping~$w\mapsto v(w)$ makes sense for interior white faces and denote by~$V'\subset V$ the image of~$W\smallsetminus\partial W$ under this mapping. If~$v$ is not adjacent to boundary faces, then the face~$w(v)$ is well-defined and one has $v=v(w(v))\in V'$ as above.}

Finally, {to define the inverse mapping $w\mapsto v(w)$ in presence of degenerate faces, one simply says that $v(w)=A$ if~$\phi_{b_B}-\phi_{b_c}=\pi-\theta(A,w)$ (and similarly for~$B$ and~$C$). Clearly, this remains a bijective correspondence if~$v$'s are considered as vertices of~$\G^*$ and not as those of~$\cT+\overline{\cO}$.}
\end{proof}

{
Assume for a moment that the T-graph~$\cT+\overline{\cO}$ has no degenerate faces. Given a vertex~$v$ (we assume that $v$ is not adjacent to boundary faces in the finite case), introduce the transition rates
\begin{equation}
\label{eq:q-tilde-def}
\widetilde{q}(w(v)\to w_k)\ :=\ \frac{\tan(\phi_{b_{k-1}}) - \tan(\phi_{b_k})}{8S_{w(v)}},\quad k\ne k(v)
\end{equation}
provided that $b_{k-1}$, $w_k$, $b_k$ are consecutive faces adjacent to $v$ (and set all other transition rates from~$w(v)$ to zero). Clearly, the equation~\eqref{eq:backward-harm} can be viewed as the harmonicity property with respect to the corresponding (continuous time) random walk~$\widetilde{X}_t$. Moreover, Lemma~\ref{lem:bijection} provides a bijective correspondence $v\leftrightarrow w(v)$, which allows to view this random walk as being defined on vertices of~$\cT+\overline{\cO}$. We are now in the position to prove the main result of this section.}

\begin{proof}[Proof of Proposition~\ref{prop:backward_harmonic}] We first consider the case when $\cT + \overline \cO$ has no degenerate faces.
We have seen above that $\Im( F_\frw^\tw)$ is harmonic for the walk $\widetilde X_t$. {Thus, it remains to check that} its transition rates agree with the time reversal of the walk on $\cT + \overline\cO$ {discussed in Proposition~\ref{prop:white_Tgraph}. Consider two white faces}~$w, w'$ of~$\cT$ sharing a vertex $v$. 
The transition rate~$\widetilde{q}(w\to w')$ is non-zero if and only if~$w=w(v)$. {In this situation,~$(\cT+\overline{\cO})(v)$ is an interior point of the segment~$(\cT+\overline{\cO})(w)$ and hence is an endpoint of~$(\cT+\overline{\cO})(w')$. Therefore, the forward random walk on~$\cT+\overline{\cO}$ also has a non-zero transition rate from~$v(w')$ to~$v$ if and only if~$w'\ne w(v)$. Moreover, it is easy to see from Proposition~\ref{prop:backward_harmonic}(iv) that
\[
q(v(w_k)\to v)\ =\ \frac{\tan(\phi_{b_{k-1}})-\tan(\phi_{b_k})}{8S_{w_k}},\quad k\ne k(v)
\]
if~$b_{k-1},w_k,b_{k+1}$ are faces adjacent to~$v$ listed counterclockwise. Since the invariant measure of the random walk is given by~$\mu(v)=S_{w(v)}$, this concludes the proof in the non-degenerate case.}



Finally, if $\cT + \overline \cO$ has degenerate faces, we still see that $\Pr( F_\frw^\tw, \alpha \R )$ is harmonic for the time reversal of the walk on $\cT - \alpha^2 \overline \cO$, for generic values~$ \alpha$. By continuity, we can take the limit $ \alpha\to i$ in order to  obtain the desired result for~$\cT+\overline{\cO}$, {see Remark~\ref{rq:degenerate_face} and Remark~\ref{rem:Im-on-deg-T}.}
\end{proof}

\section{{Generalization to faces of higher degree}}\label{sec:non_triangulation}

In this section we extend the framework of t-holomorphicity from triangulations to higher degree faces. This discussion also applies \emph{ad verbum} to boundary quads of a finite triangulation provided that {the} t-holomorphic functions in question satisfy standard boundary conditions.
Recall that the general definition of a t-embedding and of the origami map was given in Section~\ref{sec:setup} without any restriction on degrees of faces. The general idea is that the `proper' notion to extend is the kernel of $K$ and the link between t-holomorphic functions {on a t-embedding} and harmonic ones {on the corresponding} T-graphs. Compared to triangulations, the main missing point is the exact extension of functions~$F_\frw$ or $F_\frb$ from one bipartite class to the other (e.g., an extension of~$F_\frw^\tb$ to~$F_\frw^\tw$).

Below we define such an extension by splitting higher degree faces into triangles, similar to our treatment of the boundary of a finite triangulation discussed in Section~\ref{sub:t-graphs}; see also Fig.~\ref{boundary_t_graph}. {(A simplest example of this kind appears when discussing the link between the most standard discretization of the complex analysis on the square grid and the framework developed in this paper, we refer the interested reader to Section~\ref{subsub:appendix-Z2} for details.)} Though the exact values~$F_\frw^\tw$ on these new triangles depend on the choice of a splitting, this dependence is local: if one changes {the} splitting of a single face~$w$, only the values of~$F_\frw$ on {the} new faces obtained from~$w$ change. Moreover, {the} a priori regularity estimates discussed in Section~\ref{sec:regularity} (e.g., see Proposition~\ref{prop:Holder}) eventually imply that these values are actually almost independent of the way in which faces are split, at least for bounded t-holomorphic functions and at faces lying in the bulk of a t-embedding.

Recall that~$\eta$ and~$\cO$ are the origami square root function and the origami map associated to a t-embedding~$\cT$, and that $B$ and~$W$ are the sets of black and white faces of~$\cT$, respectively. The link between t-embeddings~$\cT$ and T-graphs~$\cT+\alpha^2\cO$ and~$\cT+\overline{\alpha^2\cO}$ remains exactly the same as in Proposition~\ref{prop:geomT} and Proposition~\ref{prop:white_Tgraph}, respectively, with the same proof. Namely,
\begin{itemize}
\item for each~$\alpha\in \mathbb{T}$, both~$\cT+\alpha^2\cO$ and~$\cT+\overline{\alpha^2\cO}$ are T-graphs with possibly degenerate faces; for a generic choice of~$\alpha$ there are no degenerate faces;
\item in the T-graph~$\cT+\alpha^2\cO$ the following holds:
\begin{samepage}
\begin{itemize}
\item for {each} $b \in B$, the image of $b$ is a translate of $2\Pr( \cT(b), {\alpha}\overline{\eta}_b \R)$;
\item for {each} ${w \in W}$, the image of $w$ is a translate of $(1\!+\!\alpha^2\eta_w^2)\cT(w)$;
\item {if a face~$(1+\alpha^2\eta_w^2)\cT(w)$ is degenerate (i.e., if $\alpha^2=-\overline{\eta}_w^2$), then the 'infinitesimal' polygon assigned to it is homothetic to~$\alpha\eta_w\cT(w)$;}
\end{itemize}
\end{samepage}
\item in the T-graph~$\cT+\overline{\alpha^2\cO}$ the following holds:
\begin{itemize}
\item for {each} $w\in W$, the image of $w$ is a translate of $2\Pr( \cT(w), {\overline{\alpha}}\overline{\eta}_w \R)$;
\item for {each} ${b\in B}$, the image of~$b$ is a translate of $(1 + {\overline{\alpha}^2}\eta_b^2 ) \cT(b)$;
\item {if a face~$(1+\overline{\alpha}^2\eta_b^2)\cT(w)$ is degenerate (i.e., if $\alpha^2=-\eta_b^2$), then the 'infinitesimal' polygon assigned to it is homothetic to~$\overline{\alpha}\eta_b\cT(w)$.}
\end{itemize}
\end{itemize}

To simplify the discussion, in the rest of {this} section we focus  on the T-graphs $\cT+\cO$ and $\cT + \overline \cO$ and assume that both contain no degenerate faces. Moreover, we also assume that no pair of distinct vertices of~$\cT$ is mapped onto the same vertex of~$\cT+\cO$ or~$\cT+\overline{\cO}$ (beyond triangulations, this might happen even in absence of degenerate faces in the T-graph {if, e.g., two vertices of~$\cT(b)$ are projected to the same point of the segment~$(\cT+\cO)(w)$ from opposite sides}). As in {Section~\ref{sec:Tgraph}, these} non-degeneracy assumptions can be dropped using continuity arguments with respect to~$\alpha$, thus all the statements readily extend to the general case.

Note that the definition of harmonic functions on T-graphs still makes sense: one says that a function~$H$, defined on vertices of, e.g.,~$\cT+\cO$, is harmonic if it admits a linear continuation onto each edge, the only difference is that these edges can now contain more than one interior point. In particular, the derivative~$\rD[H]$ of a harmonic function {on~$\cT+\cO$} is still well-defined {on black faces of~$\cT$.} As already mentioned above, we keep {this} link with harmonic functions on T-graphs (see Proposition~\ref{prop:representation_derivative}) as the definition of a t-holomorphic function. {Recall that we denote by~$B_U$ the subset of~$B$ lying in a region~$U\subset\C$ and similarly for other sets.}

\begin{definition} \label{def:thol-gen}
Let~$U\subset\C$ be a subregion of a t-embedding~$\cT$. We say that a function $F_\frw^\tb$ defined on~${B_U}$ is t-white-holomorphic if there exists a real-valued harmonic function~$H$ defined on the corresponding subset of the T-graph~$\cT+\cO$ such that~$F_\frw^\tb=\rD[H]$. If the region~$U$ is not simply connected, we require that such a harmonic primitive exists on all simply connected {subregions} of~$U$.

Similarly, a function $F_\frb^\tw$ defined on ${W_U}$ is t-black-holomorphic if it can be locally viewed as the derivative of a real-valued harmonic function defined on the corresponding subset of~$\cT + \overline \cO$.
\end{definition}

{As in} the case of triangulation, this definition can be reformulated in a more invariant way via the Kasteleyn matrix~$K$ and the contour integration. For instance,~$F_\frw^\tb$ is t-white-holomorphic if and only if
\[
\begin{cases}
F_\frw^\tb(\bpr)\in \eta_b\R & \text{for all $b\in{B_U}$,}\\
\oint_{\partial w} F^\tb_\frw\,d\cT = 0 & \text{for all~$w\in{W_U}$;}
\end{cases}
\]
the latter condition is equivalent to say that~$(F_\frw^\tb K)(w)=0$. This equivalence also shows that considering derivatives of $\alpha\R$-valued harmonic functions on the T-graph $\cT+\alpha^2\cO$ one gets the same notion of t-holomorphic functions on~$\cT$. The standard boundary conditions are defined exactly {in} the same {way} as for triangulations.

\def\spl{\mathrm{spl}}

We now move on to {defining} the `true' values~$F_\frw^\tw$ of a t-white-holomorphic function out of its values~$F_\frw^\tb$. Recall that all faces of a t-embedding are convex due to the angle condition. {If the degree of~$w\in W$ is $n>3$, then the condition~$(F_\frw^\tb K)(w)=0$, which can be written as a linear equation on $n$ real variables, no longer guarantees the existence of a complex value~$F_\frw^\tw(w)$ that has prescribed projections on the directions~$\eta_b\R$ for all~$b\in B$ surrounding~$w$. This motivates the following definition.}
\begin{definition}\label{def:split}
Given a t-embedding $\cT$, we say that~$\cT_\spl^\tw$ is a white splitting of~$\cT$ if $\cT_\spl^\tw$ is obtained from~$\cT$ by adding diagonal segments in all its white faces of degree at least {four} so that they are decomposed into triangles; see Fig.~\ref{split}. With a slight abuse of the terminology we still view $\cT_\spl^\tw$ as a \mbox{t-embedding} by interpreting each added segment as a black {bigon} with zero angles, note that this does not break the angle condition. Let $\G_\spl^\tw=B_\spl^\tw\cup W_\spl^\tw$ be the associated dimer graph; {note that~$B_\spl^\tw$ contains not only all black faces of~$\cT$ but also the newly constructed bigons.}
\end{definition}

\begin{figure}
\begin{center}
\includegraphics[width =\textwidth]{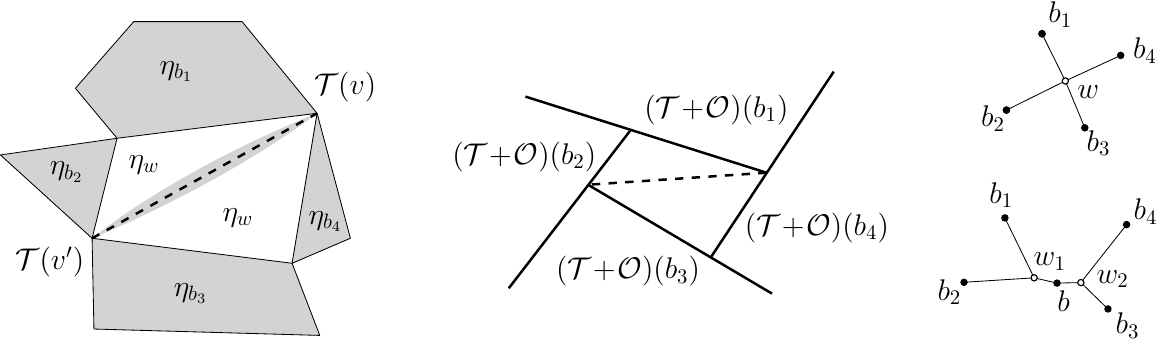}
\caption{A schematic representation of the effect of splitting a face, left in the t-embedding, middle in the T-graph and right in the combinatorial graph. The grey region around the added edge in the left picture represents the fact that we consider this edge as a black {bigon} in the t-embedding~$\cT_\spl^\tw$. {We~set $\eta_{w_1}\!=\!\eta_{w_2}\!:=\!\eta_w$ and} {also assign a value~$\eta_b$ to the new bigon~$b$ so that~\eqref{eq:def-eta} holds for~$\cT_\spl^\tw$.}}\label{split}
\end{center}
\end{figure}

{We mention several simple properties of this construction in the next lemma.}

\begin{lemma}\label{lem:split-triv} The origami square root function~$\eta$ extends {from~$\G$ to~$\G_\spl^\tw$} so that it keeps its {values} on all original black faces, {coherently assigns values to black bigons,} and inherits its value~$\eta_w$ on all white faces obtained from~$w$. In particular, {the} t-embeddings~$\cT$ and~$\cT_\spl^\tw$ define the same origami map~$\cO$.

The T-graph $(\cT+\cO)_\spl^\tw$ is obtained from $\cT + \cO$ by splitting all {its} white faces with the same diagonals as in the definition of $\cT_\spl$. Each harmonic function {defined on vertices and edges of} $\cT+\cO$ extends to {new edges of} $(\cT+\cO)_\spl^\tw$ in a unique way.
\end{lemma}
\begin{proof}
For the definition of the origami square root function on~$\cT_\spl^\tw$, note that, when propagating its value {through} the different pieces of a white face of $\cT$, we reflect two times on each added diagonal. Therefore, these values have to be equal on all such pieces and hence the origami square root function on~$\cT_\spl^\tw$ agrees with the original function~$\eta$ everywhere.

For the fact that $(\cT+\cO)_\spl^\tw$ is obtained from $\cT+\cO$ by splitting faces (see Fig.~\ref{split}), the proof is simply to write explicitly the restriction of $(\cT\!+\!\cO)_\spl^\tw$ to a white face of~$\cT$. 
Finally, note that the extra edges that we add when splitting white faces have no interior vertices, therefore the two notions of harmonic functions on $\cT + \cO$ and on $(\cT+\cO)_\spl^\tw$ are tautologically the same.
\end{proof}

With a slight abuse of the notation we still denote the origami square root function on~$\cT_\spl^\tw$ by~$\eta$. Thanks to the correspondence between harmonic functions on {(vertices and edges of)}~$\cT+\cO$ and on~$(\cT+\cO)_\spl^\tw$, we can now reformulate the t-holomorphicity condition on~$\cT$ similarly to the case of triangulations. {Geometrically, this can be seen as considering complex-values gradients of affine extensions of a harmonic function~$H$ inside faces of~$\cT+\cO$, which is not necessarily possible for faces of high degree and thus requires a choice of a triangulation.}

\begin{proposition} \label{prop:spl-t-hol}
For each white splitting~$\cT_\spl^\tw$ of~$\cT$, a function~$F_\frw^\tb$ defined on a {set~$B_U$} is t-white-holomorphic {in the sense of Definition~\ref{def:thol-gen}} if and only if it can be extended to faces of~$\cT_\spl^\tw$ so that
\begin{equation}
\label{eq:spl-t-hol}
F_\frw^\tb(b)\ =\ \Pr(F_\frw^\tw(w),\eta_b\R)\quad \text{if}~b\sim w,\ \ b\in {(B_\spl^\tw)_U,\ \ w\in (W_\spl^\tw)_U.}
\end{equation}
(In particular, note that we also extend~$F_\frw^\tb$ from~$B$ to {the bigger set}~$B_\spl^\tw$.) The t-black-holomorphicity property admits a similar reformulation via black splittings $\cT_\spl^\tb$ of~$\cT$, defined similarly to Definition~\ref{def:split}.
\end{proposition}
\begin{proof}
Let $F_\frw^\tb$ be t-white-holomorphic, and let $H$ be its primitive {defined on vertices and edges of} $\cT+\cO$. By Lemma~\ref{lem:split-triv}, the function $H$ can also be seen as a harmonic function on {vertices and edges of} $(\cT+\cO)_\spl^\tw$ and {hence} its derivative can be defined on~$W_\spl^\tw$ and satisfies~\eqref{eq:spl-t-hol} because {$W_\spl^\tw$ has triangular faces.} Conversely, if $F_\frw$ satisfies~\eqref{eq:spl-t-hol}, then it admits a real harmonic primitive defined on {vertices and edges of} $(\cT+\cO)_\spl^\tw$, therefore {its} restriction to ${B_U}$ is t-white-holomorphic.
\end{proof}

We now move to forward and backward random walks on T-graphs $\cT+\cO$ and $\cT+\overline{\cO}$ obtained from general t-embeddings. Recall that we assume that these graphs do not have degenerate faces {and, moreover, that all vertices of~$\cT$ have distinct positions in these T-graphs. (A general case follows by considering, e.g, T-graphs~$\cT+\alpha^2\cO$ with~$\alpha\to 1$ and using the continuity in~$\alpha$.)} The only difference with the case of triangulations is that the edges of T-graphs can now contain more than one interior vertex.

\begin{definition}
On a T-graph, we say that a continuous time Markov chain~$X_t$ is a version of the (continuous time) random walk on this T-graph if its jump rates~$q(v\to v')$ satisfy the following property. For any segment $L$ of the T-graph and any interior vertex $v$ lying on $L$, there exist exactly two vertices $v^-$ and $v^+$ in the closure of $L$ and on different sides from $v$ such that
\[
q( v \to v^\pm)\ =\ (|v^\pm-v|\cdot|v^+-v^-|)^{-1}
\]
and $q(v\to v')=0$ for all other transitions.
\end{definition}

Note that all versions~$X_t$ of the random walk on a T-graph are martingales and we also normalize the jump rates so that the process~$|X_t|^2-t$ is a martingale too. The standard random walk is obtained when one always chooses~$v^\pm$ to be the endpoints of~$L$. Introducing versions {of this random walk means} that we also allow transitions between interior vertices on the same segment. Since the chain is still a martingale, this difference essentially  amounts to a time change only. Nevertheless, it is crucial for a geometric interpretation of the invariant measure as discussed below. In fact, we already used a version of the standard walk when discussing the boundary of a finite triangulations in Section~\ref{sub:t-graphs}, see Fig.~\ref{boundary_t_graph}.

It is easy to see that each splitting $\cT_\spl^\tb$ of the black faces of~$\cT$ naturally defines a version of the random walk on $\cT + \cO$ by using the recipe described in Lemma~\ref{lem:proj}. In particular, the correspondence~$v\mapsto b(v)\in B_\spl^\tb$ can be defined in exactly the same way by inspecting the increments of the nearby values~$\phi_{w_j}$, including those assigned to~$w_j\in W_\spl^\tb\smallsetminus W$. Geometrically, this can be seen as interpreting each segment of $\cT + \cO$ (which may have several interior points) as the superimposition of several segments of $(\cT + \cO)_\spl^\tb$ (each with only one interior point), let us repeat that this procedure was already sketched in Section~\ref{sub:t-graphs} to treat the boundary faces.

As in the triangulation case, we can define a measure on vertices of the \mbox{T-graph} $\cT+\cO$ by setting~$\mu_\spl^\tb:=S_{v(b)}$, in the notation we emphasize the fact that this measure \emph{depends} on the splitting since~$b(v)$ is a face of~$\cT_\spl^\tb$ and not of~$\cT$ itself. (Recall also that possible degeneracies in the T-graph can be treated by working with~$\cT+\alpha^2\cO$ instead and passing to the limit~$\alpha\to 1$.)
The next lemma clarifies the advantage of considering versions of the random walk on T-graphs.
\begin{lemma}\label{lem:inv-split}
In the whole plane case, $\mu_\spl^\tb$ is an invariant measure for the version of the random walk on~$\cT+\cO$ associated to the splitting $\cT_\spl^\tb$. In the finite case this measure, considered on inner vertices of~$\cT+\cO$, is subinvariant provided that boundary vertices act as sinks.
\end{lemma}
\begin{proof}
This follows from exactly the same computation as in the proof of Corollary~\ref{cor:invariant_measure} {as} in this proof we never used the fact that white faces are triangles but only local relations around vertices of $\G^*$ and the geometry of black faces. For interior vertices adjacent to boundary faces, the mapping~$v\mapsto b(v)$ is still well-defined, we have the same outgoing rate but the incoming rate might be smaller; see Fig.~\ref{boundary_t_graph}.
\end{proof}

Note that the same measure is invariant for each of the random walks on \mbox{T-graphs} $\cT+\alpha^2\cO$ provided that we use a version corresponding to a fixed splitting of black faces. Clearly, similar considerations apply to (versions of) random walks on \mbox{T-graphs} $\cT+\overline{\alpha^2\cO}$ associated with a splitting~$\cT_\spl^\tw$ of white faces. Let~$\mu_\spl^\tw$ be the corresponding invariant measure and~$\widetilde{X}_t$ be the reversed (with respect to this measure) random walk on~$\cT+\overline{\cO}$. We now claim that the key Proposition~\ref{prop:backward_harmonic} also generalizes to faces of higher degrees in a straightforward way.

\begin{proposition} Let~$\cT_\spl^\tw$ be a white splitting of~$\cT$ and $F_\frw$ be a t-white-holomorphic function, whose values~$F_\frw^\tw$ on white faces of~$\cT_\spl^\tw$ are defined according to~\eqref{eq:spl-t-hol}. Then, {the function} $\Im F_\frw^\tw$ is harmonic for the time reversal (with respect to~$\mu_\spl^\tw$) of the {random} walk on $\cT_\spl+\overline{\cO}$, under the identification~$v\mapsto w(v)\in W_\spl^\tw$.

Similarly, for each~$\alpha\in\mathbb{T}$, the function~$\Pr(F_\frw^\tw ,\alpha\R)$ is harmonic with respect to the time reversal {(with respect to~$\mu_\spl^\tw$)} of the random walk on~$\cT+\overline{\alpha^2\cO}$.
\end{proposition}

\begin{proof} Again, the proof repeats the proof of Proposition~\ref{prop:backward_harmonic}, almost word by word. Indeed, in the proof of Lemma~\ref{lem:backward_identity} we never used the fact that black faces are triangles. This allows us to define a Markov chain $\widetilde{X}_t$ on vertices of $\cT + \overline{\cO}$ such that the function $\Im( F_\frw^\tw )$ is harmonic for that chain. Furthermore, the definition of the correspondence~$v\mapsto w(v)$ via the increments of~$\phi_{b_k}$, $b_k\in B_\spl^\tw$ remains the same. This ensures that~$\widetilde{X}_t$ is a time reversal of some version of the random walk on~$\cT+\overline{\cO}$. Finally, the algebraic part of computations of the transition rates also does not rely upon the shape of black faces and thus still holds word by word.
\end{proof}

\begin{remark}
Formally, the choice of a specific version of the forward random walk affects in a non-trivial way the law of its time reversal. However one can check that the law of the discrete time sequence of {edges of~$\cT+\overline{\cO}$} used by the forward or backward walk does not depend on which version we choose, at least among versions associated to different splittings. This reflects the fact that changing the splitting of a single white face one does not change values of~$F_\frw^\tw$ at other white faces {of $\cT$}.
\end{remark}

We conclude this section by formulating the modifications required in assumption~{\ExpFat{\delta}} in the case of higher degrees of faces. Recall that we call a triangle~`$\rho$-fat' if it contains a disc of radius~$\rho$; we use the same terminology for faces {of degree~$n\ge 4$. We also introduce an artificial extension of this notion to the case~$n=2$: a bigon $b\in B^\tw_\spl$ is called~$\rho$-fat if all the edges on at least one of the corresponding boundary arcs of the white face~$w\in W$ containing~$b$ belong to~$\rho$-fat black faces. (The notion of $\rho$-fat bigons~$w\in W^\tb_\spl$ is defined similarly.)}

\begin{assumption}[\ExpFat{\delta}, general case] \label{assump:ExpFat-general}
We say that t-embeddings~$\cT^\delta$ satisfy {the assumption}~\ExpFat{\delta} {(or, more accurately, \ExpFatPrime{\delta}{\delta'})} as~$\delta\to 0$ on a common region~${U\subset\C}$ (or, more generally, on regions~$U^\delta$ {depending on~$\delta$}) if~there exist {auxiliary scales~$\delta'=\delta'(\delta)$ and} splittings~$\cT_\spl^{\tw,\delta}$,~$\cT_\spl^{\tb,\delta}$ of~$\cT$ such that {$\delta'\to 0$ as $\delta\to 0$ and the following is fulfilled:}

\begin{itemize}
\item if one removes all `${\delta\exp(-\delta'\delta^{-1})}$-fat' black faces {(including bigons as defined above)} and all `${\delta\exp(-\delta'\delta^{-1})}$-fat' white triangles from {$\cT_\spl^{\tw,\delta}$, then each of its remaining vertex-connected components has diameter at most~$\delta'$,}
\end{itemize}
{and if a similar condition for~$\cT_\spl^{\tb,\delta}$ holds.}
\end{assumption}
{Though the notion of `$\rho$-fat bigons' used in Assumption~\ref{assump:ExpFat-general} does not look very natural, it turns out to be useful in the context of a priori Harnack estimates of t-holomorphic functions via their harmonic primitives on T-graphs. We refer the reader to the proof of Corollary~\ref{cor:Lipschitz} given below for the motivation of this definition.}

\section{{A priori regularity theory for t-holomorphic functions on t-embeddings and for harmonic functions on associated T-graphs}}\label{sec:regularity}
This section is devoted to a priori regularity properties of {t-holomorphic functions on t-embeddings and of harmonic functions on T-graphs obtained thereof.} Let us emphasize that in what follows we do \emph{not} rely upon any type of `uniformly bounded angles' or `uniformly fat faces' assumptions. In particular, the crucial uniform ellipticity estimate for random walks on T-graphs (see Section~\ref{sub:ellipticity}) is fully independent of the microscopic (below the scale~$\delta$) structure of {the} t-embeddings~$\cT^\delta$ provided that the corresponding origami maps~$\cO^\delta$ satisfy the `Lipschitzness at large scales' assumption~{\LipKd{\kappa}{\delta}} with~$\kappa<1$.

We then discuss corollaries of this estimate in Sections~\ref{sub:crossing} and~\ref{sub:sub-limits}, notably the a priori H\"older-type regularity for both harmonic and t-holomorphic functions and its consequences for subsequential scaling limits. In Section~\ref{sub:Lipschitz} we go even further and prove the a priori Lipschitz-type regularity for harmonic functions on T-graphs under a mild additional assumption~\ExpFat{\delta}.
Finally, Section~\ref{sub:Fpmpm=O(1)} also relies upon additional assumption~{\ExpFat{\delta}} and contains a technical result which we later use in Section~\ref{sec:convergence}.

\subsection{Preliminaries} We begin this section with a usual estimate for large deviations of a martingale process with bounded increments. Then we discuss simple distortion estimates for \mbox{`$\kappa$-Lipschitz} at large scale', $\kappa<1$, perturbations of the identity mapping in the complex plane. {When speaking about random walks on T-graphs obtained from general t-embeddings we always assume that an appropriate splitting (either black or white) is made as discussed in Section~\ref{sec:non_triangulation} and that~$X_t$ denotes a version of the random walk corresponding to this splitting.}

\begin{proposition} \label{prop:concentration} For all~$t,\lambda>0$, the following estimate is fulfilled:
\[
\P\big(\sup\nolimits_{s\in [0,t]}|X_s-X_0|\ge 2\lambda\sqrt{t}\,\big)\ \le\ 4\exp\big(-\tfrac{1}{2}\lambda^2\cdot(1+\tfrac{2}{3}\delta\lambda t^{-1/2}\,)^{-1}\big)\,.
\]
{In particular, the left-hand side is exponentially small in~$\lambda$ uniformly over all~$t\ge\delta^2$.}
\end{proposition}
\begin{proof} Let~$Y_t:=\Re(X_t)$, note that the process~$Y_t$ inherits the martingale property of~$X_t$. Since the jumps of~$X_t$ are bounded by~$2\delta$, so do those of~$Y_t$. Therefore, one can apply (a continuous time version of) Bennett's inequality, which says that
\[
\P\big(\sup\nolimits_{s\in [0,t]}(Y_s-Y_0)\ge a\,\big)\ \le\ \exp\biggl(-\frac{\var(Y_t)}{4\delta^2}H\biggl(\frac{2\delta a}{\var(Y_t)}\biggr)\biggr)
\]
for all~$a,t>0$, where~$H(x):=(1+x)\log(1+x)-x\ge \frac{1}{2}x^2\cdot (1+\frac{1}{3}x)^{-1}$ for~$x\ge 0$. Since the function~$\sigma^2H(x/\sigma^2)$ is decreasing in~$\sigma$, and~$\var(\Re X_t)\le \mathrm{Tr}\var (X_t)=t$, one gets
\begin{align*}
\P\big(\sup\nolimits_{s\in [0,t]}\Re(X_s-X_0)\ge \lambda\sqrt{t}\,\big)\ &\le\ \exp\big(-\tfrac{1}{4}t\delta^{-2}H(2\delta\lambda t^{-1/2})\big)\\
&\le\ \exp\big(-\tfrac{1}{2}\lambda^2\cdot(1+\tfrac{2}{3}\delta\lambda t^{-1/2}\,)^{-1}\big),
\end{align*}
a version of the Bernstein inequality; and similarly for $-\Re(X_s-X_0)$ and {for} $\pm\Im(X_s-X_0)$. We conclude the proof by saying that, for~$\sup_{s\in [0,t]}|X_s-X_0|$ to be greater than~$2\lambda$, at least one among these quantities must be greater than~$\lambda$.
\end{proof}

\begin{corollary}\label{cor:inv-concentration-triv}
There exist {$n_0\in\mathbb{N}$} such that the following estimate holds:
\[
\P\big({\sup\nolimits_{s\in [0,n_0t]}|X_s}-X_0|\ge {\tfrac14}\sqrt{t}\,\big)\;\ge\; {\tfrac34}\quad\text{for~all}\ \ t\ge \delta^2.
\]
\end{corollary}
\begin{proof} This is a straightforward corollary of the tail estimate {given above and of the Markov property.
Recall that we have $\E|X_t-X_0|^2=t$ for each~$t\ge 0$. Let us now find a (big) constant~$C_0>0$ such that
\[
\E(|X_t-X_0|^2 \mathbbm{1}_{|X_t-X_0|\le C_0\sqrt{t}})\ \ge\ \tfrac{1}{2}t
\]
and hence
\[
\P(|X_t-X_0|\ge \tfrac12\sqrt{t})\ \ge\ \tfrac14C_0^{-2}
\]
for all $t\ge\delta^2$. It is now enough to choose big enough~$n_0$ so that~$(1-\tfrac14C_0^{-2})^{n_0}\le \tfrac 14$ and apply this estimate~$n_0$ times using the Markov property of~$X_t$.}
\end{proof}

\def\cF{\mathcal{F}}

{We now discuss distortion properties of the correspondence between t-embeddings and T-graphs under assumption~\LipKd{\kappa}{\delta}. Let~$\alpha\in\mathbb{T}$ and~$\cF$ be one of the mappings \mbox{$z\mapsto (\cT+\alpha^2\cO)(z)$} or \mbox{$z\mapsto (\cT+\overline{\alpha^2\cO})(z)$}. Note that~$\cF$ is `almost a homeomorphism': it can be viewed as a limit of bi-{Lipschitz} mappings defined similarly with~$|\alpha|<1$. It is easy to see that assumption~{\LipKd{\kappa}{\delta}} implies the inclusions
\begin{equation}
\label{eq:T+0-discs}
B(\cF(z),(1-\kappa)r)\subset \cF(B(z,r))\subset B(\cF(z),(1+{\kappa})r)\ \ \text{provided that~$r\ge\delta$}.
\end{equation}
Indeed, the upper bound is trivial while the lower one follows from the fact that the image of the boundary~$\cF(\partial B(z,r))$ remains at distance at least~$(1-{\kappa})r$ from~$\cF(z)$ and that this curve necessarily encircles~$\cF(z)$ due to the ``dog on a leash'' lemma or Rouch\'e's theorem similarly to the proof of Proposition~\ref{prop:geomT}.}
\begin{lemma}\label{lem:area}
There exist constants $q_0=q_0(\kappa)\ge 1$ and $c_1(\kappa),c_2(\kappa)>0$ such that the following estimates hold for all T-graphs~$\cT+\alpha^2\cO$ and~$\cT+\overline{\alpha^2\cO}$, $\alpha\in\mathbb{T}$, obtained from t-embeddings satisfying assumption~{\LipKd{\kappa}{\delta}} and for all~$\beta\in\mathbb{T}$:
\[
c_1(\kappa)\cdot {\area(Q)}\ \le\ \sum_{v\in Q}|\Re(\,\overline{\beta}\eta_{b(v)}\,)|^2S_{b(v)}\ \le\ \sum_{v\in Q}S_{b(v)}\ \le\ c_2(\kappa)\cdot {\area(Q)}
\]
for each square~$Q$ of size~$(q_0\delta)\times(q_0\delta)$ drawn over the T-graph. {Let us also denote}
\begin{equation}
\label{eq:c0-def}
{c_0(\kappa)\ :=\ c_1(\kappa)/c_2(\kappa)\,.}
\end{equation}
\end{lemma}
\begin{proof} Let~$Q_\cT$ be the preimage of~$Q$ on the t-embedding and $V(Q_\cT)$ denote the set of vertices {of~$\cT$ lying in~$Q_\cT$.} The upper bound~$c_2(\kappa)\delta^2$ follows from the fact that {$Q_\cT$ is contained in a disc of radius~$(1-\kappa)^{-1}\cdot 2^{-1/2}q_0\delta$ due to~\eqref{eq:T+0-discs} and that all faces~$b$ have diameter less than~$\delta$.} To verify the lower bound~$c_1(\kappa)\delta^2$, note that
\[
\sum_{v\in Q}|\Re(\,\overline{\beta}\eta_{b(v)}\,)|^2S_{b(v)}\ =\ \frac{1}{4}\sum_{v\in Q}\area((\cT+\overline{\beta^2\cO})(b(v))).
\]
{It follows from~\eqref{eq:T+0-discs} that the image~$(\cT+\overline{\beta^2\cO})(Q_\cT)$ contains a disc of radius~$3\delta$ provided that~$q_0$ is chosen big enough. Therefore, the union of images of black faces~$b(v)$ with~$v\in Q$ contains a disc of radius~$\delta$, which implies the result.}
\end{proof}

\def\osc{\operatorname{osc}}

\subsection{Variance estimate for the random walks on t-graphs {under assumption~{\LipKd{\kappa}{\delta}}}} \label{sub:ellipticity}
In this section we prove the key ellipticity estimate for the continuous {time} random walks~$X_t$ on \mbox{T-graphs} {associated to} t-embeddings; see Definitions~\ref{def:chain},~\ref{def:chain-deg}. 
The estimates given below are fully independent of the microscopic (below the scale~$\delta$) structure of t-embeddings~$\cT^\delta$ provided that the corresponding origami maps~$\cO^\delta$ satisfy the `Lipschitzness at large' assumption~{\LipKd{\kappa}{\delta}} with~$\kappa<1$. Throughout this section the scale~$\delta$ is fixed, thus we write~$\cT=\cT^\delta$ and~$\cO=\cO^\delta$ for shortness. {In what follows,} all constants notated like~$t_0,\sigma_0,{q_0,c_0}$ etc {can (and actually do)} depend on~$\kappa$ but \emph{not} on {the t-embedding~$\cT^\delta$ or on~$\delta$.}

\begin{proposition} \label{prop:ellipticity} There exist constants~$t_0=t_0(\kappa)>0$ and~$\sigma_0=\sigma_0(\kappa)>0$ such that, for each {t-embedding}~$\cT$ satisfying assumption~{\LipKd{\kappa}{\delta}}, each~$\beta\in\mathbb{T}$, and each starting point~$X_0$, the following estimate holds for the continuous time random walk~$X_t$ on~$\cT+\cO$:
\begin{equation}
\label{eq:ellipticity}
\var (\Re(\,\overline{\beta}(X_{t_0\delta^2}-X_0)))\ge \sigma_0^2 \delta^2.
\end{equation}
{Due to the Markov property,~\eqref{eq:ellipticity} also implies that~$\var(\Re(\overline{\beta}(X_t-X_0)))\ge \tfrac{1}{2}\sigma_0^2 t$ for all~$t\ge t_0\delta^2$.}
\end{proposition}
\begin{proof} Without loss of generality, we can assume that~$X_0=0$,~$\delta=1$ and $\beta=i$, i.e., we aim to prove that~$\var(\Im X_{t_0}^{{0}})\ge \sigma_0$. The proof goes by contradiction and relies upon two lemmas given below. Eventually, we will set (see Lemma~\ref{lem:D-exit}(ii) for the motivation of this choice)
\begin{equation}\label{eq:x-st0-def}
\sigma_0^2:={\tfrac{1}{8}q_0^2L^{-1}},\qquad t_0:={16n_0q_0^2L^4+K_0L}
\end{equation}
for a large enough~$L\in 2\mathbb{N}$, where~${n_0\in\mathbb{N}}$ is fixed in Corollary~\ref{cor:inv-concentration-triv}, $q_0=q_0(\kappa)\ge 1$ is fixed in Lemma~\ref{lem:area}, {and a (big)} constant $K_0={K_0(\kappa)}$ will be chosen later.

{Denote}
\begin{equation}
\label{eq:rect-D-def}
D:=[-q_0L^2,q_0L^2]\times [-q_0,q_0(L-1)]
\end{equation}
and $\nu(\cdot):=\mu(\cdot)/\mu(D)$ be the normalized invariant measure of the random walk on~$\cT+\cO$, restricted to vertices of the T-graph lying \emph{inside~$D$}. Denote by~$X_t^\nu$ the random walk started at the measure~$\nu$ and \emph{stopped} when leaving~$D$. Let~$\nu_t$ be the law of~$X_t^\nu$, note that~$\nu_t(v)<\nu_0(v)$ for all~$v\in D$ because of the contribution of those trajectories {that} exit~$D$ and a lack of those who enter~$D$ from outside. Of course,~$\nu_t$ remains a probability measure: the remaining mass is concentrated at the boundary of~$D$.

{Let us first informally explain the intuition behind the proof given below. The assumption \mbox{$\var(\Im X_{t_0}^0)<\sigma_0^2$} for big~$t_0$ and small~$\sigma_0$ means that the random walk started at the origin propagates almost only in the horizontal direction. This implies the existence of a `bottom-screening' path~$\Gamma$ (see the proof of Lemma~\ref{lem:D-exit}(ii)) that crosses the rectangle~$D$ horizontally near its bottom side and has the property that for \emph{each} vertex~$v\in\Gamma$ the probability that the random walk started at~$v$ exits from~$D$ through the bottom side is small. If we now start the random walk at the invariant measure~$\nu$, then the existence of such a `bottom screening' path implies that the particles cannot exit from~$D$ through the bottom side as, for topological reasons, they should cross~$\Gamma$ before doing that. In this scenario, the martingale property of~$\Im(X_t)$ implies that \emph{all} particles in~$D$ move almost only in the horizontal direction, which is not hard to rule out using the assumption~\LipKd{\kappa}{\delta} and the geometric interpretation of the invariant measure~$\nu$; see Lemma~\ref{lem:D-variance}.

\smallskip

Recall that the constant \mbox{$c_0=c_0(\kappa)>0$} is given by~\eqref{eq:c0-def}.}

\begin{lemma} \label{lem:D-variance} For each~{$K\ge 1$, the estimate}
\[
\var (\Im (X^\nu_{KL}-X^\nu_0))\ \ge\ {\tfrac12c_0\cdot K}L
\]
{holds} for all sufficiently (depending on~$\kappa$ and~${K}$) large $L$.
\end{lemma}
\begin{proof}
Since the edge~$(\cT+\cO)(b(v))$ of the T-graph corresponding to a vertex~$v$ goes in the direction~$\overline{\eta}_b$, we see that
\[
\var(\Im(X^\nu_{KL}-X^\nu_0))\ =\ \int_0^{KL}\sum_{v\in D}|\Im\eta_{b(v)}|^2\nu_t(v)dt.
\]
(Note that {this expression holds without any restriction on the degree of faces of a t-embedding provided that~$X_t$ is a version of the random walk on~$\cT+\cO$ corresponding to a black splitting. Moreover, the same expression} in presence of degenerate vertices follows, e.g., from continuity arguments.) {Recall (see Section~\ref{sub:t-graphs}) that the invariant measure~$\nu=\nu_0$ is, up to the multiplicative normalization, nothing but the area of black faces of~$\cT$. Therefore,} it is easy to see from Lemma~\ref{lem:area} and Proposition~\ref{prop:concentration} that
\begin{align*}
\textstyle \sum_{v\in D}|\Im\eta_{b(v)}|^2\nu_t(v)\ &\textstyle \ge\ \sum_{v\in D}|\Im\eta_{b(v)}|^2\nu_0(v)-(\nu_0(D)-\nu_t(D))\\
&\textstyle \ge {c_0\cdot\sum_{v\in D}\nu_0(v)}-O((K/L)^{1/2})\ {=\ c_0-O((K/L)^{1/2})}
\end{align*}
for all~$t\le KL$. {Hence,} $\var(\Im(X^\nu_{KL}-X^\nu_0)) \ge {KL\cdot(c_0-O((K/L)^{1/2}))}$,
which {is} greater than~${\tfrac12c_0\cdot KL}$ for large enough~$L$.
\end{proof}

The next lemma provides a bound for the probability that the random walk~$X^\nu_t$ exits from the rectangle $D$ before time~$K_0L$ through (i) the vertical or (ii) the bottom side; {the latter is the key ingredient of the proof of Proposition~\ref{prop:ellipticity}.} Recall that, by our convention, we stop~$X^\nu_t$ right after the exit {from}~$D$.

\begin{lemma}\label{lem:D-exit} (i) Provided that~$L$ is large enough (depending on~${K_0}$), one has
\[
\P(|\Re X^\nu_{K_0L}|> {q_0}L^2)\ \le\ 1/L\,.
\]
(ii) Let~$t_0$ and~$\sigma_0$ be related to~$L$ by~\eqref{eq:x-st0-def} and assume, by contradiction, that $\var(\Im X^0_{t_0})<{\sigma_0^2}$ for the random walk started at~$0$. Then, provided that $L$ is large enough, the following holds:
\[
\P(\Im X^\nu_{K_0L}<0)\ \le\ {(1+2c_0^{-1})}/L\,,
\]
{where the constant~$c_0=c_0(\kappa)>0$ is given by~\eqref{eq:c0-def}.}
\end{lemma}
\begin{proof}
(i) This is an easy corollary of Lemma~\ref{lem:area} and Proposition~\ref{prop:concentration}. Since the width of the rectangle~$D$ is of order~$L^2$, the probability to exit from~$D$ before time~$K_0L$ through its vertical sides starting from~$\nu$ is actually of order~$L^{-3/2}$.

\smallskip

\noindent (ii) {We call a vertex~$v$ \emph{`bottom-screening'} if the probability that the random walk started at~$v$ and run for time~$K_0L$ exits the rectangle~$D$ given by~\eqref{eq:rect-D-def} through its bottom side is less than $1/L$. Let $R:=[-q_0L^2,q_0L^2]\times[-q_0,q_0]$ be the bottom part of~$D$. Let us first show that there exists a (non-oriented) path~$\Gamma$ on the \mbox{T-graph} that crosses the rectangle $R$ horizontally and consists of `bottom-screening' vertices.

Assume that such~$\Gamma$ does not exist. For topological reasons, this implies existence of a path~$\gamma$ crossing~$R$ vertically and not containing `bottom-screening' vertices. Due to Corollary~\ref{cor:inv-concentration-triv}, if we start the random walk at the origin and wait for time $t_0':=16n_0q_0^2L^4$, then it exists from~$R$ with probability at least~$\frac34$. Note however that, due to our assumption~$\var(\Im X_{t_0}^0)<\sigma_0^2$, the probability that $X_t^0$ exits~$R$ through the top or the bottom sides before~$t'_0<t_0$ is bounded by~$\sigma_0^2q_0^{-2}<\frac18$. Hence, with probability at least~$\frac58$ this happens through one of the vertical sides. Moreover, as~$\Re X_t^0$ is a martingale, $X_t^0$ exists~$R$ before time~$t_0'$ through \emph{each} of the vertical sides with probability at least~$\frac18$. For topological reasons, this also implies that~$X_t^0$ crosses the path~$\gamma$ earlier on. Therefore,} if we additionally wait for time~$K_0L$ in the latter case so that {this random walk, restarted on~$\gamma$,} hits the bottom side of~$R$ with probability at least~$L^{-1}$, {we see that}
\[
\var(\Im X^0_{t_0})\ge {\tfrac{1}{8}L^{-1}q_0^2}={\sigma_0^2},\qquad t_0={t_0'+K_0L=16n_0q_0^2L^4+K_0L.}
\]
This is exactly the choice of constants~$t_0$ and~$\sigma_0$ made in~\eqref{eq:x-st0-def}, {which is a contradiction. The proof of the existence of the `bottom-screening' path~$\Gamma$ is complete.}

{It is easy to see that the existence of~$\Gamma$ implies the required estimate of the probability that the random walk exits~$D$ through the bottom side if started at the invariant measure~$\nu$.} Let a vertex~$v\in D$ lie above~$\Gamma$. For topological reasons, the probability {that the random walk started at~$v$ hits} the bottom side of~$D$ before time~$K_0L$ is bounded by~$1/L$ as in this case {the walk} should first cross~$\Gamma$. Finally, the total mass (in the measure~$\nu$) of vertices lying below~$\Gamma$ is bounded by~$2c_2(\kappa)/(Lc_1(\kappa))$ due to Lemma~\ref{lem:area} and the fact that~$\Gamma\subset R$.
\end{proof}

We now move back to the proof of Proposition~\ref{prop:ellipticity}. Recall that the choice of the {constant~$K_0=K_0(\kappa)$} is postponed until the end of the proof and that~$L$ will be eventually chosen large enough (depending on~${K_0}$).

Recall {also} that~$\nu_{K_0L}(v)\le \nu_0(v)$ for all~$v\in D$ and that the remaining mass $1-\sum_{v\in D}\nu_{K_0T}(v)$ is located at distance at most~$2\delta=2$ from the boundary of~$D$. It is easy to see that one can construct a coupling of the laws~$\nu_0$ and~$\nu_{K_0L}$ such that they differ if and only if the latter variable does not belong to~$D$. Let~$(\xi,\xi')$ denote the corresponding coupling obtained by taking the imaginary part: $\xi$ has the law of $\Im X^\nu_0$ and~$\xi'$ has the law of $\Im X^\nu_{K_0L}$.

Note that~$\E(\xi')=\E(\xi)$ due to the martingale property of the random walk. Therefore,
\begin{align}
{\tfrac12c_0\cdot K_0}L\ \le\ \var(\xi')-\var(\xi)\ &=\ \P(\xi'\ne\xi)\cdot (\var(\xi'|\xi'\ne\xi)-\var(\xi|\xi'\ne\xi)) \notag \\ &\le\ \P(\xi'\ne\xi)\cdot \var(\xi'|\xi'\ne\xi), \label{eq:x-cond-var}
\end{align}
see Lemma~\ref{lem:D-variance} for the first inequality. It remains to prove that~\eqref{eq:x-cond-var} can be bounded from above (provided that~$L$ is large enough, depending on~${K_0}$) by $c(\kappa)L$, where {$c(\kappa)$ does \emph{not} depend on $K_0$: once} this is done, choosing first~${K_0}$ and then~$L$ big enough one obtains a desired contradiction.

Trivially,~$\xi'\in [-3q_0,(L+1)q_0]$ as the steps of the random walk are bounded by~${\delta=1}\le 2q_0$. Let
\[
p_\pm:=\P(E_\pm),\quad\text{where}\ \ \begin{array}{l}E_-:=\{\omega:\xi'\ne\xi\ \&\ \xi'\in [-3q_0,(L-1)q_0])\},\\
E_+:=\{\omega:\xi'\ne\xi\ \&\ \xi'\in ((L-1)q_0,(L+1)q_0])\}\end{array}
\]
be the probabilities that~$X^\nu_{K_0L}$ exited from the rectangle~$D$ through the bottom or vertical ($E_-$) or through the top ($E_+$) side, respectively. It follows from Lemma~\ref{lem:D-exit} that
\begin{equation}
\label{eq:x-p-bound}
p_-\ \le\ {2(1+c_0^{-1})}\cdot L^{-1}
\end{equation}
provided that~$L$ is large enough (depending on~${K_0}$).
Though we do not have a similar estimate of the exit probability~$p_+$, the bound~\eqref{eq:x-p-bound} alone is already enough to control~\eqref{eq:x-cond-var}. Indeed, we have
\begin{align*}
&\P(\xi'\ne\xi)\cdot \var(\xi'|\xi'\ne\xi)\\
& = p_-\cdot \var(\xi'\,|\,E_-)\ +\ p_+\cdot\var(\xi'\,|\,E_+)\ +\ p_-p_+\cdot(\E(\xi'\,|\,E_-)-\E(\xi'\,|\,E_+))^2\\
&\le\ {2(1+c_0^{-1})}L^{-1}\cdot \tfrac{1}{4}((L+2)q_0)^2\ +q_0^2\ +{2(1+c_0^{-1})}L^{-1}\cdot((L+4)q_0)^2\\
&\le\ c(\kappa)L
\end{align*}
for large {enough~$L$ (depending of~$K_0$), where the constant~$c(\kappa)$ does \emph{not} depend on~$K_0$.} The proof is complete.
\end{proof}

\def\cst{\operatorname{cst}}

\subsection{Crossing estimates for forward and backward random walks} \label{sub:crossing}
In this section we first collect several standard corollaries of the uniform ellipticity estimate~\eqref{eq:ellipticity} for the forward random walk and harmonic functions on T-graphs. Then we argue that the same statements hold for the backward random walk. Since the backward random walk is \emph{not} a martingale, standard arguments do not apply, instead we derive crossing estimates for this walk from those available for the forward one; see Proposition~\ref{prop:backward_uniform_crossing}. Below we always assume that regions of T-graphs under consideration are such that assumption~{\LipKd{\kappa}{\delta}} is fulfilled (in the corresponding regions of t-embeddings) with a fixed constant~$\kappa<1$.

\begin{lemma}\label{lem:(S)} There {exist} constants~$\rho_0,\eta_0,\varsigma_0>0$ (depending on~$\kappa$) such that the following holds: for all discs~$B(v,r)$ with~$r\ge\rho_0\delta$ drawn over the T-graph and centered at its vertex, and for all intervals~$I\subset\R/2\pi\mathbb{Z}$ of length~$\pi-\eta_0$,
\begin{equation}
\label{eq:x-P-exit}
\P^v\big(\text{$X_t$ exits~$B(v,r)$ through the boundary arc~$\{v+r e^{i\theta}, \theta \in I\}$}\big)\ \ge\ \varsigma_0,
\end{equation}
where~$\P^v$ denotes the law of the (continuous time) random walk on the T-graph started at~$v$.
\end{lemma}
\begin{proof} See the proof of~\cite[Lemma~3.7]{BLRnote}, we briefly recall this proof here for completeness. Without loss of generality, assume that~$I=(\frac{1}{2}\eta_0,\pi-\frac{1}{2}\theta_0)$. Let~$\tau$ denote the exit time from~$B(v,r)$, {note that~$\tau<\infty$ almost surely due to Corollary~\ref{cor:inv-concentration-triv} and the Markov property. Also, note that $\E(\tau)=\E(|X^v_\tau-v|^2)\ge r^2$ since~$|X^v_t-v|^2-t$ is a martingale.} As the process~$\Im(X^v_t)$ is a martingale, {in order to prove~\eqref{eq:x-P-exit} it is enough} to find a constant~$\sigma_0'>0$ such that
\[
\var(\Im(X^v_\tau))\ge (\sigma'_0)^2{r^2}\quad \text{for all}\ \ r\ge\rho_0{\delta}.
\]
Proposition~\ref{prop:ellipticity} {implies that} the \emph{discrete time} process \[Y_k:=|\Im(X^v_{kt_0\delta^2}-v)|^2-k\sigma_0^2\delta^2\] is a submartingale. 
The optional stopping applied to the stopping time
\[
\varkappa:=\lceil(t_0\delta^2)^{-1}\tau\rceil{}{\in (\tau,\tau+t_0\delta^2]}
\]
gives the desired result since
$\var(\Im(X^v_{\varkappa t_0\delta^2}-X^v_\tau))\le\E(\varkappa-\tau)\le t_0\delta^2$.
\end{proof}

\begin{figure}
\begin{center}
\includegraphics[width = .8\textwidth]{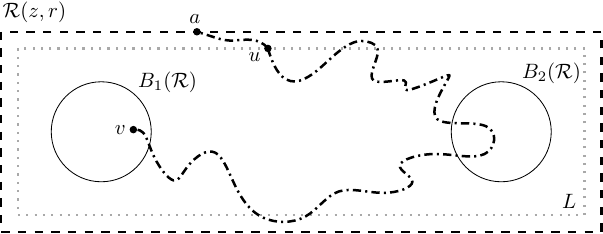}
\caption{Notation used in the discussion of the uniform crossing property for forward and backward walks on T-graphs.}\label{fig:domains}
\end{center}
\end{figure}

Given~$z\in\C$ and~$r>0$, let
\begin{align*}
&\cR(z,r)\ :=\ {z +[-3r,3r]\times[-r,r]},\\
&B_1(\cR):=B(z\!-\!2r,\tfrac{1}{2}r),\quad B_2(\cR):=B(z\!+\!2r,\tfrac{1}{2}r)
\end{align*}
be a rectangle and two discs drawn over the T-graph; see {Fig.~\ref{fig:domains}}. The following property of random walks was called the \emph{uniform crossing property} in~\cite{BLR1}.
\begin{lemma}\label{lem:forward_uniform_crossing}
There exist constants~$\rho'_0,\varsigma'_0>0$ such that the following holds for all rectangles~$\cR(z,r)$ with $r\ge\rho'_0\delta$ drawn over the T-graph:
\[
\P^v ( \text{$X_t$ hits $B_2(\cR)$ before exiting $\cR$} )\ \ge\ \varsigma'_0\quad \text{for all $v\in B_1(\cR)$}.
\]
\end{lemma}
\begin{proof}
This is a simple geometrical corollary of Lemma~\ref{lem:(S)}; e.g., see~\cite[Appendix]{chelkak-toolbox} or the proof of~\cite[Theorem~3.8]{BLRnote}; note that the latter uses the martingale property of~$X_t$ once more.
\end{proof}
The crossing estimates discussed above easily imply the \emph{elliptic Harnack principle} for positive harmonic functions on planar graphs.
\begin{proposition} \label{prop:Harnack-positive}
For each~$\rho<1$ there exists a constant~$c(\rho)=c(\rho,\kappa)>0$ such that, for each positive harmonic function~$H$ defined inside a disc~$B(v_0,r)$ drawn over the T-graph, we have
\begin{samepage}
\[
\min\nolimits_{v\in B(v_0,\rho r)}H(v)\ \ge\ c(\rho)\cdot \max\nolimits_{v\in B(v_0,\rho r)}H(v)
\]
provided that~$(1-\rho)r\ge\cst\cdot\delta$ for a constant~$\cst$ depending on~$\kappa$ only.
\end{samepage}
\end{proposition}
\begin{proof}
This is a standard argument, which we also recall for completeness. Let~$v_\mathrm{max}$ and~$v_\mathrm{min}$ be the vertices in~$B(v_0,\rho r)$ at which~$H$ attains its maximal and minimal values, respectively. It follows from the maximal principle that~$H(\cdot)\ge H(v_\mathrm{max})$ along some nearest-neighbor path~$\gamma_\mathrm{max}$ going from~$v_\mathrm{max}$ to the boundary of~$B(v_0,r)$. The uniform crossing estimates ensure that the probability that the random walk started at~$v_\mathrm{min}$ hits~$\gamma_\mathrm{max}$ before exiting~$B(v_0,r)$ is uniformly bounded from below. The optional stopping theorem concludes the proof.
\end{proof}
As pointed out in~\cite[Lemma~4.4]{BLR1}, the elliptic Harnack principle allows one to strengthen the claim of Lemma~\ref{lem:forward_uniform_crossing} by additionally conditioning the random walk to exit from~$\cR$ through a fixed vertex~$a$.
\begin{lemma}\label{lem:forward-crossing-conditioned}
Let~$\tau$ be the exit time from~$\cR=\cR(z,r)$ of the random walk~$X_t$ on the T-graph. There exists a constant~$\varsigma''_0=\varsigma''_0(\kappa)>0$ such that the following holds: for all $v\in B_1(\cR)$ and all {exit points~$a$} such that~$\P(X_\tau={a})>0$, we have
\[
\P^v ( \text{$X_t$ hits $B_2(\cR)$ before exiting $\cR$}\ |\ X_\tau={a} )\ \ge\ \varsigma''_0
\]
provided that~$r\ge\cst\cdot\delta$ for a constant~$\cst$ depending on~$\kappa$ only. Due to the strong Markov property, the same lower bound also applies to any conditioning made after the exit time~$\tau$.
\end{lemma}
\begin{proof} Let~$H(v'):=\P^{v'}(X_\tau={a})$, note that~$H$ is a positive harmonic function in~$\cR$. We have
\begin{align*}
\P^v (\text{$X_t$ hits $B_2(\cR)$}~&\text{before exiting $\cR$ and}\ X_\tau={a} )\\
\ge\ & \P^v(\text{$X_t$ hits $B_2(\cR)$ before exiting $\cR$})\cdot \min\nolimits_{v'\in B_2(\cR)}H(v').
\end{align*}
Since~$\min\nolimits_{v'\in B_2(\cR)}H(v')$ is comparable to $\P^v(X_\tau\!={a})=H(v)$ due to the Harnack principle, the required estimate follows.
\end{proof}

We are now in the position to prove a similar uniform crossing estimate for \emph{backward} random walks on T-graphs. Again, since this estimate does not depend on the exit point, the same bound holds for any conditioning made after the exit from~$\cR$ (and, in particular, unconditionally).

\begin{proposition} \label{prop:backward_uniform_crossing}
Let~$\widetilde{\tau}$ be the exit time from~$\cR=\cR(z,r)$ for the backward random walk~$\widetilde{X}_t$ on the T-graph. There exists a constant~$\widetilde{\varsigma}_0=\widetilde{\varsigma}_0(\kappa)>0$ such that the following holds: for all $v\in B_1(\cR)$ and all {exit points~$a$} such that~$\widetilde{\P}(\widetilde{X}_{\widetilde{\tau}}={a})>0$, we have
\[
\widetilde{\P}^v ( \text{$\widetilde{X}_t$ hits $B_2(\cR)$ before exiting $\cR$}\ |\ \widetilde{X}_{\widetilde{\tau}}={a} )\ \ge\ \widetilde{\varsigma}_0
\]
provided that~$r\ge\cst\cdot\delta$ for a constant~$\cst$ depending on~$\kappa$ only.
\end{proposition}
\begin{proof}
We decompose the backward random walk~$\widetilde{X}_t$ as first a sequence of loops around~$v$ followed by an excursion from~$v$ to~$a$ inside~$\cR$ conditioned not to return to $v$. Clearly, the loops can only contribute positively to the probability to hit~$B_2$ before the exit from~$\cR$. Let~$\widetilde{\mathbb{P}}^{v\to {a}}$ denote the probability measure on such excursions of the backward walk~$\widetilde{X}_t$. Up to reversing the direction of trajectories, this measure is the same as the measure~$\P^{{a}\to v}$ on excursions from~$a$ to~$v$ staying inside~$\cR$ of the \emph{forward} random walk~$X_t$. Therefore, it is enough to obtain the uniform lower bound
\begin{align*}
\P^{{a}\to v}(\,\text{excursion}&~X_t~\text{visits}~B_2\,)\\
&=\ \P^{a}(\,X_{t\wedge\tau}~\text{visits~$B_2$ before~$v$}\;|\; X_{t\wedge\tau}~\text{visits~$v$}\,)\ \ge\ \widetilde{\varsigma}_0>0,
\end{align*}
where, as before,~$\tau$ stands for the exit time from~$\cR$ of the forward walk.

Let~$L$ be a contour going along the T-graph near the boundary of a slightly smaller rectangle (e.g.,~$z+[-\tfrac{17}{6}r,\tfrac{17}{6}r]\times[-\tfrac{5}{6}r,\tfrac{5}{6}r]$, see Fig.~\ref{fig:domains}) which still contains both discs~$B_1(\cR)$, $B_2(\cR)$. Since each random walk trajectory running from~$a$ to~$v$ should cross~$L$, it is enough to prove the same uniform estimate for trajectories started at~$L$:
\[
\P^{u}(\,X_{t\wedge\tau}~\text{visits~$B_2$ before~$v$}\;|\; X_{t\wedge\tau}~\text{visits~$v$}\,)\ \ge\ \widetilde{\varsigma}_0>0,\quad \text{for all ${u}\in L$.}
\]
This statement follows from Lemma~\ref{lem:forward-crossing-conditioned} applied to a suitable chain of smaller rectangles $\cR_1,\ldots,\cR_n$: on each step we condition on the event that the random walk visits~$v$ before hitting the boundary of~$\cR$, both~$v$ and~$\partial\cR$ being outside of~$\cR_j$.
\end{proof}

\begin{corollary} The elliptic Harnack inequality (see Proposition~\ref{prop:Harnack-positive}) also holds for positive functions which are harmonic with respect to the backward random walk on the T-graph.
\end{corollary}
\begin{proof}
The same arguments as those given in the proof of Proposition~\ref{prop:Harnack-positive} (based upon the non-conditioned version of Proposition~\ref{prop:backward_uniform_crossing}) apply.
\end{proof}

\subsection{Subsequential limits of harmonic and t-holomorphic functions} \label{sub:sub-limits}
We begin this section with an a priori H\"older-type bound for oscillations of harmonic functions on T-graphs and t-holomorphic functions on t-embed\-dings. Then we use these bounds to claim the existence of subsequential limits of uniformly bounded sequences of functions defined on a sequence of T-graphs or t-embeddings~$\cT^\delta$ satisfying assumption~{\LipKd{\kappa}{\delta}} with~$\delta\to 0$ and the same~$\kappa<1$. For shortness, we do not include the superscript~$\delta$ in the notation until Corollary~\ref{cor:pre-compact}.

For a real-valued function~$H$ defined on vertices of a T-graph and a region~$U$, denote
\[
\osc_U H\ :=\ \max\nolimits_{v\in U}H(v)-\min\nolimits_{v\in U}H(v)\,.
\]
In the same way, for a t-white-holomorphic function~$F_\frw$ (similarly, for a t-black-holomorphic function~$F_\frb$) and a region~$U$ of a t-embedding, let
\begin{align*}
\osc_U F_\frw\ &:=\ \max\nolimits_{w,w'\in {W_U}}|F_\frw^\tw(w')-F_\frw^\tw(w)|\,,\\
\osc_U F_\frb\ &:=\ \max\nolimits_{b,b'\in {B_U}}|F_\frb^\tb(b')-F_\frb^\tb(b)|\,.
\end{align*}
\begin{proposition}\label{prop:Holder} There exist constants $\beta=\beta(\kappa)>0$ and~$C=C(\kappa)>0$ such that the following estimates hold for all
harmonic functions~$H$ (resp., t-holomorphic functions~$F$) defined in a ball of radius~$R>r$ drawn over a T-graph (resp., over a t-embedding):
\[
\osc_{B(v,r)}H\ \le\ C(r/R)^\beta\osc_{B(v,R)}H\]
and
\[\osc_{B(z,r)}F\ \le\ C(r/R)^\beta\osc_{B(z,R)}F
\]
provided that $r\ge \cst\cdot \delta$ for a constant~$\cst$ depending on~$\kappa$ only.
\end{proposition}

\begin{proof}
The estimate for harmonic functions is straightforward from Lem\-ma~\ref{lem:forward_uniform_crossing}. Indeed, the same argument as in the proof of Proposition~\ref{prop:Harnack-positive} ensures that
\[
\min\nolimits_{B(v,r)}H\ \ge\ p_0\max\nolimits_{B(v,r)}H+(1\!-\!p_0)\min\nolimits_{B(v,2r)}H
\]
for some~$p_0=p_0(\kappa)>0$. Together with a similar inequality for~$-H$, this yields
\begin{equation}
\label{eq:x-osc-H}
(1+p_0)\osc\nolimits_{B(v,r)}H\ \le\ (1-p_0)\osc_{B(v,2r)}H.
\end{equation}
Iterating~\eqref{eq:x-osc-H}, one obtains the desired bound with the exponent
\[\beta=\log 2/\log((1+p_0)/(1-p_0)).\]

To prove the same result for t-{holomorphic} functions~$F=F_\frw$, recall that Proposition~\ref{prop:backward_harmonic} implies that both~$\Re F_\frw^\tw$ and~$\Im F_\frw^\tw$ can be viewed as harmonic functions with respect to the \emph{backward} random walks on appropriate T-graphs. Using Proposition~\ref{prop:backward_uniform_crossing} and the inclusions~\eqref{eq:T+0-discs} and applying the same argument as above, we obtain the estimate
\[
(1+p_0')\osc_{B(z,(1+\kappa)^{-1}r)}\Re F_\frw^\tw\ \le\ (1-p_0')\osc_{B(z,(1-\kappa)^{-1}r)}\Re F_\frw^\tw
\]
(and similarly for~$\Im F_\frw^\tw$) for some~$p_0'=p_0'(\kappa)>0$. Therefore, there {exist constants} \mbox{$\beta'=\beta'(\kappa)>0$} and~$C'=C'(\kappa)>0$ such that
\[
\osc_{B(z,r)}\Re F_\frw^\tw\ \le\ C'(r/R)^{\beta'}\osc_{B(z,R)}\Re F_\frw^\tw,
\]
and similarly for~$\Im F_\frw^\tw$. Since~$\osc_{B(z,r)} F_\frw\le\osc_{B(z,r)}\Re F_\frw^\tw+\osc_{B(z,r)}\Im F_\frw^\tw$, this gives a similar estimate 
for oscillations of the function~$F_\frw^\tw$ itself.
\end{proof}

\begin{corollary}\label{cor:pre-compact} {(i) Let a sequence of t-embeddings~$\cT^\delta$ with $\delta\to 0$, a constant $\kappa<1$, and an open set~$U\subset\C$ be such that~$\cT^\delta$ covers~$U$ and satisfies the assumption~\LipKd{\kappa}{\delta} for all sufficiently small~$\delta$.} Assume that~$F^\delta$ are t-holomorphic functions on~$\cT^\delta\cap U$ and that these functions are uniformly bounded on compact subsets of~$U$. Then, the family $\{F^\delta\}$ is pre-compact in the topology of the uniform convergence on compact subsets of~$U$.

\smallskip

\noindent (ii) In the same setup, let a sequence of T-graphs associated to~$\cT^\delta$ {be} given and an open set~$V\subset\C$ be covered by each of {them}. If functions~$H^\delta$ are harmonic (on T-graphs) and uniformly bounded on compact subsets of~$V$, then the family~$\{H^\delta\}$ is pre-compact in the topology of the uniform convergence {on} compact subsets of~$V$.
\end{corollary}
\begin{proof}
Both statements are just applications of the {Arzel\`a}-Ascoli criterium. Indeed, Proposition~\ref{prop:Holder} yields that, on each compact set, the families~$\{F^\delta\}$ and~$\{H^\delta\}$ are equicontinuous on scales above~$\delta$. To get rid of small scales one can, for instance, consider convolutions of, say, $F^\delta$ with mollifiers of size~$\delta^{1/2}$: thus obtained functions stay close to~$F^\delta$ due to Proposition~\ref{prop:Holder} and are equicontinuous on all scales.
\end{proof}

In the same setup, assume now that the origami maps~$\cO^\delta$ associated to~$\cT^\delta$ converge as~$\delta\to 0$:
\begin{equation}
\label{eq:O-to-theta}
\cO^\delta(z)\to \vartheta(z),\quad \text{uniformly on compact subsets of~$U$;}\quad \vartheta:U\to\C.
\end{equation}
(Note that one can always find a subsequential limit since all~$\cO^\delta$ are~$1$-Lipschitz functions, defined up to a constant.) Clearly,~$\vartheta$ {has} to be a~{$\kappa$-Lipschitz function.}
\begin{proposition} \label{prop:f-closed-forms} In the setup described above, for each subsequential limit $f_\frw:U\to\C$ of {uniformly} bounded {(on compact subsets of~$U$)} t-white-holomorphic functions~$F_\frw^\delta$ the differential form~$f_\frw(z)dz+\overline{f_\frw(z)}d\overline{\vartheta(z)}$ is closed.

Similarly, for each subsequential limit~$f_\frb:U\to\C$ of {uniformly} bounded t-black-holomorphic functions~$F_\frb^\delta$ the differential form~$f_\frb(z)dz+\overline{f_\frb(z)}d\vartheta(z)$ is closed.

In particular, if~$\vartheta\equiv 0$, then all such subsequential limits are holomorphic in~$U$.
\end{proposition}
\begin{proof} We consider limits~$f_\frw$ of t-white-holomorphic functions only, the case of~$f_\frb$ is fully similar. Recall that Proposition~\ref{prop:integral_simple} and Lemma~\ref{lem:FdT-in-C} imply that, for each t-white-holomorphic function~$F_\frw=F_\frw^\delta$, the differential form $F_\frw^\tw(z)dz+\overline{F_\frw^\tw(z)}d\overline{\cO(z)}$ is closed. (Let us emphasize that we view such forms as being defined everywhere in~$U\subset\C$ and not only on edges of t-embeddings.) Let~$\gamma$ be a smooth loop in~$U$ and~$F_\frw^{\tw,\delta}\to f_\frw$ as~$\delta\to 0$ uniformly on~$\gamma$. We claim that
\begin{equation}
\label{eq:x-FdO-conv}
\oint_\gamma F_\frw^{\tw,\delta}(z)dz\ \to\ \oint_\gamma f_\frw(z)dz
\quad \text{and}\quad
\oint_\gamma F_\frw^{\tw,\delta}(z)d\cO^\delta(z)\ \to\ \oint_\gamma f_\frw(z)d\vartheta(z)
\end{equation}
as~$\delta\to0$, where the last integral is understood in the Riemann--Stieltjes sense. Indeed, the former convergence is a triviality. To prove the latter note that, for each~$\varepsilon>0$, one can split~$\gamma$ into pieces~$\gamma^{(1)}_\varepsilon,\ldots,\gamma^{(n_\varepsilon)}_\varepsilon$, of diameter at most~$\varepsilon$.
Let~$z^{(j)}_\varepsilon$ be an (arbitrarily chosen) point on~$\gamma^{(j)}_\varepsilon$. Proposition~\ref{prop:Holder} implies that
\[
\osc_{\gamma^{(j)}_\varepsilon}(F_\frw^{\tw,\delta})\ =\ O(\varepsilon^\beta)\quad \text{for each $j=1,\ldots,n_\varepsilon$}
\]
and some exponent~$\beta>0$ independent of~$\delta$. Since all functions~$\cO^\delta$ are $1$-Lipschitz, we see that
\[
\oint_\gamma F_\frw^{\tw,\delta}(z)d\cO^\delta(z)\ =\ \sum_{j=1}^{n_\varepsilon}F_\frw^{\tw,\delta}(z^{(j)}_\varepsilon)\int_{\gamma^{(j)}_\varepsilon}d\cO^\delta(z)\ +\ O(\varepsilon^\beta\cdot\operatorname{length(\gamma)}),
\]
where the O-estimate is uniform in~$\delta$. As~$\delta\to 0$ for a fixed~$\varepsilon>0$, the main term converges to a Riemann-Stieltjes sum approximating the contour integral~$\oint_\gamma f_\frw(z)d\vartheta(z)$. Thus, sending first~$\delta\to 0$ and then~$\varepsilon\to 0$ we arrive at~\eqref{eq:x-FdO-conv}.

Finally, in the `small origami' case $\vartheta\equiv 0$, all subsequential limits~$f_\frw$ and~$f_\frb$ are holomorphic due to Morera's theorem.
\end{proof}

\begin{remark} Proposition~\ref{prop:f-closed-forms} allows one to identify a limit of `discrete complex structures' associated to the notion of {t-holomorphicity} on t-embed\-dings. For a generic limit~$\vartheta(z)$ of origami maps, {the} scaling limits~$f_\frw$ and~$f_\frb$ satisfy different conditions, neither of which is complex-linear. However, if we assume that, say, the differential form~$f^{[\eta]}(z)dz+\overline{f^{[\eta]}(z)}d\overline{\vartheta(z)}$ is closed for \emph{each} of the functions \mbox{$f^{[\eta]}(z):=\tfrac{1}{2}(\overline{\eta}f^{[+]}(z)+\eta f^{[-]}(z))$}, $\eta\in\C$, then this condition, reformulated in terms of the pair~$(f^{[+]},\overline{f}{}^{[-]})$, becomes complex-linear. This setup is explicitly relevant for the dimer model coupling functions due to Proposition~\ref{prop:Fpmpm-def}(ii).
\end{remark}

\subsection{Assumption~\ExpFat{\delta}~and Lipschitzness of harmonic functions on T-graphs} \label{sub:Lipschitz}
This section is devoted to the a priori regularity (Lipschitzness) theory for harmonic functions on T-graphs obtained from t-embeddings satisfying assumption~\LipKd{\kappa}{\delta}. From now onwards we additionally rely upon assumption {\ExpFatPrime{\delta}{\delta'}}\
(see Assumption~\ref{assump:ExpFat-triang} for the case of triangulations and Assumption~\ref{assump:ExpFat-general} for a generalization to arbitrary degrees of faces). Working on, say, a \mbox{T-graph} $\cT+\alpha^2\cO$, we mostly focus on~$\alpha\R$-valued harmonic functions: as discussed in Section~\ref{sub:harmonicity} (in particular, see Remark~\ref{def:D}), their gradients~$\rD[H]$ are t-white-holomorphic functions. {For simplicity, below we assume that~$\alpha=1$.} (Note that we do not really lose generality here: for each~$\delta$, the origami square root function is defined up to a multiple, thus one can always modify them so that $\alpha=1$ for all~$\delta$.

The next theorem is the key result of this section. Loosely speaking, {it} says that {the} gradients of bounded harmonic functions satisfy the standard Harnack-type estimate at least if they do not blow up exponentially fast as~$\delta\to 0$. We later use {the additional assumption~\ExpFatPrime{\delta}{\delta'}}~to forbid the pathological blow-up scenario; {see Corollary~\ref{cor:Lipschitz}}. Note however that we prefer to formulate {these results} in a constructive manner so that no limit passage as~$\delta\to 0$ is involved. In order to ease the notation, below we do not distinguish points on t-embeddings and their images on T-graphs.

\begin{theorem} \label{thm:Lipschitz}
For each~$\kappa<1$ there {exist positive} constants~$\beta_0=\beta_0(\kappa)$ and \mbox{$C_0=C_0(\kappa)$} such that the following {holds}.
Let~$H^\delta$ be a harmonic function in a ball~$B(v,d)$ drawn over a T-graph obtained from a t-embedding~$\cT^\delta$ satisfying assumption~\LipKd{\kappa}{\delta}. Then,
\begin{align*}
\text{either}\ \ &\max\nolimits_{B(v,\frac{1}{2}d)}|\rD[H^\delta]|\ \le\ C_0d^{-1}\cdot \osc\nolimits_{B(v,d)}(H^\delta)\\
\text{or}\quad\,\quad&\max\nolimits_{B(v,{\frac{3}{4}d})}|\rD[H^\delta]|\ >\ {\exp(\beta_0d\delta^{-1})\cdot C_0d^{-1}}\osc\nolimits_{B(v,d)}(H^\delta),
\end{align*}
provided that~$d\ge\cst\cdot \delta$ for a constant~$\cst$ depending on~$\kappa$ only.
\end{theorem}
\begin{proof} Recall that the gradient~$F_\frw^\delta:=\rD[H^\delta]$ is a t-white-holomorphic function defined in the preimage of~$B(v,d)$ on~$\cT^\delta$, which we identify with this ball to ease the notation. Proposition~\ref{prop:Holder} and distortion estimate~(\ref{eq:T+0-discs}) yield the existence of constants~$A=A(\kappa)>1$ and~$r_0=r_0(\kappa)>0$ such that
\begin{equation}
\label{eq:x-const-A}
\osc_{B(z,r)}(F_\frw^{\tw,\delta})\ \le\ \frac{1-\kappa}{{12}}\osc_{B(z,Ar)}(F_\frw^{\tw,\delta})\quad \text{provided that}\ \ r\ge r_0\delta.
\end{equation}
Below we assume that $d\geq {8} A r_0\delta$ and let~$C_0:={8}A(1-\kappa)^{-1}$ be a big enough constant.
 Without loss of generality, assume {also} that $\osc_{B(v,d)}(H^\delta)=1$ and that the first alternative does not hold, i.e., $M_0:=|F_\frw^\delta(z_0)|> C_0d^{-1}$ {at} a point \mbox{$z_0\in B(v,\frac{1}{2}d)$}. We now claim that one can iteratively construct, {while the condition~$z_n\in B(v,d\!-\!Ar_0\delta)$ holds,} a sequence of points~$z_0,z_1,\ldots$ such that
\[
M_{n+1}:=|F_\frw^{\tw,\delta}(z_{{n+1}})|\ge 2M_n\quad \text{and}\quad
|z_{n+1}-z_n|\le A\cdot\max ((1-\kappa)^{-1}M_n^{-1}\,,\,r_0\delta).
\]
Indeed, integrating the differential form~$F_\frw^{\tw,\delta}(z)dz+\overline{F_\frw^{\tw,\delta}(z)}d\overline{\cO^\delta(z)}$ (see Lemma~\ref{lem:FdT-in-C}) along an appropriately oriented segment of length~${2r}$
{centered at} the point~$z_n$ one gets the estimate
\begin{align*}
1\, \ge\, \osc_{B(z_n,r)}(H^\delta)\, &\ge\, {2r}\cdot\big((M_n-\osc_{B(z_n,r)}(F_\frw^{\tw,\delta}))-
{(\kappa\cdot M_n+\osc_{B(z_n,r)}(F_\frw^{\tw,\delta}))}\big)\\
&=\, {2r}\cdot\big((1-\kappa)M_n-{2}\osc_{B(z_n,r)}(F_\frw^{\tw,\delta})\big)
\end{align*}
{If $r\ge (1-\kappa)^{-1}M_n^{-1}$,}
then we must have $\osc_{B(z_n,r)}(F_\frw^{\tw,\delta})\ge {\frac{1}{4}}(1-\kappa)M_n$
and hence {$\osc_{B(z_n,Ar)}(F_\frw^{\tw,\delta})\ge 3M_n=3|F_\frw^{\tw,\delta}(z_n)|$} due to the choice of the constant~$A$ made above. Therefore, the maximal value of~$|F_\frw^{\tw,\delta}|$ in the disc~$B(z_n,Ar)$ must be at least~$2M_n$, as required.
It is easy to see that
\[
|z_{n+1}-z_n|\,\le\, {\max(2^{-n}|z_1-z_0|,Ar_0\delta)}\, \le\, \max(2^{-n-3}d\,,\,Ar_0\delta).
\]
{Hence,} this sequence {has} to make at least~${\tfrac{1}{8}}d({A}r_0\delta)^{-1}$ steps {in order} to leave the disc~$B(z_0,{\frac34d})$ {if started inside~$B(z_0,\frac12d)$.} Since the value~$M_n$ at least doubles at each step, the last such value should be at least~$\exp[\beta_0 d\delta^{-1}]\cdot M_0$ provided we set~$\beta_0:={\tfrac{1}{8}(A}r_0)^{-1}\log 2$. The proof is complete.
\end{proof}

{\begin{corollary}\label{cor:Lipschitz} In the setup of Theorem~\ref{thm:Lipschitz}, let us additionally assume that the t-embedding~$\cT^\delta$ satisfies the assumption \ExpFatPrime{\delta}{\delta'}. Then,
\[
\max\nolimits_{B(v,\frac{1}{2}d)}|\rD[H^\delta]|\ \le\ C_0d^{-1}\cdot \osc\nolimits_{B(v,d)}(H^\delta)
\]
provided that~$d\ge\mathrm{cst}\cdot\max(\delta,\delta')$, where the constant $\mathrm{cst}$ depends only on~$\kappa$.
\end{corollary}
\begin{proof} Without loss of generality, assume that $\osc\nolimits_{B(v,d)}(H^\delta)=1$ and that we work with the T-graph~$\cT^\delta+\cO^\delta$; recall that in this case the gradient~$\rD[H^\delta]$ is a t-white-holomorphic function which we denote~$F_\frw$ in what follows.

For simplicity, let us first consider the case when~$\cT^\delta$ is a triangulation. Provided that~$\frac14(1-\kappa)d\ge \delta'$, the assumption \ExpFatPrime{\delta}{\delta'}\ guarantees that each point in $B(v,\frac 34 d)$ is surrounded by an edge-connected circuit consisting of $\delta\exp(-\delta'\delta^{-1})$-fat triangles (and running inside~$B(v,d)$); let $\gamma$ denote this circuit. The image of each black face~$b\in\gamma$ in the T-graph $\cT^\delta+\cO^\delta$ has length at least~$2\delta\exp(-\delta'\delta^{-1})$, which trivially gives the estimate
\begin{equation}
\label{eq:x-Fb-exp-bound}
|F^\tb_\frw(b)|\ =\ |\rD[H^\delta]|\ \le\ \tfrac12\delta^{-1}\exp(\delta'\delta^{-1})\ \ \text{for all}\ b\in \gamma\cap B.
\end{equation}
Moreover, for each~$w\in W_\gamma$ one can use this estimate at two neighboring black triangles~$b_1,b_2\in \gamma$ to show that
\[
|F^\tw_\frw(w)|\ \le\ \delta^{-1}\exp(2\delta'\delta^{-1})\ \ \text{for all}\ w\in \gamma\cap W.
\]
due to the explicit formula
\begin{equation}
\label{eq:x-reconstruction-Fw}
F^\tw_\frw(w)\ =\ \frac{F^\tb_\frw(b_1)\eta_{b_2}\overline{\eta}_{b_1}-F^\tb_\frw(b_2)\overline{\eta}_{b_2}\eta_{b_1}} {i\Im(\eta_{b_2}\overline{\eta}_{b_1})}
\end{equation}
and the fact that~$\Im(\eta_{b_2}\overline{\eta}_{b_1})\ge \rho\delta^{-1}$ if~$w$ is~$\rho$-fat and has diameter less than~$\delta$. Due to Proposition~\ref{prop:backward_harmonic}, the function~$|F^\tw_\frw|$ satisfies the maximum principle, which allows us to conclude that
the estimate $|F^\tw_\frw|\le\delta^{-1}\exp(2\delta'\delta^{-1})$ holds \emph{everywhere} in the ball $B(v,\tfrac{3}{4}d)$. This rules out the second (pathological) scenario in Theorem~\ref{thm:Lipschitz} provided that
\[
C_0\delta d^{-1}\exp(\beta_0d\delta^{-1})\,\ge\,\exp(2\delta'\delta^{-1}),
\]
which holds true if~$d\ge\mathrm{cst}(C_0,\beta_0)\cdot\max(\delta,\delta')$.

For general t-embeddings~$\cT^\delta$, the same arguments go through with the only caveat that we should prove an appropriate replacement of the estimate~\eqref{eq:x-Fb-exp-bound} for `$\rho$-fat' bigons~$b\in \gamma\cap B^{\tw}_\spl$; see the definition given before Assumption~\ref{assump:ExpFat-general}. To this end, note that such~$b$ has length (in the t-embedding) at least $2\rho$, where $\rho:=\delta\exp(-\delta'\delta^{-1})$, since it is a side of a neighboring (in the circuit~$\gamma$) white $\rho$-fat triangle. Let~$w\in W$ be the white face containing~$b$ and $b_1,...,b_k\in B$ denote the $\rho$-fat black faces that are adjacent to the corresponding boundary arc of~$w$. It is easy to see that the
the equation $\oint F^\tb_\frw d\cT^{\bullet,\delta}_\spl=0$ (and the fact that the face~$\cT^\delta(w)$ is convex and has diameter less than~$\delta$) imply that
\[
2\rho\cdot |F^\tb_\frw(b)|\ \le\ |\cT^{\bullet,\delta}_\spl(b)|\cdot |F^\tb_\frw(b)|\ \le\ 2\delta\cdot\max(|F^\tb_\frw(b_1)|,\ldots,|F^\tb_\frw(b_k)|).
\]
Therefore, the estimate~\eqref{eq:x-Fb-exp-bound} for black faces $b_1,\ldots,b_k$ yields
\[
|F^\tb_\frw(b)|\ \le \tfrac12\delta^{-1}\exp(2\delta'\delta^{-1})\ \ \text{for all bigons}\ b\in \gamma \cap B^\tw_\spl.
\]
The rest of the proof given above goes through with minor modifications.
\end{proof}}

Assume now that we are given a sequence of t-embeddings~$\cT^\delta$ covering a common open set~$U\subset\C$ and such that~$\cO^\delta(z)\to\vartheta(z)$ uniformly on compact subsets of~$U$. {Let~$V:=(\mathrm{id}+\vartheta)(U)$, recall that we denote by~$W^{1,\infty}(V)$ the Sobolev space of functions on~$V$ whose derivatives are uniformly bounded on compact subsets of~$V$.}

\begin{corollary} \label{cor:H-sub-limits}
In the setup {described above, let} t-embeddings~$\cT^\delta$ satisfy {assumptions \LipKd{\kappa}{\delta}\ and \ExpFat{\delta} on~$U$.} Let~$H^\delta$ be uniformly (in~$\delta$) bounded real-valued harmonic functions on T-graphs~$(\cT^\delta+O^\delta)\cap V$. Then, the family~$\{H^\delta\}$ is pre-compact in the topology of the Sobolev space~$W^{1,\infty}(V)$.

The gradients~$f_\frw:=2\partial h\circ(\mathrm{id}+\vartheta)$ of all subsequential limits~$h:V\to\R$ of functions~$H^\delta$, considered as functions in~$U$, are~$\beta$-H\"older, with the exponent~$\beta$ given in Proposition~\ref{prop:Holder}. Moreover, the form~$f_\frw(z)dz+\overline{f_\frw(z)}d\overline{\vartheta(z)}$ is closed in~$U$.

In particular, {if}~$\vartheta\equiv 0$, {then} all subsequential limits~$h$ are harmonic in $V=U$.
\end{corollary}
\begin{proof} {The uniform (on compact sets) boundedness of the gradients~$F_\frw^\delta:=\rD[H^\delta]$ of the functions~$H^\delta$ follows from Corollary~\ref{cor:Lipschitz}. Applying the Arzel\`a--Ascoli theorem, we can assume} that~$H^\delta(v)\to h(v)$ uniformly on compact subsets of~$V$. Proposition~\ref{prop:Holder} guarantees the a priori H\"older regularity of~$F_\frw^\delta$ and the existence of subsequential limits~$f_\frw:U\to \C$ of t-holomorphic functions~$F_\frw^\delta$. Passing to the limit in the identity~$H^\delta=\rI_{\R}[F_\frw^\delta]$ (see Section~\ref{sub:harmonicity}), one sees that
\[
\textstyle  h(v)\,=\,\int^{(\mathrm{id}+\vartheta)^{-1}(v)}\Re(f_\frw(z)dz+\overline{f_\frw(z)}d\overline{\vartheta(z)})
\,=\, \int^{(\mathrm{id}+\vartheta)^{-1}(v)}\Re(f_\frw(z)d(z+\vartheta(z))),
\]
i.e., $2\partial h = f_\frw\circ (\mathrm{id}+\vartheta)$.
Finally, the form $f_\frw(z)dz+\overline{f_\frw(z)}d\overline{\vartheta(z)}$ is closed in~$U$ due to Proposition~\ref{prop:f-closed-forms}.
\end{proof}

\begin{remark}\label{rem:diffusion-theta}
Similarly to the case~$\vartheta\equiv 0$, one can use the above description of gradients~$f_\frw=2\partial h$ of subsequential limits~$h$ of harmonic functions $H^\delta$ to identify the coefficients of a second-order elliptic PDE $(a(v)\partial_{xx}+2b(v)\partial_{xy}+c(v)\partial_{yy})h=0$, $v=x+iy\in V$, {that} all such limits~$h$ satisfy.
Also, it is worth noting that
there exists a very particular case in which this PDE can be viewed {simply} as the harmonicity of~$h$ after a change of variable~$v$. Namely, if~$(z,\vartheta(z))$ is a space-like {maximal} surface in {the Minkowski space~$\R^{2,2}$,} then~$h$ is harmonic in the conformal metric of this surface; {see the companion paper~\cite{CLR2} for more details on this setup.}
\end{remark}

\subsection{Boundedness of functions~$\F\pm\pm$ under assumption~{\ExpFat{\delta}}} \label{sub:Fpmpm=O(1)}
This section contains a technical result needed for the proof of Theorem~\ref{thm:main-GFF} given in Section~\ref{sec:convergence}. In that theorem we assume that the dimer coupling functions~$K_{\cT^\delta}^{-1}(w,b)$ are uniformly bounded as~$\delta\to 0$ provided that~$w$ and~$b$ remain at {a} definite distance from the boundary of~$\Omega$ and from each other. Recall that Proposition~\ref{prop:Fpmpm-def} gives a representation of these functions via four functions~$\F\pm\pm_{\cT^\delta}$, similar to the link between {the} `true' values of a t-white-holomorphic function~$F_\frw^\tw$ and their projections~$F_\frw^\tb$. In Section~\ref{sec:convergence} we rely upon the fact that these functions~$\F\pm\pm_{\cT^\delta}$ are also uniformly bounded. This bound is obtained in Proposition~\ref{prop:Fpmpm-bound}.

\begin{lemma}\label{lem:Fw-via-Fb-bound}
Let~$U\subset\C$ be an open set and t-embeddings~$\cT^\delta$ satisfy both assumption~{\LipKd{\kappa}{\delta}} (with a fixed constant~$\kappa<1$) and assumption \ExpFat{\delta} (as $\delta\to 0$) on~$U$. Assume that~$F_\frw^\delta$ {are} t-white-holomorphic functions on~$\cT^\delta\cap U$. If the values~$F_\frw^{\tb,\delta}$ are uniformly bounded on compact subsets of~$U$ as~$\delta\to 0$, then the same is true for the values~$F_\frw^{\tw,\delta}$.
\end{lemma}
\begin{proof} First, let us note that {the assumption~\LipKd{\kappa}{\delta} implies that} each disc~$D$ of radius~$r\ge\delta$ on a {t-embedding $\cT^\delta$} must intersect two black faces~${b_1,b_2}$ such that ${|\Im(\eta_{b_2}\overline{\eta}_{b_1})|\ge \tfrac12(1-\kappa)}$. Indeed, otherwise there would exist~$\alpha\in\mathbb{T}$ such that $|\eta_b^2-\alpha^2|<1-\kappa$ for all {$b$ in~$D$. In its turn,} this would mean that~$|1+\overline{\alpha}^2\eta_b^2|>1+\kappa$ for all such~$b$, so the image of~$D$ {in} the T-graph~$\cT+\overline{\alpha^2\cO}$ would {have} too big {area} to fit the distortion estimate~\eqref{eq:T+0-discs}. {Using the formula~\eqref{eq:x-reconstruction-Fw} it is not hard to see that} this observation implies the estimate
\[
{\max\nolimits_D|F_\frw^\tw|\ \le\ 4(1-\kappa)^{-1}\cdot(\max\nolimits_D|F_\frw^\tb|+\osc_D(F_\frw^\tw)).}
\]
Provided that the constants~$A$ and~$r_0$ are chosen as in~\eqref{eq:x-const-A}, this {gives}
\begin{align*}
\osc_{B(z,Ar_0\delta)}(F_\frw^\tw)\ &\ge\ {12(1-\kappa)^{-1}\cdot\osc_{B(z,r_0\delta)}(F_\frw^\tw)}\\
&\ge\ {3}\cdot(\max\nolimits_{B(z,r_0\delta)}|F_\frw^\tw|\ -\ 4(1-\kappa)^{-1}\max\nolimits_{B(z,r_0\delta)}|F_\frw^\tb|).
\end{align*}
In particular, if~$\max\nolimits_{B(z,r_0\delta)}|F_\frw^\tw|\ge 16(1-\kappa)^{-1}\max\nolimits_{B(z,r_0\delta)}|F_\frw^\tb|$, then
\[
\max\nolimits_{B(z,Ar_0\delta)}|F_\frw^\tw|\ \ge\ \tfrac{1}{2}\osc_{B(z,Ar_0\delta)}(F_\frw^\tw)\ \ge\ \tfrac{9}{8}\max\nolimits_{B(z,r_0\delta)}|F_\frw^\tw|.
\]
Thus, if the function~$F_\frw^\tw$ attains, {at a certain point in the bulk of~$U$,} a much bigger value than {the maximum of~$|F_\frw^\tb|$ in a vicinity of this point,} then one can iterate the above estimate similarly to the proof of Theorem~\ref{thm:Lipschitz} and observe an exponential {(in~$\delta^{-1}$)} blow-up of~$F_\frw^{\tw,\delta}$. However, this is not possible under assumption~{\ExpFat{\delta}}: {a contradiction is obtained similarly to the proof of Corollary~\ref{cor:Lipschitz} as the reconstruction of the values~$F_\frw^\tw$ from~$F_\frw^\tb$ on $\delta\exp(-\delta'\delta^{-1})$-fat white triangles (and first on relevant black bigons if necessary in the general case) can only give a subexponential factor~$\exp(2\delta'\delta^{-1})$ with~$\delta'\to 0$.}
\end{proof}

\begin{proposition}\label{prop:Fpmpm-bound} Let~$U_1,U_2\subset\C$ be disjoint open sets and {assume that the} \mbox{t-embeddings} $\cT^\delta$ satisfy both assumptions~{\LipKd{\kappa}{\delta}} (with a fixed constant~$\kappa<1$) and  assumption~{\ExpFat{\delta}} (as~$\delta\to 0$) on~$U_1\cup U_2$. If the functions~$K^{-1}_{\cT^\delta}(\cdot,\cdot)$ are uniformly bounded on compact subsets of~$U_1\times U_2$ as~$\delta\to 0$, then the same is true for the functions~$\F\pm\pm_{\cT^\delta}$ defined in Proposition~\ref{prop:Fpmpm-def}.
\end{proposition}
\begin{proof}
Let $w\in U_1$ and~$F_w^{\tb,\delta}(\cdot)=\overline{\eta}_wK^{-1}_{\cT^\delta}(w,\cdot)$ be the values of a t-white-holomorphic function on black faces of~$\cT^\delta$ lying in~$U_2$. {(If $\cT^\delta$ is not a triangulation, we use its black splitting in~$U_1$ and a white splitting in~$U_2$ in what follows.)} Provided that~$w$ stays on a compact subset of~$U_1$ as~$\delta\to 0$, Lemma~\ref{lem:Fw-via-Fb-bound} ensures that the functions~$F_w^{\tw,\delta}$ are uniformly bounded on compact subsets of~$U_2$. Moreover, this estimate is also uniform in~$w$ provided that it stays on a compact subset of~$U_1$. We now use the identity
\[
F_w^{\tw,\delta}(\wc)\ =\ \tfrac{1}{2}\big(\overline{\eta}_w\F++_{\cT^\delta}(\bc,\wc)+\eta_w\F-+_{\cT^\delta}(\bc,\wc)\big),
\]
{where} $w\sim \bc\in U_1$ {and $\wc \in U_2$.}
Frow now onwards, let us fix the second argument~$\wc$. Arguing as in the proof of Lemma~\ref{lem:Fw-via-Fb-bound}, for each disc~$D$ of radius {greater than}~$\delta$ one obtains the estimate
\begin{align*}
\tfrac{1}{2}\big(\osc\nolimits_{D}(\F++_{\cT^\delta}(\cdot,\wc))\ &+\ \osc\nolimits_{D}(\F-+_{\cT^\delta}(\cdot,\wc))\,\big)\\ &\ge\ \tfrac{1}{4}(1-\kappa)\cdot\max\nolimits_D|\F++_{\cT^\delta}(\cdot,\wc)|\ -\ \max\nolimits_{w\in D}|F_w^{\tw,\delta}(\wc)|
\end{align*}
and a similar estimate with~$\max\nolimits_D|\F-+_{\cT^\delta}(\cdot,\wc)|$ in the right-hand side. Denote
\[
G_+^{{\tb,}\delta}({\bc})\,:=\,\tfrac{1}{2}\big(\F++_{\cT^\delta}+{\F+-_{\cT^\delta}\big)(\bc,\wc)}\,,\quad
G_-^{{\tb,}\delta}({\bc})\,:=\,\tfrac{i}{2}\big(\F++_{\cT^\delta}+{\F+-_{\cT^\delta}\big)(\bc,\wc)}\,;
\]
{recall that~$\F+-$ is the conjugate of~$\F-+$. It follows from Proposition~\ref{prop:Fpmpm-def} that $G_\pm^{\delta}$ are t-black-holomorphic functions
and that their values on white faces are given by
\begin{equation}
\label{eq:x-Gpm-white-values}
G_+^{\tw,\delta}(w)\,=\,\eta_w\Re F_w^{\tw,\delta}(\wc),\qquad G_-^{\tw,\delta}(w)\,=\,-\eta_w\Im F_w^{\tw,\delta}(\wc).
\end{equation}
For each disc~$D$ of radius greater than $\delta$ the following estimate holds:}
\begin{align*}
&\osc\nolimits_{D}(G_+^{\tb,\delta})+\osc\nolimits_{D}(G_-^{\tb,\delta})\\
 &\quad \ge\ \tfrac{1}{2}\big(\osc\nolimits_{D}(\F++_{ \cT^\delta}(\cdot,\wc))+\osc\nolimits_{D}(\F-+_{\cT^\delta}(\cdot,\wc))\,\big)\\
&\quad \ge\ \tfrac{1}{8}(1-\kappa)\cdot\big(\max\nolimits_D|\F++_{\cT^\delta}(\cdot,\wc)|+\max\nolimits_D|\F-+_{\cT^\delta}(\cdot,\wc)|\,\big)\, -\, \max\nolimits_{w\in D}|F_w^{\tw,\delta}(\wc)|\\
&\quad \ge\ {\tfrac{1}{8}}(1-\kappa)\cdot\big(\max\nolimits_D|G_+^{\tb,\delta}|+\max\nolimits_D|G_-^{\tb,\delta}|\,\big)\, -\, \max\nolimits_{w\in D}|F_w^{\tw,\delta}(\wc)|
\end{align*}
The proof can be now completed similarly to that of Lemma~\ref{lem:Fw-via-Fb-bound}. Since the functions~$G_\pm^{\tb,\delta}$ are t-holomorphic, there exists a constant~$A'=A'(\kappa)>1$ such that
\begin{align*}
\osc_{B(z,A'r)}(G_\pm^{\tb,\delta})\ \ge\ {32}(1-\kappa)^{-1}\cdot\osc_{B(z,r)}(G_\pm^{\tb,\delta})\quad \text{provided that}\ \ r\ge r_0\delta.
\end{align*}
{Denote~$D:=B(z,r_0\delta)$ and~$A'D:=B(z,A'r_0\delta)$. We now have the estimate}
\begin{align*}
&\max\nolimits_{{A'D}}|G_+^{\tb,\delta}|+\max\nolimits_{{A'D}}|G_-^{\tb,\delta}|\ \ge\ \tfrac{1}{2}\big(\osc_{{A'D}}(G_+^{\tb,\delta})+\osc_{{A'D}}(G_-^{\tb,\delta})\,\big)\\
&\qquad \ge\ 2\cdot\big(\max\nolimits_{{D}}|G_+^{\tb,\delta}|+\max\nolimits_{{D}}|G_-^{\tb,\delta}|\,\big)\ -\ {16(1-\kappa)^{-1}}\cdot\max\nolimits_{w\in {D}}|F_w^{\tw,\delta}(\wc)|.
\end{align*}
If (at least) one of the functions~$G_\pm^{\tb,\delta}$ attained, in the bulk of~$U$, a value {much} greater than the maximum of $|F_w^{\tw,\delta}(\wc)|$ {in a vicinity of this point}, this would imply the existence of exponentially {(in~$\delta^{-1}$)} big values of (at least one of) these functions. {However, this scenario is not possible due to a subexponential cost of the reconstruction of t-black holomorphic functions~$G_\pm^{\tb,\delta}(\bc)$ from their values~$G_\pm^{\tw,\delta}(w)$, which are given by~\eqref{eq:x-Gpm-white-values} and, as already mentoned at the beginning of the proof, are uniformly (on compacts) bounded due to Lemma~\ref{lem:Fw-via-Fb-bound}.}
\end{proof}

\section{Convergence to the GFF: {a general framework}}
\label{sec:convergence}

\newcommand\f[1]{f^{\scriptscriptstyle [#1]}}
\newcommand\of[1]{\overline{f}{}^{\scriptscriptstyle [#1]}}
\newcommand\h[1]{h^{\scriptscriptstyle [#1]}}
\newcommand\cTT[2]{\cT^{\scriptscriptstyle [#1]}_{\scriptscriptstyle [#2]}}

This section is devoted to the proof of Theorem~\ref{thm:main-GFF}. In particular, instead of the very mild Lipschitzness condition~\LipKd{\kappa}{\delta}, below we rely upon a {much stronger} `small origami' assumption: the origami maps ${\cO_m}$ tend to $0$ as ${m\to\infty}$, uniformly on compact subsets. Though we do not include such a discussion into this paper, let us nevertheless mention that a similar (though more involved) analysis can be performed {assuming} that the origami maps ${\cO_m}(z)$ converge, as ${m\to\infty}$, to a function~$\vartheta(z)$, which is a graph of a {maximal} surface in {the Minkowski space~$\R^{2,2}$.} In this situation one eventually obtains the GFF in the conformal metric of the corresponding surface and not just in the Euclidean metric on~$\Omega$. We refer the interested reader to {the companion} paper~\cite{CLR2} for details and focus on the case $\vartheta(z)\equiv 0$ from now onwards. Also, in Section~\ref{sub:discussion} we briefly discuss how several known setups (Temperleyan~\cite{kenyon-gff-a,kenyon-gff-b,kenyon-honeycomb}, piecewise {Temperleyan}~\cite{russkikh-t}, hedgehog domains~\cite{russkikh-h}) fit the general framework developed in our paper.

\subsection{Subsequential limits of the dimer coupling {functions}}
Recall the expression of the dimer coupling function
\[
K^{-1}(w,b)\ =\ {\tfrac14(\F++ + \eta_b^2\F+- +\eta_w^2\F-+ +\eta_w^2\eta_b^2\F--)(\bc,\wc),}
\]
{where $w\sim\bc$ and $b\sim\wc$,} via the functions \mbox{$\F\pm\pm:(B\smallsetminus\partial B)\times (W\smallsetminus \partial W)\to\C$} given in Proposition~\ref{prop:Fpmpm-def}. {Let~$U:=\mathrm{Int}K$ be the interior of a compact subset $K\subset\Omega$, where~$\Omega$ stands for the limiting domain of t-embeddings $\cT_m$ under consideration. Under the assumptions~\LipKd{\kappa}{\delta}\ and \ExpFatPrime{\delta}{\delta'} on~$K$}, Proposition~\ref{prop:Fpmpm-bound} and Proposition~\ref{prop:Holder} imply that the functions $\F\pm\pm_{{\cT_m}}$ are uniformly bounded and equicontinuous on scales above ${\delta_m=\delta_m(K)}$ provided that their arguments remain at a {definite} distance from ${\partial U}$ and from each other. {Using the Arzel\`a-Ascoli theorem on each such~$U$ and applying the diagonal process for a sequence of compacts~$K$ approximating~$\Omega$ from inside, one obtains} the existence of \emph{subsequential} limits:
\begin{equation}
\label{eq:fpmpm-def}
\F\pm\pm_{{\cT_m}}(\bc,\wc)\,\to\,\f{\pm\pm}(z_1,z_2)\quad \text{if}\ \ \bc\to z_1,\ \ \wc\to z_2\ \ \text{as}\ \ {m=m_k\to \infty;}
\end{equation}
the convergence is uniform provided that~$z_1$ and~$z_2$ remain at a {definite} distance from~$\partial\Omega$ and from each other. We list the key properties of functions~$\f{\pm\pm}$ in the next proposition; {note that these properties do \emph{not} define~$\f{\pm\pm}$ uniquely.}

\begin{proposition} \label{prop:fpmpm}
In the setup described above, for each subsequential {limit} $\f{\pm\pm}$ (which might depend on the sequence ${m=m_k\to\infty}$) the following is fulfilled:

\smallskip
\noindent (i) $\f{--}(z_1,z_2)=\overline{\f{++}(z_1,z_2)}$ and $\f{+-}(z_1,z_2)=\overline{\f{-+}(z_1,z_2)}$;

\smallskip
\noindent (ii) for each $z_1\in\Omega$, both functions $\f{\pm +}(z_1,\cdot)$ are holomorphic in $\Omega\smallsetminus\{z_1\}$; similarly, for each $z_2\in\Omega$, both functions $\f{+\pm}(\cdot,z_2)$ are holomorphic in $\Omega\smallsetminus\{z_2\}$;

\smallskip
\noindent (iii) the following asymptotics hold as $z_2\to z_1\in\Omega$ (similarly, as $z_1\to z_2\in\Omega$):
\begin{equation}
\label{eq:fpmpm-res}
\f{++}(z_1,z_2)=\frac{2}{\pi i}\cdot \frac{1}{z_2-z_1}+O(1)\quad \text{and}\quad \f{-+}(z_1,z_2)=O(1).
\end{equation}
\end{proposition}
\begin{proof}
Item (i) is a triviality since the same relations hold for the functions $\F\pm\pm$ before taking the limit, see Proposition~\ref{prop:Fpmpm-def}(i). To prove (ii), recall that, for each $\eta\in\C$ and $\bc\in {B}$, the functions $\overline{\eta}\F++(\bc,\cdot)+\eta\F-+(\bc,\cdot)$ are t-holomorphic due to Proposition~\ref{prop:Fpmpm-def}(iii). Therefore, Proposition~\ref{prop:f-closed-forms} and the `small origami' assumption $\vartheta(z)\equiv 0$ imply that
\[
(\overline{\eta}\f{+ +}(z_1,z_2)+\eta\f{- +}(z_1,z_2))dz_2\ \ \text{is~a~closed~form~in}~\Omega\smallsetminus\{z_1\}.
\]
Morera's theorem yields that $\overline{\eta}\f{++}(z_1,\cdot)+\eta\f{-+}(z_1,\cdot)$ is a holomorphic function of~$z_2$. Varying~$\eta$, one concludes that both functions $\f{++}(z_1,\cdot)$ and $\f{-+}(z_1,\cdot)$ are holomorphic in~$\Omega\smallsetminus\{z_1\}$. The holomorphicity of the functions~$\f{+\pm}(\cdot,z_2)$ follows by the same reasoning.

It remains to {demonstrate} (iii), i.e., {to identify} the behavior of the functions $\f{\pm\pm}$ at $z_1=z_2$. To this end, given ${w\in W}$, consider the complex-valued primitive
\[
I_w(\cdot)\ :=\ \rI_\C[F_w],\quad \text{where}\ \ F_w^\tb(\cdot)=\overline{\eta}_w K^{-1}(w,\cdot),
\]
see Definition~\ref{def:rI}. As the function~$F_w$ is t-holomorphic in a \emph{punctured} domain, the function $I_w$ is {not} well-defined: in fact, it has an additive \emph{monodromy} $2\overline{\eta}_w$ around the white face~$w$ as the integral of the differential form~\eqref{eq:Fb_dT=} around $w$ is equal to~$2\overline{\eta}_w\sum_{b:\,b\sim w}K^{-1}(w,b)K(b,w)=2\overline{\eta}_w$. Therefore,
the function
\[
H_w(\cdot)\ :=\ \rI_{i\overline{\eta}_w\R}[F_w]\ =\ \Pr(I_w(\cdot),i\overline{\eta}_w\R)
\]
is well-defined and harmonic, except at the image of~$w$, on the corresponding T-graph $\cT-\overline{\eta}_w^2\cO$; see Proposition~\ref{prop:representation_derivative}. Moreover, Proposition~\ref{prop:geomT} implies that this image is a \emph{single} (degenerate) vertex of the T-graph~$\cT-\overline{\eta}_w^2\cO$ and thus the real-valued function $i\eta_wH_w$ satisfies either the maximum or the minimum principle in a vicinity of ${w}$.

We now use the representation of the function $F_w$ via $\F\pm\pm$ provided by Proposition~\ref{prop:Fpmpm-def}(ii) to claim that
\begin{equation}
\label{eq:x-Fw=fpmpm+o(1)}
F_w^\tw(\wc)\ =\ \tfrac{1}{2}\big(\overline{\eta}_{w}\f{++}(z_1,z_2)+\eta_w\f{-+}(z_1,z_2)\big)+o(1)\
\end{equation}
for $w\to z_1$ {and} $\wc\to z_2$, 
uniformly over $z_1,z_2$ at a {definite} distance from $\partial\Omega$ and from each other. Denote by $\h{\pm +}(z_1,\cdot)$ the (complex-valued, having additive monodromy around $z_1$) primitive of the holomorphic function~$\f{\pm +}(z_1,\cdot)$. As the primitive of the second term in~\eqref{eq:Fb_dT=} vanishes in the limit ${m\to\infty}$ due to the `small origami' assumption, we have
\[\
H_w(v)\ =\ i\overline{\eta}_w\cdot\tfrac{1}{2}\Im\big(\h{++}(z_1,z_2)+\eta_w^2\h{-+}(z_1,z_2)\big)+o(1)\ \ \text{for}\ w\to z_1,\ v\to z_2.
\]
Recall that the discrete functions $i\eta_w H_w$ always satisfy a one-sided maximum principle in a vicinity of $z_1$. {It is easy to see from  Lemma~\ref{lem:area} that for all sufficiently {large~$m$} one can find $w$ and $w'$ in a small vicinity of a given point $z_1$ such that~$|\eta_{w}^2-\eta_{w'}^2|=2|\Im(\overline{\eta}_w{\eta}_{w'})|\ge 2(c_1(\kappa))^{1/2}>0$; {see also the proof of Lemma~\ref{lem:Fw-via-Fb-bound}.} Therefore, both functions $\h{\pm+}(z_1,\cdot)$ can have only logarithmic singularities at $z_1$.} Moreover, since $H_w$ never has an additive monodromy around $w$, the function $\h{-+}(z_1,\cdot)$ must be well-defined and thus cannot have a logarithmic singularity. Therefore, the function $\f{++}(z_1,\cdot)$ has (at most) a simple pole at~$z_1$ and $\f{-+}(z_1,\cdot)$ does not have any singularity in~$\Omega$.

Finally, if $\f{++}(z_1,z_2)=c\cdot (z_2-z_1)^{-1}+O(1)$ as $z_2\to z_1$, then the function $\h{++}(z_1,\cdot)$ has an additive monodromy~$2\pi i c$ around $z_1$. Since the monodromy of $I_w=\rI_\C[F_w]$ around $w$ is known to be equal to~$2\overline{\eta}_w$, integrating asymptotics~\eqref{eq:x-Fw=fpmpm+o(1)} over a fixed contour surrounding $z_1$ and passing to the limit~${m\to\infty}$ one gets the identity $2\overline{\eta}_w=\tfrac{1}{2}\overline{\eta}_w\cdot 2\pi i c$. This concludes the proof of~\eqref{eq:fpmpm-res}.
\end{proof}

\subsection{Proof of Theorem~\ref{thm:main-GFF}} Recall from the discussion after  Theorem~\ref{thm:main-GFF} that we want to prove the convergence of height fluctuations to the GFF without proving convergence of $K^{-1}$. We first give an expression of the limits of correlation functions~$H_{{T_m},n}$ via \emph{unknown} subsequential limits~\eqref{eq:fpmpm-def} of the dimer coupling function and then identify these limits up to an unknown holomorphic factor~$\chi:\Omega\to\C$ (see Lemma~\ref{lem:fpmpm-chi} and the forthcoming Section~\ref{sub:discussion} for more details), which turns out to be enough for the proof {of} Theorem~\ref{thm:main-GFF}.

\begin{proposition} \label{prop:dH-limit}
In the setup of Theorem~\ref{thm:main-GFF} and Proposition~\ref{prop:fpmpm}, let {$v_{k,1}^{(m)}\to v_{k,1}$ as $m\to \infty$ for~$k=1,\ldots,n$ and} pairwise distinct points~${v_{1,1},\ldots,v_{n,1}}\in\Omega$, and similarly ${v_{k,2}^{(m)}\to v_{k,2}}$ for pairwise distinct points~${v_{1,2},\ldots,v_{n,2}}\in\Omega$. Then,
\begin{align*}
&\sum_{r_1,\ldots,r_n\in{\{1,2\}}}(-1)^{r_1+\ldots+r_n}  {H_{\cT_m,n}(v^{(m)}_{1,r_1},\ldots,v^{(m)}_{n,r_n})} \notag\\
&\qquad \to\ 4^{-n}\int_{{v_{1,1}}}^{{v_{1,2}}}\!\!\ldots\int_{{v_{n,1}}}^{{v_{n,2}}} \sum_{{s_1,\ldots,s_n\in\{\pm\}}}\det\big[\mathbbm{1}_{j\ne k}f^{[s_j,s_k]}(z_j,z_k)\big]_{j,k=1}^n\cdot {\prod_{k=1}^n} dz^{[s_k]}_k \label{eq:dH-limit}
\end{align*}
as~${m\to\infty}$, where 
$dz^{\scriptscriptstyle [+]}:=dz$, $dz^{\scriptscriptstyle[-]}:=d\overline{z}$. The multiple integral can be evaluated over an arbitrary collection of pairwise non-intersecting paths~$\gamma_k$ linking~${v_{k,1}}$ and~${v_{k,2}}$ ({i.e.}, the integrand is {an exact} differential form in {each of the variables} $z_1,\ldots,z_n$). The convergence is uniform provided that the points~${v_{k,1}}$ remain at a definite distance from each other and from~$\partial\Omega$ and {that} the same is true {for~$v_{k,2}$}.
\end{proposition}
\begin{proof} The proof essentially repeats the classical argument of Kenyon~\cite{kenyon-gff-a,kenyon-gff-b} in our setup. Let $\gamma_k^{{(m)}}$ be a path running over edges of the t-embedding~$\cT_{{m}}$ from~${v_{k,1}^{(m)}}$ to~${v_{k,2}^{(m)}}$ near~$\gamma_k$. (Note that, in general, we do \emph{not} control the total length of these paths as we do not assume that the angles of t-embeddings are uniformly bounded from below). Let~${(b_kw_k)^*}\in\gamma^{{(m)}}_k$ be some edges on these paths. It is well known that the probability that all the dimers ${(b_kw_k)}$, $k=1,\ldots,n$, {are} present in a random dimer cover of~${\G_m}$ can be written as $\det [K^{-1}_{{\cT_m}}(w_j,b_k)]_{j,k=1}^n\cdot\prod_{k=1}^nK_{{\cT_m}}(b_k,w_k)$. {Therefore,}
\begin{align*}
\sum_{r_1,\ldots,r_n\in{\{1,2\}}}&(-1)^{r_1+\ldots+r_n} {H_{\cT_m,n}(v_{1,r_1}^{(m)}\,,\,\ldots\,,\,v_{n,r_n}^{(m)})} \\[-4pt] & =\ \int_{{v_{1,1}^{(m)}}}^{{v_{1,2}^{(m)}}}\!\!\ldots\int_{{v_{n,1}^{(m)}}}^{{v_{n,2}^{(m)}}} \det\big[\mathbbm{1}_{j\ne k}K^{-1}_{{\cT_m}}(w_j,b_k)\big]_{j,k=1}^n\cdot \prod_{k=1}^n {\big(\pm d\cT_m ((b_kw_k)^*)\big)},
\end{align*}
where the {`$\pm$' signs depend} on whether~$b_k$ is to the right or to the left from {the path~$\gamma_k^{(m)}$} so that the increment~${\pm d\cT_m((b_kw_k)^*)}$ is always oriented from~${v_{k,1}^{(m)}}$ to~${v_{k,2}^{(m)}}$. (The diagonal $j=k$ is excluded since we are interested in the correlations of the fluctuations~${\hbar_{\cT_m}=h_{\cT_m}-\mathbb E[h_{\cT_m}]}$ and not {in the functions~$h_{\cT_m}$} themselves.) {In what follows we denote thus oriented edge~$\pm (b_kw_k)^*$ of~$\G_m^*$ by~$e_k$.}

Expanding the determinant one obtains {the expression}
\begin{align*}
\sum_{\sigma\in S_n: \sigma(j)\ne j}(-1)^{\operatorname{sign}(\sigma)} \int_{{v_{1,1}^{(m)}}}^{{v_{1,2}^{(m)}}}\!\!\ldots\int_{{v_{n,1}^{(m)}}}^{{v_{n,2}^{(m)}}} \ \prod_{j=1}^nK^{-1}_{{\cT_m}}(w_j,b_{\sigma(j)})\cdot {\prod_{k=1}^n} {d\cT_m(e_k)},
\end{align*}
where the sum is taken over all permutations~$\sigma$ having no fixed points. Note that, in each of the variables~${e_k=\pm}(b_kw_k)^*$, the integrand is proportional to
\begin{align*}
 K^{-1}_{{\cT_m}}(w_{\sigma^{-1}(k)},b_k)&K^{-1}_{{\cT_m}}(w_k,b_{\sigma(k)}){d\cT_m(e_k)}\\
 &=\ {\eta_{w_{\sigma^{-1}(k)}}\eta_{b_{\sigma(k)}}}F^\tb_{w_{\sigma^{-1}(k)}}(b_k)F^\tw_{b_{\sigma(k)}}(w_k){d\cT_m(e_k)}.
\end{align*}
According to Remark~\ref{rem:FFdT-in-C}, this differential form can be extended from the edges of the t-embedding~${\cT_m}$ into the {complex} plane, as
\begin{align*}
&\tfrac{1}{4}\big(\,F_{w_{\sigma^{-1}(k)}}^\tw(z_k)F_{b_{\sigma(k)}}^\tb(z_k)dz_k+ F_{w_{\sigma^{-1}(k)}}^\tw(z_k)\overline{F_{b_{\sigma(k)}}^\tb(z_k)}\, d{\cO_m}(z_k)\\
 &\hspace{60pt} + \overline{F_{w_{\sigma^{-1}(k)}}^\tw(z_k)}F^\tb_{b_{\sigma(k)}}(z_k)\, d\overline{{\cO_m}(z_k)}+\overline{F_{w_{\sigma^{-1}(k)}}^\tw(z_k)}\overline{F_{b_{\sigma(k)}}^\tb(z_k)}d\overline{z}_k\,\big).
\end{align*}
Moreover, using the functions~$\F\pm\pm$ introduced in Proposition~\ref{prop:Fpmpm-def}, one can write this extension in all the variables~$z_k$, $k=1,\ldots,n$, simultaneously. {Namely,} recall that
\[
K^{-1}(w_j,b_{\sigma(j)})=\tfrac{1}{4}\big(\,\F++ + \eta_{b_{\sigma(j)}}^2\F+- + \eta_{w_j}^2\F-+ + \eta_{b_{\sigma(j)}}^2\eta_{w_j}^2\F--\,\big)(\bc_j,\wc_{\sigma(j)})
\]
if~$w_j\sim\bc_j$ and~$b_{\sigma(j)}\sim\wc_{\sigma(j)}$ (e.g., one can take~$\bc_j:=b_j$ and~$\wc_{\sigma(j)}:=w_{\sigma(j)}$) and denote
\[
\begin{array}{ll}
d\cTT++:=d\cT, & d\cTT+-:=d\cO=\eta_{w}^2d\cT,\\[2pt]
d\cTT-+:=d\overline{\cT}=\eta_b^2\eta_w^2d\cT,\quad & d\cTT--:=d\overline{\cO}=\eta_b^2d\cT.
\end{array}
\]
Then,
\begin{align*}
\prod_{j=1}^nK^{-1}(w_j,b_{\sigma(j)})\,\cdot &{\prod_{k=1}^n} {d\cT_m(e_k)}\\[-4pt]
& = \ 4^{-n}\!\!\sum_{p_k,q_k=\pm}\ \prod_{j=1}^n F^{[p_j,q_{\sigma(j)}]}(b_j,w_{\sigma(j)})\cdot\prod_{k=1}^n d\cT^{[q_k]}_{{m},[p_kq_k]}{(e_k)}\,,
\end{align*}
where the sum is taken over all~$2^{2n}$ possible combinations of signs~$p_k$ and~$q_k$.
Due to this extension from edges of~${\cT_m}$ to~$\C$, the paths~$\gamma_k^{{(m)}}$ can now be assumed to \emph{coincide} with~$\gamma_k$ except near the endpoints and, in particular, to have uniformly {(as~$m\to\infty$)} bounded lengths.

It follows from the `small origami' assumption that the integrals against~$d\cT^{[q_k]}_{{m},[-]}$ vanish in the limit~${m\to\infty}$, thus only terms corresponding to \mbox{$s_k:=p_k=q_k\in\{\pm\}$} survive; {see the proof of Proposition~\ref{prop:f-closed-forms} for a similar statement}. The convergence~\eqref{eq:fpmpm-def} allows us to conclude that, {as $m\to\infty$},
\begin{align*}
 & \sum_{r_1,\ldots,r_n\in{\{1,2\}}}(-1)^{r_1+\ldots+r_n} {H_{\cT_m,n}(v_{1,r_1}^{(m)}\,,\,\ldots\,,\,v_{n,r_n}^{(m)})} \ \to\\
&4^{-n}\!\!\!\!\!\!\sum_{\sigma\in S_n: \sigma(j)\ne j} (-1)^{\operatorname{sign}(\sigma)}\int_{{v_{1,1}}}^{{v_{1,2}}}\!\!\ldots\int_{{v_{n,1}}}^{{v_{n,2}}}\!\!
\sum_{{s_1,\ldots,s_n\in\{\pm\}}}\ \prod_{j=1}^n f^{[s_j,s_{\sigma(j)}]}(z_j,z_{\sigma(j)})\prod_{k=1}^n dz^{[s_k]}_k
\end{align*}
Interchanging the summations over~$\sigma$ and over the signs~$s_k$ gives the claim.
\end{proof}

\begin{proof}[Proof of Theorem~\ref{thm:main-GFF}]  Given {t-embeddings~$\cT_m$ approximating a continuous domain~$\Omega$,} consider \emph{subsequential} limits~\eqref{eq:fpmpm-def} of the functions~$\F\pm\pm_{{\cT_m}}$ and define
\begin{equation}
\label{eq:x-An=fpmpm}
\mathcal{A}_n(z_1,\ldots,z_n)\ :=\ 4^{-n}\sum_{s_1,\ldots,s_n\in\{\pm\}}\det\big[\mathbbm{1}_{j\ne k}f^{[s_j,s_k]}(z_j,z_k)\big]_{j,k=1}^n\cdot {\prod_{k=1}^n} dz^{[s_k]}_k\,.
\end{equation}
Proposition~\ref{prop:dH-limit}, in particular, implies that~$\mathcal{A}_n(z_1,\ldots,z_n)$ is an exact differential form in each of the arguments~$z_1,\ldots,z_n$. Let~$\widetilde{v}_1$,\ldots,$\widetilde{v}_n$ be auxiliary points close to the boundary of~$\Omega$ (but {lying} at a definite distance from each other). {Let us now deduce from the assumption (III) that} the function
\begin{equation}
\label{eq:x-hn-def}
h_n(v_1,\ldots,v_n)\ :=\ \lim\nolimits_{\widetilde{v}_1\to\partial\Omega}\ldots\lim\nolimits_{\widetilde{v}_n\to\partial\Omega} \int_{\widetilde{v_1}}^{v_1}\!\!\ldots\int_{\widetilde{v}_n}^{v_n}\mathcal{A}_n(z_1,\ldots,z_n).
\end{equation}
is well-defined 
(i.e., {that the iterated limit exists and} does not depend on the way in which the auxiliary points $\widetilde{v}_n,\ldots,\widetilde{v}_1$ consecutively approach the boundary of~$\Omega$) and that, {as~$m\to\infty$},
\begin{equation}
\label{eq:x-Hn-to-hn}
H_{{\cT_m,n}}(v_1^{{(m)}},\ldots,v_n^{{(m)}})\ \to\ h_n(v_1,\ldots,v_n)\ \ \text{if}\ \ {v^{(m)}_k\to v_k;}
\end{equation}
uniformly over~$v_1,\ldots,v_n$ at a definite distance from $\partial\Omega$ and from each other. {Denote $v_{k,1}:=\widetilde{v}_k$, $v_{k,2}:=v_k$ and similarly for the vertices of~$\cT_m$ that approximate these points.} It follows from {Proposition~\ref{prop:dH-limit}} and from the assumption (III) applied to all terms of the sum containing~${v_{n,1}^{(m)}=\widetilde{v}_n^{(m)}}$ that
\begin{align}
\sum_{r_1,\ldots,r_{n-1}\in{\{1,2\}}}(-1)^{r_1+\ldots+r_{n-1}} &{H_{\cT_m,n}(v_{1,r_1}^{(m)},\ldots,v_{n-1,r_{n-1}}^{(m)},v_n^{(m)})}\notag \\[-4pt]
=\ \int_{\widetilde{v_1}}^{v_1}\!\!\ldots&\int_{\widetilde{v}_n}^{v_n}\mathcal{A}_n(z_1,\ldots,z_n) +o_{\widetilde{v}_n\to\partial\Omega}(1)+o_{{m\to\infty}}(1)\,,
\label{eq:x-dH-conv}
\end{align}
where the $o_{\widetilde{v}_n\to\partial\Omega}(1)$ error term in the right-hand side is \emph{uniform in~${m}$} provided that all other auxiliary points~$\widetilde{v}_1,\ldots,\widetilde{v}_{n-1}$ stay in the bulk of~$\Omega$ (and that {$m$ is big} enough depending on~$\widetilde{v}_n$). Since the left hand-side does not depend on~$\widetilde{v}_n$ and the main term in the right-hand side does not depend on~${m}$, this implies that
\begin{equation}
\label{eq:x-hn-Dirichlet}
\int_{\widetilde{v_1}}^{v_1}\!\!\ldots\int_{\widetilde{v}_n}^{\widetilde{v}'_n}\mathcal{A}_n(z_1,\ldots,z_n)\ \to\ 0\ \ \text{as both}\ \ \widetilde{v}_n,\widetilde{v}'_n\to\partial\Omega.
\end{equation}

Therefore, the primitive $\int_{\widetilde{v_1}}^{v_1}\!\!\ldots\int_{\widetilde{v}_n}^{v_n}\mathcal{A}_n(z_1,\ldots,z_n)$ has a well-defined limit as \mbox{$\widetilde{v}_n\to\partial\Omega$}. {Moreover, choosing first~$\widetilde{v}_n$ close enough to~$\partial\Omega$ and then~$m$ big enough (depending on~$\widetilde{v}_n$) in~\eqref{eq:x-dH-conv}, we see that}
\begin{align*}
\sum_{r_1,\ldots,r_{n-1}\in{\{1,2\}}}(-1)^{r_1+\ldots+r_{n-1}} &{H_{\cT_m,n}(v_{1,r_1}^{(m)},\ldots,v_{n-1,r_{n-1}}^{(m)},v_n^{(m)})}\\[-4pt]
&=\ \lim_{\widetilde{v}_n\to\partial\Omega}\int_{\widetilde{v_1}}^{v_1}\!\!\ldots\int_{\widetilde{v}_n}^{v_n}\mathcal{A}_n(z_1,\ldots,z_n) +o_{{m\to\infty}}(1).
\end{align*}
Applying the same argument as~$\widetilde{v}_{n-1}\to\partial\Omega,\ldots,\widetilde{v}_1\to\partial\Omega$ one obtains the existence of the {iterated} limit~\eqref{eq:x-hn-def} and the convergence~\eqref{eq:x-Hn-to-hn}. {Moreover, note that both these convergences are uniform if the points~$v_1,\ldots,v_n$ remain at a definite distance from~$\partial\Omega$ and from each other.}

It remains to identify the functions~$h_n$. {Due to the holomorphicity/anti-holo\-morphicity of the functions~$\f{\pm\pm}(\cdot,\cdot)$, for each fixed reference points~$\widetilde{v}_1,\ldots,\widetilde{v}_n$, the function $\int_{\widetilde{v}_1}^{v_1}\ldots\int_{\widetilde{v}_n}^{v_n}\mathcal{A}_n(z_1,\ldots,z_n)$ is continuous and harmonic in each of the variables~$v_1,\ldots,v_n$ away from the reference points and from the diagonals~$v_p=v_q$, $p\ne q$. Since the iterated limit in~\eqref{eq:x-hn-def} is uniform on compacts, the functions~$h_n$ are also harmonic in each of~$v_1,\ldots,v_n$.} It is {also} easy to see that $h_n$ is symmetric in~$v_1,\ldots,v_n$ since the left-hand side in~\eqref{eq:x-Hn-to-hn} is symmetric. {Moreover, Assumption~(III) implies that $h_n$ satisfies Dirichlet boundary conditions as one of the variables (e.g.,~$v_n$) approaches~$\partial\Omega$.} We start with two particular cases $n=2$ and~$n=3$.

Let~$n=2$. It follows from Proposition~\ref{prop:fpmpm}(iii) {that
\[
\f{++}(z_1,z_2)\ =\ \frac{2}{\pi i}\frac{1}{z_2-z_1}+\f{++}_\mathrm{reg}(z_1,z_2),
\]
where the function~$\f{++}_\mathrm{reg}$ is continuous in~$(z_1,z_2)$ away from the diagonal and holomorphic with a removable singularity as a function of each of the variables~$z_1,z_2$. Due to Hartogs's lemma, this implies that~$\f{++}_\mathrm{reg}$ is continuous and holomorphic as a function of~$(z_1,z_2)$; in particular,~$\f{++}_\mathrm{reg}(z_2,z_1)-\f{++}_\mathrm{reg}(z_1,z_2)=O(|z_2-z_1|)$. Hence,
\[
\f{++}(z_1,z_2)\f{++}(z_2,z_1)\ =\ \frac{4}{\pi^2}\frac{1}{(z_2-z_1)^2}+O(1)
\]
and}
\[
\mathcal{A}_2(z_1,z_2)\ =\ -\frac{1}{2\pi^2}\Re\biggl(\,\frac{dz_1dz_2}{(z_2-z_1)^2}\biggr)+{O(1)}
\]
provided that~$z_1,z_2$ stay in the bulk of~$\Omega$ {(but not necessarily far from each other). This means that the function} $h_2(v_1,\,\cdot\,)$ satisfies the asymptotics
\[
h_2(v_1,v_2)\ =\ (-2\pi^2)^{-1}\log|v_2-v_1|+O(1)\ \ \text{as}\ \ v_2\to v_1
\]
and thus can be identified with the Green function~$h_2(v_1,v_2)=\pi^{-1}G_\Omega(v_1,v_2)$ due to the harmonicity and Dirichlet boundary conditions.

Let~$n=3$. {A similar consideration shows that in this case the differential form
\begin{align*}
\textstyle \mathcal{A}_3(z_1,z_2,z_3)\ =\ \frac{1}{64}\sum_{s_1,s_2,s_3\in\{\pm\}}\bigl(\f{s_1,s_2}(z_1,z_2)\f{s_2,s_3}(z_2,z_3)\f{s_3,s_1}(z_3,z_1)&\\
\ +\f{s_1,s_3}(z_1,z_3)\f{s_3,s_2}(z_3,z_2)\f{s_2,s_1}(z_2,z_1)\bigr)dz_1^{[s_1]}dz_2^{[s_2]}&dz_3^{[s_3]}
\end{align*}
does \emph{not} have singularities at all: for instance, for both~$s_1\in\{\pm\}$ one has
\[
\f{s_1,+}(z_1,z_2)\f{++}(z_2,z_3)\f{+,s_1}(z_3,z_1)\ =\ \frac{2}{\pi i}\frac{\f{s_1,+}(z_1,z_2)\f{+,s_1}(z_2,z_1)}{z_3-z_2}+O(1)
\]
as~$z_3\to z_2\ne z_1$ and the same pole appears with the opposite sign in the second term $\f{s_1,+}(z_1,z_3)\f{++}(z_3,z_2)\f{+,s_1}(z_2,z_1)$ contributing to~$\mathcal{A}_3$.

Therefore, the function} $h_3(v_1,v_2,\,\cdot\,)$ does not have a singularity as~$v_3\to v_{{2}}$ or as~$v_3\to v_{{1}}$.
Due to the harmonicity and Dirichlet boundary conditions this yields~$h_3(v_1,v_2,v_3)=0$ for all~$v_1,v_2,v_3\in\Omega$.

\smallskip

The rest of the proof of Theorem~\ref{thm:main-GFF} boils down to the following simple lemma.
\begin{lemma}\label{lem:fpmpm-chi}
Let differential forms~$\mathcal{A}_2,\mathcal{A}_3$ be defined by~\eqref{eq:x-An=fpmpm}, where functions $\f{\pm\pm}(z_1,z_2)$ are holomorphic/anti-{holomorphic} in each of the variables and satisfy the relations $\f{--}=\of{++}$, $\f{+-}=\of{-+}$. If $\mathcal{A}_2=\pi^{-1}d_{v_1}d_{v_2}G_\Omega(v_1,v_2)$ and $\mathcal{A}_3=0$, then there exists a holomorphic function~$\chi:\Omega\to{\C\smallsetminus\{0\}}$ such that
\begin{align*}
\f{++}(z_1,z_2)\ &=\ \frac{\chi(z_1)}{\chi(z_2)}\cdot \f{++}_0(z_1,z_2),\quad \f{++}_0(z_1,z_2):=\frac{2}{\pi i}\cdot \frac{\phi'_\Omega(z_1)^{\frac 12}\phi'_\Omega(z_2)^{\frac 12}}{\phi_\Omega(z_2)-\phi_\Omega(z_1)},\\
\f{-+}(z_1,z_2)\ &=\ \frac{\overline{\chi(z_1)}}{\chi(z_2)}\cdot \f{-+}_0(z_1,z_2),\quad \f{-+}_0(z_1,z_2):=\frac{2}{\pi i}\cdot \frac{\overline{\phi'_\Omega(z_1)}^{\frac12}\phi'_\Omega(z_2)^{\frac 12}}{\phi_\Omega(z_2)-\overline{\phi_\Omega(z_1)}},
\end{align*}
where~$\phi_\Omega:\Omega\to\mathbb{H}$ is a conformal uniformization~$\Omega$ onto the upper half-plane~$\mathbb{H}$.
\end{lemma}
\begin{proof} {Denote} $g^{\scriptscriptstyle [\pm\pm]}(z_1,z_2):=\f{\pm\pm}(z_1,z_2)/\f{\pm\pm}_0(z_1,z_2)$. {As we have} $\f{--}_0=\of{++}_0$ and $\f{+-}_0=\of{-+}$, {the functions~$g^{\scriptscriptstyle [\pm\pm]}$ satisfy the same relations. Note also that these functions do not have singularities as the simple pole of~$\f{++}$ at~$z_1=z_2$ cancels out by the same pole of~$\f{++}_0$, which also implies that~$g^{\scriptscriptstyle [++]}(z,z)=1$.} Since
\[
\pi^{-1}d_{v_1}d_{v_2}G_\Omega(v_1,v_2)\ =\ -\frac{1}{16}\sum_{s_1,s_2\in\{\pm\}}f^{[s_1,s_2]}_0(z_1,z_2)f^{[s_2,s_1]}_0(z_2,z_1)dz_1^{[s_1]}dz_2^{[s_2]}\,,
\]
{one} can rewrite the identity~$\mathcal{A}_2=\pi^{-1}d_{v_1}d_{v_2}G_\Omega(v_1,v_2)$ as
\begin{equation}
\label{eq:x-g12=g21}
g^{[s_1,s_2]}(z_1,z_2)g^{[s_2,s_1]}(z_2,z_1)=1,
\end{equation}
{for all $s_1,s_2\in\{\pm\}$ and $z_1,z_2\in\Omega$. In particular, $g^{\scriptscriptstyle[\pm\pm]}(z_1,z_2)\ne 0$ everywhere in~$\Omega\times\Omega$.}
Similarly, expanding the definition~\eqref{eq:x-An=fpmpm} of the differential form~$\mathcal{A}_3=0$ {as above} one sees that
\begin{align*}
g^{[s_1,s_2]}(z_1,z_2)g^{[s_2,s_3]}(z_2,z_3)&g^{[s_3,s_1]}(z_3,z_1)\\
&=\ g^{[s_1,s_3]}(z_1,z_3)g^{[s_3,s_2]}(z_3,z_2)g^{[s_2,s_1]}(z_2,z_1)
\end{align*}
for all $s_1,s_2,s_3\in\{\pm\}$ and $z_1,z_2,z_3\in\Omega$. {Due to~\eqref{eq:x-g12=g21}, the two sides of this identity are inverse to each other, which means that they do not depend on~$z_1,z_2,z_3$ and are equal to~$\pm 1$. Moreover, substituting~$z_p=z_q$ if~$s_p=s_q$ one sees that only~$+1$ is possible and hence
\[
g^{[s_1,s_3]}(z_1,z_3)\ =\ g^{[s_1,s_2]}(z_1,z_2)g^{[s_2,s_3]}(z_2,z_3)
\]
for all $s_1,s_2,s_3\in\{\pm\}$ and $z_1,z_2,z_3\in\Omega$. Let us now fix a point~$z_0\in\Omega$ and denote
\[
\chi(z)\ :=\ (g^{\scriptscriptstyle[+-]}(z_0,z_0))^{1/2}\cdot g^{\scriptscriptstyle[++]}(z,z_0)\ =\ (g^{\scriptscriptstyle[+-]}(z_0,z_0))^{-1/2}\cdot g^{\scriptscriptstyle[+-]}(z,z_0)\,;
\]
note that~$|g^{\scriptscriptstyle[+-]}(z_0,z_0)|=1$ due to~\eqref{eq:x-g12=g21} and since~$g^{\scriptscriptstyle[-+]}=\overline{g}^{\scriptscriptstyle[+-]}$. Therefore,
\[
g^{\scriptscriptstyle[++]}(z_1,z_2)\,=\, \frac{g^{\scriptscriptstyle[++]}(z_1,z_0)}{g^{\scriptscriptstyle[++]}(z_2,z_0)}\,=\, \frac{\chi(z_1)}{\chi(z_2)}\quad \text{and}\quad
g^{\scriptscriptstyle[-+]}(z_1,z_2)\,=\, \frac{g^{\scriptscriptstyle[-+]}(z_1,z_0)}{g^{\scriptscriptstyle[++]}(z_2,z_0)}\,=\, \frac{\overline{\chi(z_1)}}{\chi(z_2)}
\]
as required.}
\end{proof}

We are now ready to conclude the proof of Theorem~\ref{thm:main-GFF}. It is easy to see that the differential forms~$\mathcal{A}_n$ (see~\eqref{eq:x-An=fpmpm}) {do} \emph{not} depend on the factors~$\chi(\cdot)$ in the representation of functions~$\f{\pm\pm}$ provided by Lemma~\ref{lem:fpmpm-chi}: all these factors simply cancel out when one considers a product~$\prod_{j=1}^nf^{[s_j,s_{\sigma(j)}]}(z_j,z_{\sigma(j)})$. Similarly, the global~$\pm$ sign in the expression of~$\f{\pm\pm}$ can only affect the sign of~$\mathcal{A}_n$ for odd~$n$. In other words, even though we are unable to identify the functions~$\f{\pm\pm}$ themselves, Lemma~\ref{lem:fpmpm-chi} provides enough information to identify the differential forms~$\mathcal{A}_n$. It is well known (e.g., see~\cite[Proposition~3.2]{kenyon-gff-b} and~\cite[Eq.~(12.53)]{yellow-book}) that
\[
\mathcal{A}_{n}\ =\ \pi^{-{n/2}}\cdot d_{v_1}\ldots d_{v_{n}}G_{\Omega,n}(v_1,\ldots,v_{n})\,,
\]
where {$G_{\Omega,2k+1}=0$} and the correlation functions~${G_{\Omega,2k}}$ of the Gaussian Free Field in~$\Omega$ {are given by~\eqref{eq:Green-correlations}.}
Thus, we have the identity
\[
[d_{v_1}\ldots d_{v_n}(h_n-G_{\Omega,n})](z_1,\ldots,z_n)=0,\quad z_1,\ldots,z_n\in\Omega,
\]
and it remains to note that the Dirichlet boundary conditions of the harmonic function \mbox{$h_n-G_{\Omega,n}$} as~$v_n\to\partial\Omega$ yield (via the Harnack principle) the same boundary conditions, e.g., for the gradient~$d_{v_1}\ldots d_{v_{n-1}}(h_n-G_{\Omega,n})$ as~$v_n\to\partial\Omega$. By induction, this allows us to conclude that
\[
[d_{v_1}\ldots d_{v_{n-k}}(h_n-G_{\Omega,n})](z_1,\ldots,z_{n-k},v_{n-k+1},\ldots,v_n)=0
\]
for all~$k=1,2,\ldots,n$. The proof is complete.
\end{proof}

\subsection{Discussion}\label{sub:discussion} We now briefly discuss how several cases from the existing literature fit our setup; with a particular emphasis on the explicit expression for functions~$\f{\pm\pm}$, the scaling limits of~$\F\pm\pm_{{\cT_m}}$ as~${m\to\infty}$, available in these cases. In all these situations one obtains~$\f{\pm\pm}$ explicitly by solving some conformally covariant boundary value problem, so these functions are also conformally covariant. Recall that~$\f{--}(z_1,z_2)=\overline{\f{++}(z_1,z_2)}$ and~$\f{+-}(z_1,z_2)=\overline{\f{-+}(z_1,z_2)}$.

For classical Temperleyan domains on~$\mathbb{Z}^2$~\cite{kenyon-gff-a,kenyon-gff-b} (see also~\cite{zhongyang-li}), as well as for Temperleyan-type domains coming from T-graphs~\cite{kenyon-honeycomb} (see also {\cite{laslier-21} and}~\cite{BLR1,BLRnote}), the functions~$\f{\pm\pm}(z_1,z_2)$ are conformally invariant in the first variable and conformally covariant with exponent~$1$ in the second. This implies the explicit expressions
\begin{equation}
\label{eq:x-fpmpm-Tempreleyan}
\f{++}(z_1,z_2)\,=\,\frac{2}{\pi i}\frac{\phi_{\Omega}'(z_2)}{\phi_{\Omega}(z_2)\!-\!\phi_{\Omega}(z_1)}\,,\quad \f{-+}(z_1,z_2)\,=\,\frac{2}{\pi i}\frac{\phi_{\Omega}'(z_2)}{\phi_{\Omega}(z_2)\!-\!\overline{\phi_{\Omega}(z_1)}}\,,
\end{equation}
where $\phi_{\Omega}:\Omega\to \mathbb{H}$ denotes a conformal uniformization of~$\Omega$ onto the upper half-plane sending the root point~$a\in\partial\Omega$ of the Temperley correspondence to~$\infty$. Equivalently, for each~$\eta_w\in\mathbb{T}$,
\begin{equation}\label{eq:x-bc-cov-1}
\Re \biggl[\int (\overline{\eta}_w\f{++}(z_1,z_2)+\eta_w\f{-+}(z_1,z_2))dz_2\biggr]\,=\,0\ \ \text{if $z_2\in\partial\Omega\smallsetminus\{a\}$.}
\end{equation}
This boundary condition is inherited from the Dirichlet boundary conditions for the discrete primitives~$\rI_\R[F_w]$ of the dimer observables (which are discrete harmonic functions on the corresponding T-graph), the particular feature which allows one to prove the convergence theorem in this case; cf. \cite{kenyon-honeycomb,laslier-21}. 
It is worth mentioning that the Temperleyan case is \emph{non-symmetric} with respect to changing the roles of black and white faces. In terms of the first variable, one observes that, for each~$\eta_b\in\mathbb{T}$,
\begin{equation}\label{eq:x-bc-cov-0}
\Im\big(\overline{\eta}_b\f{++}(z_1,z_2)+\eta_b\f{+-}(z_1,z_2)\big)\,=\,0\ \ \text{if $z_1\in\partial\Omega\smallsetminus\{a\}$.}
\end{equation}
Both boundary conditions~\eqref{eq:x-bc-cov-0} and~\eqref{eq:x-bc-cov-1} are clearly visible in classical Temperleyan domains on~$\mathbb{Z}^2$ and can be used to prove the convergence of observables, the original argument of Kenyon~\cite{kenyon-gff-a,kenyon-gff-b} went through the simpler Dirichlet boundary conditions~\eqref{eq:x-bc-cov-0}. However, in the {Tempereleyan}-like setup coming from T-graphs one of the variables plays a distinguished role:~\eqref{eq:x-bc-cov-1} is tautological while~\eqref{eq:x-bc-cov-0} does not admit a straightforward interpretation in discrete to the best of our knowledge; it would be interesting to find one.

The second case to discuss is the so-called white-piecewise Temperleyan domains on~$\mathbb{Z}^2$ studied in~\cite{russkikh-t} (we swap the colors as compared to~\cite{russkikh-t} to fit the {preceding} discussion). This setup is conceptually similar to the classical Temperleyan one, except that one gets more complicated limits
\begin{align*}
\f{\pm+}(z_1,z_2)\
=\ \frac{2}{\pi i}&\frac{\phi_{\Omega}'(z_2)}{\phi_{\Omega}(z_2)-\phi_{\Omega}^{\scriptscriptstyle [\pm]}(z_1)}\\
&\times\  \frac{\prod_{k=1}^{m-1}(\phi_{\Omega}^{\scriptscriptstyle [+]}(z_2)-\phi_{\Omega}(v_k'))^{\frac12}}{\prod_{k=1}^{m+1}(\phi_{\Omega}^{\scriptscriptstyle [+]}(z_2)-\phi_{\Omega}(v_k))^{\frac12}}\cdot \frac{\prod_{k=1}^{m+1}(\phi_{\Omega}^{\scriptscriptstyle [\pm]}(z_1)-\phi_{\Omega}(v_k))^{\frac12}}{\prod_{k=1}^{m-1}(\phi_{\Omega}^{\scriptscriptstyle [\pm]}(z_1)-\phi_{\Omega}(v_k'))^{\frac12}}\,
\end{align*}
instead of~\eqref{eq:x-fpmpm-Tempreleyan}, where~$\phi_{\Omega}^{\scriptscriptstyle[+]}(z_1):=\phi_{\Omega}(z_1)$,~$\phi_{\Omega}^{\scriptscriptstyle[-]}(z_1):=\overline{\phi_{\Omega}(z_1)}$ and~$v_k,v_k'\in\partial\Omega$ are the boundary points at which boundary conditions change. Again, the covariance exponents of the functions~$\f{\pm\pm}$ are~$(1,0)$ and the two colors play asymmetric roles. On each of the boundary arcs between points~$v_k,v_k'$, either the real or the imaginary part of the function
\[
\overline{\eta}_b\f{++}(z_1,z_2)+\eta_b\f{+-}(z_1,z_2)\,=\,\overline{\eta}_b\f{++}(z_1,z_2)+\eta_b\overline{\f{-+}(z_1,z_2)}
\]
satisfies Dirichlet conditions and one can use this boundary value problem to pass to the limit.

The third example is the so-called hedgehog domains studied in~\cite{russkikh-h}. For these domains one has
\[
\f{++}(z_1,z_2)\,=\,\frac{2}{\pi i} \frac{\phi_{\Omega}'(z_1)^\frac12\phi_{\Omega}'(z_2)^\frac12}{\phi_{\Omega}(z_2)-\phi_{\Omega}(z_1)}\,,\quad \f{-+}(z_1,z_2)\,=\,\frac{2}{\pi i} \frac{\overline{\phi_{\Omega}'(z_1)}{}^\frac12\phi_{\Omega}'(z_1)^\frac12}{\phi_{\Omega}(z_2)-\overline{\phi_{\Omega}(z_1)}}\,,
\]
the colors play symmetric roles, and the boundary conditions are the Ising-type ones:
\[\begin{array}{ll}
\Im[\int (f(z_1,z_2))^2 dz_2]=0\ \text{at}\  \partial\Omega &\text{if}\ \ f(z_1,z_2)=\overline{\eta}_w\f{++}(z_1,z_2)+\eta_w\f{-+}(z_1,z_2);\\[4pt]
\Im[\int (f(z_1,z_2))^2 dz_1]=0\ \text{at}\  \partial\Omega & \text{if}\ \ f(z_1,z_2)=\overline{\eta}_b\f{++}(z_1,z_2)+\eta_b\f{+-}(z_1,z_2).
\end{array}
\]
Already this set of conformally covariant examples clearly illustrates the fragility of the functions~$\f{\pm\pm}$ with respect to the change of the microscopic properties of the boundary. (Note that this effect has nothing to do with limit shapes: on {the square grid,} all three examples listed above lead to {an asymptotically} horizontal profile of the height function.)

In general, one {can} easily invent a setup in which the functions~$\f{\pm\pm}$ are \emph{not} conformally covariant: for instance, piecewise Temperleyan domains on~$\mathbb{Z}^2$ with arcs satisfying the Temperley condition with respect to different colors provide a relevant example. Moreover, in any reasonably general situation -- e.g., domains on the square grid composed of $2\times 2$ blocks, which again leads to the flat horizontal limit profile of the height function -- one should not expect to see \emph{any} particular boundary conditions for dimers observables. This is one of the reasons of why we believe that the framework of Theorem~\ref{thm:main-GFF}, which \emph{bypasses} the identification of the limits of the functions~$\f{\pm\pm}$, is a right way to treat the dimer model in {a} reasonably general situation.


To summarize, the 
limits of dimer observables~$F_w^\tw(\cdot)$ and ~$F_b^\tb(\cdot)$ might depend on a subsequence but their boundary conditions are always dual to each other so that the primitives $\int \Re[F_w^\tw(z)F_b^\tb(z)dz]$ satisfy the Dirichlet ones. In a general setup, it does not make much sense to consider limits of~$F_w^\tw$ and limits of~$F_b^\tb$ separately. However, being combined together these limits form a stable object, whilst each of them alone might have \emph{no} reasonable interpretation at all due to a great freedom in the choice of the unknown holomorphic factor~$\chi$ in Lemma~\ref{lem:fpmpm-chi}. The three particular cases discussed above are clearly very special in this respect.

{\section{T-holomorphicity and other discretizations of complex analysis}\label{sec:appendix2}}

\newcommand{\vl}{w_{\raisebox{-1pt}{\tiny$\lambda$}}}
\newcommand{\vlbr}{w_{\raisebox{-1pt}{\tiny $\overline\lambda$}}}
\newcommand{\uR}{b_{\raisebox{-1pt}{\tiny R}}}
\newcommand{\uI}{b_{\raisebox{-1pt}{\tiny I}}}
\renewcommand{\wc}{u^\tw}
\renewcommand{\bc}{u^\tb}
\def\spl{\mathrm{spl}}

\def\cH{\mathcal{H}}

The goal of this {last section} is to clarify the links between the setup of t-holomorphic functions on t-embeddings and more standard discretizations of complex analysis. Namely, in Section~\ref{sub:appendix-ort} we show that the setup of \emph{orthodiagonal embeddings} is a particular case of t-embeddings. Moreover, in Section~\ref{subsub:appendix-square-tilings} we indicate a link between \emph{square tilings} (or, more generally, rectangular ones) and t-embeddings. 
Further, in Section~\ref{sub:appendix-s-emb} we focus on another particular case: that of \emph{s-embeddings} appearing in the planar Ising model context, cf.~\cite[Section~7]{KLRR} and~\cite[Section~2.3]{chelkak-s-emb}. The standard framework of \emph{isoradial grids}, which can be viewed both as s-embeddings and as orthodiagonal ones, is discussed in Section~\ref{sub:appendix-isoradial}. Finally, in Section~\ref{sub:appendix-regular} we briefly explain how the two known discrete complex analysis frameworks on the \emph{square} and the \emph{honeycomb lattices} fit our setup.

\smallskip

\subsection{Orthodiagonal embeddings and square tilings}\label{sub:appendix-ort}

Let $\Gamma$ be a planar graph, denote its dual by $\Gamma^*$ and {let~$\Lambda$ be the `diamond graph' of~$\Gamma$: the vertices of~$\Lambda$ are those of~$\Gamma$ and of~$\Gamma^*$, the edges of~$\Gamma$ connect a vertex of~$\Gamma$ with an adjacent vertex of~$\Gamma^*$ so that all faces of~$\Lambda$ have degree four and correspond to edges of~$\Gamma$.}
Let $\cH_{\mathrm{ort}}:\Lambda\to\C$ be an orthodiagonal embedding, i.e., a graph  $\Gamma$ is embedded to $\mathbb{C}$ together with its dual~$\Gamma^*$ so that all edges are straight segments and dual edges $\cH_{\mathrm{ort}}(e)$, $\cH_{\mathrm{ort}}(e^*)$ intersect and are orthogonal to each other for each edge~$e$ of~$\Gamma$; see Fig.~\ref{fig:ort}. In many questions and results the intersection condition can be weakened, instead one often assumes that the images of faces of~$\Lambda$ under~$\cH_\mathrm{ort}$ are proper (though not necessarily convex) quads. While keeping the forthcoming exposition essentially self-contained, we refer the reader to~\cite{gurel-jerison-nachmias} and references therein for more background on the discrete complex analysis in this setup.

Given an orthodiagonal embedding $\cH_{\mathrm{ort}}$ one defines conductances $c_{{(vv')}}$ on edges of $\Gamma$ by
\begin{equation}
\label{eq:c-ortho-def}
c_{(vv')}\ :=\ \frac{\left|\cH_{\mathrm{ort}}(b_+^*)-\cH_{\mathrm{ort}}(b_-^*)\right|}
{\left|\cH_{\mathrm{ort}}(b_+)-\cH_{\mathrm{ort}}(b_-)\right|}\,,\qquad c_{(vv')^*}\ :=\ c_{(vv')}^{-1}\,,
\end{equation}
where we use the notation $(vv')=(b_-b_+)$ and $(vv')^*=(b_-^*b_+^*)$ for an edge of~$\Gamma$ and its dual; below we often identify vertices of `abstract' graphs with their positions in~$\C$ under~$\cH_\mathrm{ort}$ and simply write, e.g., $b_+^*$ instead of~$\cH_\mathrm{ort}(b_+^*)$. Further, let {a dimer graph $G_D$ be the `diamond graph' of~$\Lambda$: its} vertex set consists of {the} union of $\Lambda$ (two types of black vertices) and quads~$\diamondsuit:=\Lambda^*$ (white vertices), where each vertex $\wc\in\diamondsuit$ is represented in~$\cH_\mathrm{ort}$ by the intersection of the corresponding edge $(b_-b_+)$ of $\Gamma$ and its dual edge $(b_-^*b_+^*)$.

Note that the medial graph of $\cH_{\mathrm{ort}}(\Lambda)$ {forms} a t-embedding $\cT$ of $G_D^*,$ see Fig.~\ref{fig:ort}. Since all white faces of $\cT(G^*_D)$ are rectangles, an origami square root function~$\eta_b$ takes only two values on black faces depending on whether such a face corresponds to a vertex of $\Gamma$ or to that of $\Gamma^*$. If $\eta_{b}$ equals to $\pm 1$ on $\Gamma$, then it is $\pm i$ on~$\Gamma^*$ (recall that~$\eta_b$ is defined up to the sign).

\begin{figure}
\center{\includegraphics[width=0.64\textwidth]{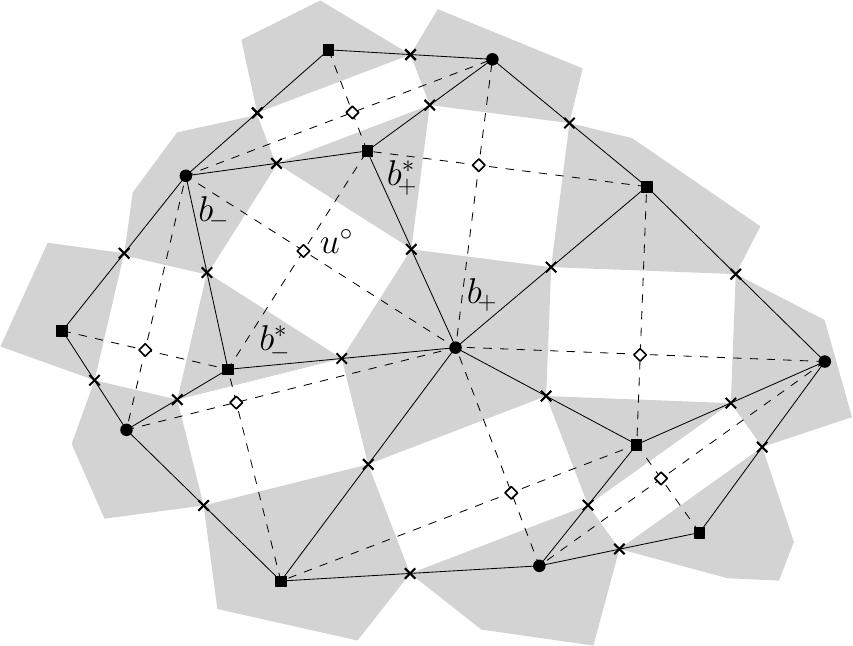}}
\caption{A portion of an orthodiagonal embedding $\cH_{\mathrm{ort}}(\Lambda)$, the corresponding dimer graph $G_D$ (dashed) and {its} t-embedding. {There are two types of black vertices in~$G_D$: black discs are vertices of~$\Gamma$ and black squares are those of~$\Gamma^*$. White vertices of~$G_D$ are}  {vertices of~$\diamondsuit$ drawn as white rhombi in the picture.}\label{fig:ort}}
\end{figure}

\begin{remark}
Given an edge-weighted graph $(\Gamma, c)$ one defines weights $\chi$ on edges of the corresponding dimer graph $G_D$ by
\begin{align*}
\chi(b_{+} \wc)=\chi(b_{-}\wc ):={c_{(b_{-}b_{+})},}\qquad
\chi(b^*_{+} \wc)=\chi(b^*_{-}\wc ):=1,
\end{align*}
where vertices $\wc, b_+, b_-, b_+^*, b_-^*\in G_D$ are as shown in Fig.~\ref{fig:ort}. If~$\cH_\mathrm{ort}$ is an orthodiagonal embedding {and~$c_{vv'}$ are given by~\eqref{eq:c-ortho-def},} then these weights are gauge equivalent to the geometrical weights $|d\cT((b\wc)^*)|$. Namely, for all adjacent $b\in \Lambda$ and $\wc\in\diamondsuit$ (i.e., for~\mbox{$b=b_+,b_-,b_+^*,b_-^*$}) the following identity holds:
\[
|d\cT((b\wc)^*)|\ =\ \chi(b\wc)\cdot{\tfrac{1}{2}}|\cH_\mathrm{ort}(b_+)-\cH_\mathrm{ort}(b_-)|\,.
\]
\end{remark}

Let $F$ be a function defined on (a subset of) $\Lambda$.
The discrete operators
${\partial}_{\mathrm{ort}}, \overline{\partial}_{\mathrm{ort}}:\mathbb{C}^{\Lambda} \to\mathbb{C}^{\diamondsuit}$ are defined by the {formulae:}
\[ [{\partial}_{\mathrm{ort}} F](\wc)\ :=\
\frac12\biggl(\frac{F(b_{+})-F(b_{-})}{{b_{+}-b_{-}}}+
 \frac{F(b_{+}^*)-F(b_{-}^*)}{{b_{+}^*-b_{-}^*}}\biggr),
\]
\[ [\overline{\partial}_{\mathrm{ort}} F](\wc)\ :=\
\frac12\biggl(\frac{F(b_{+})-F(b_{-})}{\overline{b}_{+}-\overline{b}_{-}}+
 \frac{F(b_{+}^*)-F(b_{-}^*)}{\overline{b}{}_{+}^*-\overline{b}{}_{-}^*}\biggr);
\]
see Fig.~\ref{fig:ort}.

A function $F$ is called discrete holomorphic at $\wc\in \diamondsuit$ if $[\overline{\partial}_{\mathrm{ort}} F](\wc)=0$. If~$F$ is discrete holomorphic, then replacing it by~$\Re F$ on~$\Gamma$ and by~$\Im F$ on~$\Gamma^*$ (or vice versa) also gives a discrete holomorphic function. Because of that, it is often convenient to assume that~$F$ is purely real on~$\Gamma$ and is purely imaginary on~$\Gamma^*$.

Given~$\wc\in\diamondsuit$, let
\begin{equation}
\label{eq:x-mu-u}
\mu_\diamondsuit(\wc)\ :=\ {\tfrac12\,|b_+^*-b_-^*|\,|b_+-b_-|}
\end{equation}
be the area of the corresponding orthodiagonal quad in~$\cH_\mathrm{ort}$. Further, for~$b\in\Lambda$ define
\begin{equation}
\label{eq:x-mu-b}
\textstyle \mu_\Lambda(b)\ :=\ {\frac{1}{4}}\sum_{{w_k}\sim b}{|b_{k+1}^*-b_k^*|\,|w_k-b|\,,}
\end{equation}
where the sum is taken over all vertices~${w_k}\in\diamondsuit$ adjacent to~$b$; {see Fig.~\ref{fig:isorad_t_graph} for the notation.} Let~$\partial_\mathrm{ort}^*,\overline{\partial}{}_\mathrm{ort}^*:\C^\diamondsuit\to\C^\Lambda$ be the (formal) adjoint operators to the operators~$\partial_\mathrm{ort},\overline{\partial}_\mathrm{ort}$\,, {where we assume that the scalar products in~$\C^\diamondsuit$ and~$\C^\Lambda$ are defined by the weights~\eqref{eq:x-mu-u} and~\eqref{eq:x-mu-b}, respectively.} In particular, one has
\begin{equation}
\label{eq:pa-star-ort}
[\partial^*_\mathrm{ort}G](b)\ =\ \frac{1}{\mu_\Lambda(b)}\sum_{w_k\sim b}\frac{\mu_\diamondsuit(w_k)}{2(\overline{b}-\overline{b}_k)}G(w_k)\ =\ \frac{i}{4\mu_\Lambda(b)}\sum_{w_k\sim b}(b_{k+1}^*-b_k^*)G(w_k).
\end{equation}

Now let~$H$ be a function defined on (a subset of)~$\Gamma$ or, similarly, on (a subset of) $\Gamma^*$. The so-called \emph{cotangent Laplacian} on orthodiagonal embeddings reads as
\begin{align}\label{eq:delta_ort}
[\Delta_{\mathrm{ort}} H](b)\ :=\ \frac{1}{{2\mu_\Lambda(b)}}\sum_{b_k\sim b}{c_{(bb_k)}}(H(b_k)-H(b))\,.
\end{align}
It is well known that the following factorization holds:
\[
{-\Delta_{\mathrm{ort}}}\ =\ 4{\partial}^*_{\mathrm{ort}} {\partial}_{\mathrm{ort}}\ =\ 4\overline{\partial}{}^*_{\mathrm{ort}}\overline{\partial}_{\mathrm{ort}}\,.
\]
A function~$H$ is called discrete harmonic at~$b$ if~$[\Delta_\mathrm{ort}H](b)=0$. In terms of probabilistic models, this notion is naturally associated with reversible random walks on $\Gamma$ and on $\Gamma^*$ with conductances $c_{{(vv')}}$ given by~\eqref{eq:c-ortho-def}.

The factorization of~$\Delta_\mathrm{ort}$ implies that the restrictions of discrete {holomorphic} functions to~$\Gamma$ and to~$\Gamma^*$ are discrete harmonic. In particular, it is easy to see that~$[\overline\partial_\mathrm{ort}\cH_\mathrm{ort}](\wc)=0$ for all \mbox{$\wc\in\diamondsuit$} and hence the `coordinate function' $\cH_\mathrm{ort}:\Lambda\to\C$ is discrete harmonic both on~$\Gamma$ and on~$\Gamma^*$. Thus, orthodiagonal embeddings form a subclass of \emph{Tutte's harmonic embeddings}. Vice versa, given a harmonic embedding~$\cH_\mathrm{harm}:\Gamma\to\C$, one can always construct a discrete harmonic conjugate function~$\cH^*_\mathrm{harm}:\Gamma^*\to\C$ such that the images of dual edges~$e$,~$e^*$ under~$\cH_\mathrm{harm}$ and~$\cH^*_\mathrm{harm}$, respectively, are orthogonal to each other. However, the images of faces of~${\Lambda}$ under such pairs of embeddings are not necessarily proper and/or small even if all the images of edges of both~$\Gamma$ and~$\Gamma^*$ are small. Still, one can introduce a proper t-embedding (with small faces) $\cT=\frac12(\cH_\mathrm{harm}+\cH^*_\mathrm{harm})$ of the corresponding graph~$G_D^*$ similarly to Fig.~\ref{fig:ort} and apply the techniques developed in this paper to t-holomorphic functions on~$\cT$. Let us emphasize once again that, in general, the dual harmonic embeddings~$\cH_\mathrm{harm}$ and~$\cH^*_\mathrm{harm}$ are \emph{not} close to each other and thus are not close to~$\cT$; this is the reason why discrete harmonic functions on harmonic embeddings do not necessarily converge to usual harmonic functions in the small mesh size limit while those on biorthogonal embeddings do;
see also 
Remark~\ref{rem:diffusion-theta}.

\subsubsection{T-holomorphic functions in the orthodiagonal setup}
\label{subsub:appendix-thol-orthodiag} Given an orthodiagonal embedding, let us consider the corresponding t-embedding and let the origami square root function be chosen so that~$\eta_b=\pm 1$ on~$\Gamma$ and~$\eta_b=\pm i$ on~$\Gamma^*$ as discussed above. This also implies that
\begin{equation}
\label{eq:eta_wc=}
\eta_{\wc}\ =\ {\pm ie^{-i\arg(b_+-b_-)},}\quad \wc\in\diamondsuit\,;
\end{equation}
see Fig.~\ref{fig:ort} for the notation. In the orthodiagonal context, the `true' complex values $F_\frw^\tw$ of t-white-holomorphic functions~$F_\frw$ and~$F_\frb^\tb$ of t-black-holomorphic functions~$F_\frb$ do not have much sense, in particular since the faces of the t-embedding are not triangles and there is no canonical way to split them in order to defined such complex values; see Section~\ref{sec:non_triangulation}. However, if one considers the values~$F_\frw^\tb(b)\in\eta_b\R$ and~$F_\frb^\tw(\wc)\in\eta_{\wc}\R$, then the usual notion of discrete holomorphicity arises. Indeed, it is easy to see that
\[
\oint_{\partial \wc} F^\tb \,d  \cT\ =\ {-\eta_{\wc}^2}\cdot (b_+ - b_-)(b_+^*-b_-^*)\cdot [\overline{\partial}_{\mathrm{ort}} F^\tb](\wc)
\]
for all functions~$F^\tb$ (locally) defined on~$\Lambda=\Gamma\cup\Gamma^*$ and
\begin{equation}
\label{eq:x-oint-dT=pa-star}
\oint_{\partial b} F^\tw \,d  \cT\ =\ {\sum_{w_k\sim b}\tfrac{1}{2}(b_k^*-b_{k+1}^*)F^\tw(w_k)\ =\ 2i\mu_\Lambda(b)\cdot [\partial^*_\mathrm{ort}F^\tw](b)}
\end{equation}
for all functions (locally) defined on~$\diamondsuit$. Thus, in the orthodiagonal setup, t-white-holomorphic functions $F_\frw^\tb$ are usual discrete holomorphic functions on~$\Lambda$ such that $F^\tb|_\Gamma\in\R$ and~$F^\tb|_{\Gamma^*}\in i\R$ while
t-black-holomorphic functions $F_\frb^\tw$ are usual discrete holomorphic functions on~$\diamondsuit$ such that~$F^\tw(\wc)\in\eta_{\wc}\R$ for all~$\wc\in\diamondsuit$, i.e., nothing but the discrete gradients of \emph{real-valued} discrete harmonic functions on~${\Gamma^*}$ (or, equivalently, $i\R$-valued harmonic functions on~${\Gamma}$).

\begin{figure}
\center{\includegraphics[width=0.88\textwidth]{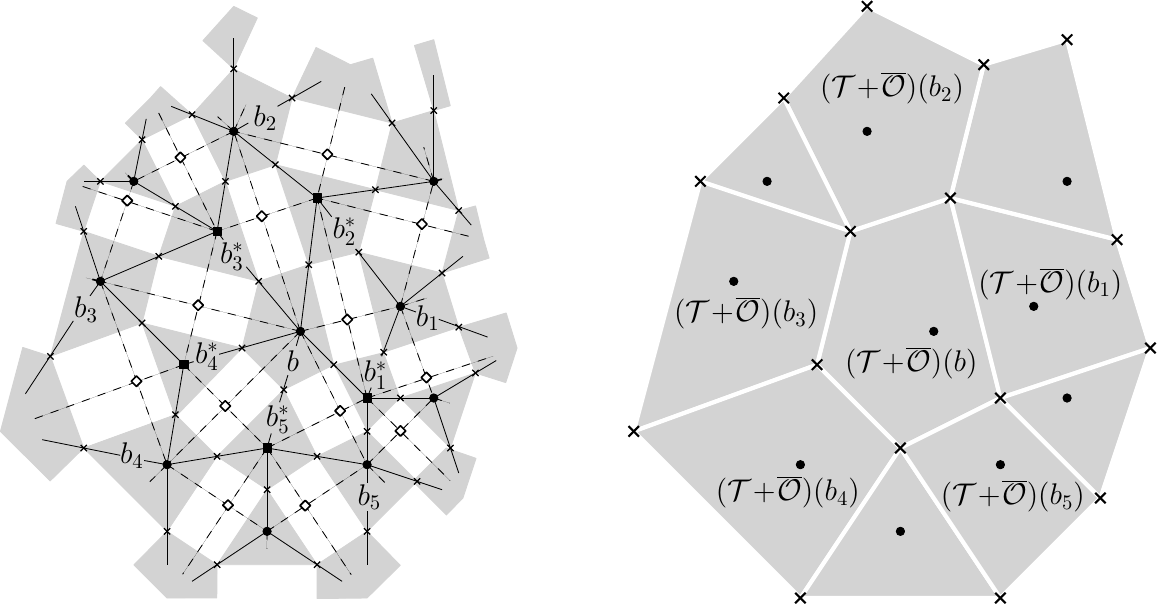}}
\caption{{A t-embedding~$\cT$ obtained from an orthodiagonal embedding (left) and the corresponding T-graph $\cT+\overline{\cO}$ with flattened white faces (right). The black faces~$\cT(b)$, $b\in\Gamma$, scale by the factor~$2$} {while the faces~$\cT(b^*)$, $b^*\in\Gamma^*$, degenerate to vertices of~$\cT+\overline{\cO}$.}\label{fig:isorad_t_graph}}
\end{figure}

\subsubsection{Discrete harmonic functions on~$\Gamma$ as harmonic functions on T-graphs}
\label{subsub:appendix-harm-Temb}
We now discuss how the notion of discrete harmonicity on orthodiagonal embeddings fits the frameworks developed in this paper. A similar discussion {can be applied to general harmonic embeddings by passing to the t-embedding~$\cT=\frac{1}{2}(\cH_\mathrm{harm}+\cH^*_\mathrm{harm})$ obtained from such an embedding~$\cH_\mathrm{harm}$ and its harmonic conjugate~$\cH_\mathrm{harm}^*$.} 

\begin{proposition}\label{prop:Tgraph-ort}
(i) If 
$\eta_{b} = \pm 1$ for $b \in \Gamma$ and $\eta_{b} = \pm i$ for $b \in \Gamma^*$, then the (degenerate) T-graph $\cT + \overline \cO$ coincides with~$\Gamma^*$. {Similarly,}~$\cT - \overline\cO$ coincides with $\Gamma$. 

\smallskip
\noindent (ii) A function is harmonic on the T-graph $\cT+\overline{\cO}$ (resp., on $\cT-\overline{\cO}$) in the sense of Remark~\ref{rq:degenerate_face} if and only if it is harmonic on~$\Gamma^*$ (resp., on~$\Gamma$) in the usual sense.
\end{proposition}

\begin{proof}
(i) Due to Proposition~\ref{prop:white_Tgraph} (see also Section~\ref{sec:non_triangulation}) all faces $(\cT+\overline{\cO})(b)$ for $b\in\Gamma^*$ (resp., all faces $(\cT-\overline{\cO})(b)$ for $b\in\Gamma$) are degenerate; see Fig.~\ref{fig:isorad_t_graph}.
Therefore, vertices of $\cT+\overline{\cO}$ (resp., $\cT-\overline{\cO}$) are in bijection with vertices of $\Gamma^*$ (resp.,~$\Gamma$). Moreover, let $w\in\diamondsuit$ correspond to the intersection of an edge~$e$ of~$\Gamma$ and the dual edge~$e^*$ of~$\Gamma^*$. Then,
\[(\cT + \overline \cO)(w) = 2\Pr (\cT(w) , \overline{\eta}_w \R) = \cH_{\mathrm{ort}}(e^*)\]
since $\overline{\eta}_w$ is the direction of $\cH_{\mathrm{ort}}(e^*)$ and $\cT(w)$ is a rectangle. {The same argument applies to the degenerate T-graph} $\cT - \overline \cO$.

\smallskip

\noindent (ii) Let~$H$ be a real-valued harmonic function on~$\Gamma^*=\cT+\overline{\cO}$. Both the harmonicity condition on~$\Gamma^*$ (i.e., the condition~$\Delta_\mathrm{ort}H=0$) and the harmonicity on the degenerate T-graph~$\cT+\overline{\cO}$) boil down to the condition that the discrete gradient
\[
\rD[H](u^\circ)\ =\ H(b_+^*)-H(b_-^*))/(b_+^*-b_-^*)\ =\ 2[\partial_\mathrm{ort}H](u^\circ),\quad u^\circ\in\diamondsuit,
\]
is a t-black-holomorphic function (recall that this is equivalent to saying that \mbox{$\partial^*_\mathrm{ort}(\rD[H])=0$}). A similar argument applies for functions defined on~$\Gamma=\cT-\overline{\cO}$.
\end{proof}

\subsubsection{Square tilings}\label{subsub:appendix-square-tilings}
We now briefly discuss a link of t-embeddings with a classical notion of square (or, more generally, rectangular) tilings introduced in~\cite{BSST}. As above, let~$\cH_\mathrm{ort}:\Lambda\to\C$ be an orthodiagonal embedding and~${\cT:G_D^*\to\C}$ be a t-embedding of the (dual of the) corresponding dimer graph; see Fig.~\ref{fig:ort}. (One can also start with a pair $\cH_\mathrm{harm}:\Gamma\to\C$, $\cH^*_\mathrm{harm}:\Gamma^*\to\C$ of dual harmonic embeddings and consider the corresponding t-embedding~$\cT=\frac12(\cH_\mathrm{harm}+\cH^*_\mathrm{harm})$.)

We now claim the following: if we consider T-graphs $\cT+{\alpha^2}\cO$, what appears is a tiling by rectangles and, in the particular case of all conductances equal to~$1$, a square tiling; see Fig.~\ref{fig:rect_tilings}. More precisely, if all conductances equal $1$, then a corresponding orthodiagonal embedding restricted to $\Gamma$ {forms} a classical Tutte (or barycentric) embedding and the corresponding T-graphs are square tilings defined by projections of $\cH_{\mathrm{ort}}$. As above, we assume that the origami square root function~$\eta$ is chosen so that~$\eta_b=\pm 1$ on~$\Gamma$ and~$\eta_b=\pm i$ on~$\Gamma^*$.

\begin{figure}
\center{\includegraphics[width=0.88\textwidth]{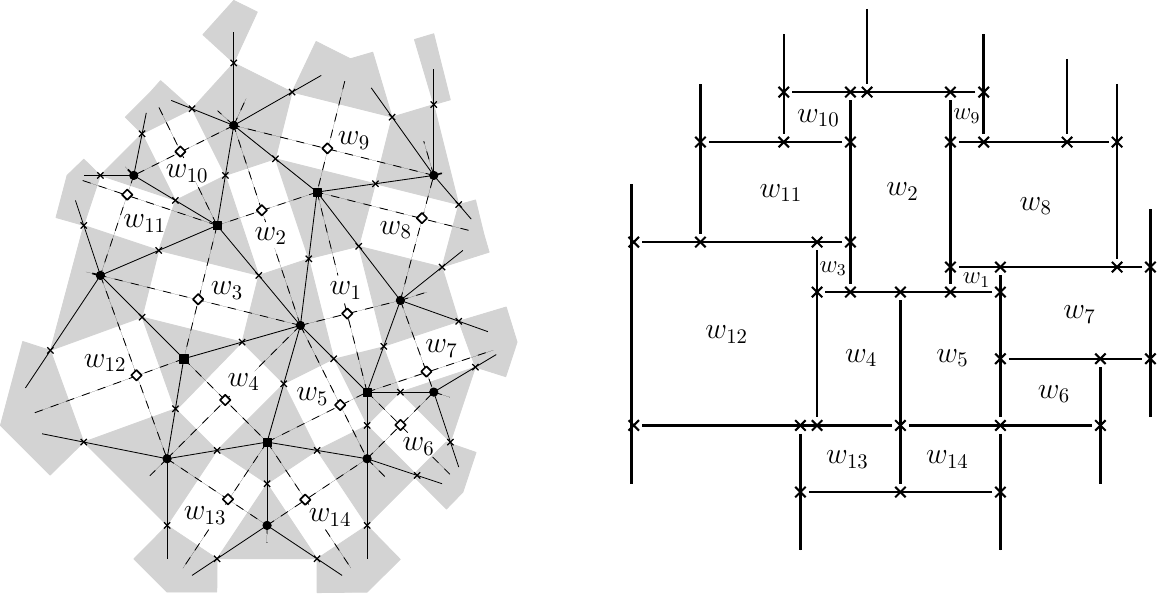}}
\caption{{A t-embedding~$\cT$ obtained from an orthodiagonal embedding (left) and the corresponding T-graph $\cT+\cO$ with flattened black faces (right). The black faces $\cT(b)$ with $b\in\Gamma$ are projected onto the horizontal direction and give rise to horizontal edges of~$\cT+\cO$ while $b^*\in\Gamma^*$ give rise to vertical ones. White} {rectangular faces have aspect ratios~\eqref{eq:c-ortho-def} both in~$\cT$ and in~$\cT+\cO$.}\label{fig:rect_tilings}}
\end{figure}

\begin{proposition}
For {each} $\alpha$ in the unit circle, the T-graph $(\cT+{\alpha^2}\cO)(G^*_D)$ of an orthodiagonal t-embedding $\cT$ forms a tiling by rectangles. Moreover, for all $b\in\Gamma$ and $b^*\in\Gamma^*$ the segments $(\cT+\alpha^2\cO)(b)$ and~$(\cT+\alpha^2\cO)(b^*)$ have directions~$\alpha\R$ and~$i\alpha\R$, respectively, and
\begin{align*}
\Pr ((\cT\!+\!\alpha^2\cO)(b), i\alpha\R)\ &=\ \Pr (\cH_{\mathrm{ort}}(b), i\alpha\R) + {\operatorname{cst}}_{i\alpha},\\[-2pt]
\Pr ( (\cT\!+\!\alpha^2\cO)(b^*), \alpha\R)\ &=\ \Pr (\cH_{\mathrm{ort}}(b^*), \alpha\R) + {\operatorname{cst}}_{\alpha},
\end{align*}
where ${\operatorname{cst}}_{\alpha}\in\alpha\R$ and ${\operatorname{cst}}_{i\alpha}\in i\alpha\R$ are constants depending on $\alpha$ only.
\end{proposition}
\begin{proof}
Proposition~\ref{prop:geomT} (see also Section~\ref{sec:non_triangulation}) imply that all edges {$(\cT+\alpha^2\cO)(b)$, \mbox{$b\in\Gamma$}, are segments with the direction ${\alpha} \R$}. Moreover, for each~$\wc\in\diamondsuit$ the image $(\cT+\alpha^2\cO)(\wc)$ is a translate of the rectangle $(1+\alpha^2\eta_{\wc}^2)\cT(\wc)$, where~$\eta_{\wc}$ is given by~\eqref{eq:eta_wc=}. It is easy to deduce from this fact that
\[
\Pr({(\cT\!+\!\alpha^2\cO)(b_+)-(\cT\!+\!\alpha^2\cO)(b_-)}, i\alpha\R)\ =\ \Pr(\cH_\mathrm{ort}(b_+)-\cH_\mathrm{ort}(b_-),i\alpha\R)
\]
for all pairs of adjacent vertices~$b_\pm\in\Gamma$. Similar arguments apply to the images of {vertices} $b^*\in\Gamma^*$ {in~$\cT+\alpha^2\cO$.}
\end{proof}

\newcommand{\GI}{G_{\mathrm{Ising}}}
\newcommand{\LI}{\Lambda_{\mathrm{Ising}}}
\newcommand{\dI}{\diamondsuit_{\mathrm{Ising}}}
\def\cX{{\mathcal{X}}}

\subsection{S-embeddings}\label{sub:appendix-s-emb}
In this section we recall a link between s-embeddings of planar graphs carrying the Ising model and t-embeddings of the corresponding bipartite dimer model; see also~\cite[Section~7]{KLRR} and~\cite[Section~2.3]{chelkak-s-emb}.
Let $\GI$ be a planar graph and {$\LI$ be its `diamond graph' (see Section~\ref{sub:appendix-ort}; recall that vertex set of~$\LI$ is the union of those of~$\GI$ and~$\GI^*$).} According to~\cite{chelkak-icm2018,chelkak-s-emb}, an s-embedding~$\cS:\LI\to\C$ 
satisfies the condition that every face of $\cS(\LI)$ is a tangential quadrilateral (i.e., admits an inscribed circle). Note that, given~$\cS$, there is also a natural way to embed the graph~$\dI:=\LI^*$ by placing its vertices at the centers of the inscribed circles of the corresponding faces of~$\LI$.

Let us construct a bipartite graph $\Upsilon^\bullet\cup\Upsilon^\circ$ associated to $\GI$ by putting a black and a white vertices to each edge of the graph~$\LI$ (or, equivalently, to each `corner' of the graph $\GI$) as shown in Fig.~\ref{s_emb}.
Due to~\cite{dubedat-bosonization}, there exists a natural correspondence between the Ising model on $\GI$ and the bipartite dimer model on $\Upsilon^\bullet\cup\Upsilon^\circ$ with weights
\begin{equation}
\label{eq:Ising-dimer-weights}
\chi_{(b_{10}w_{10})}=1, \quad \chi_{(b_{11}w_{10})}=\cos\theta_e, \quad \chi_{(b_{00}w_{10})}=\sin\theta_e,
\end{equation}
where vertices $b_{10},$ $b_{11},$ $b_{00}$ and $w_{10}$ are as shown in Fig.~\ref{s_emb} and $\theta_e$ parameterize the Ising interaction constants~$J_e$ so that~$\tanh[\beta J_e]=x_e=\tan\tfrac12\theta_e$.

Let $\Upsilon$ be the set of `corners' of the graph $\GI$ and denote by $\Upsilon^\times$ the double-cover of~$\Upsilon$ that branches over each vertex of $\LI\cup\dI$. Recall that each corner $c\in\Upsilon$ corresponds to an edge $(v^\bullet(c) v^\circ(c))$ of~$\LI$ and let~$w(c)\in\Upsilon^\circ$, $b(c)\in\Upsilon^\bullet$ denote the corresponding vertices of the bipartite graph~$\Upsilon^\circ\cup\Upsilon^\bullet$ so that $(w(c)b(c))^*=(v^\bullet(c) v^\circ(c))$.

It is easy to see that an s-embedding of a graph $\GI$ can be viewed as a \mbox{t-embedding} $\cT$ of the graph $(\Upsilon^\circ\cup\Upsilon^\bullet)^*=\LI\cup\dI$.
Indeed, the angle condition is satisfied at each vertex: if $v\in\LI$, then for each edge $e$ adjacent to~$v$ the segment $\cS(vz_e)$ is the bisector of the corresponding angle (since $\cS(z_e)$ is the center of the circle inscribed into the corresponding quad) and for each $z_e\in \dI$ the sum of two white adjacent to $\cS(z_e)$ angles is $\pi$; see Fig.~\ref{s_emb}. Moreover, the geometrical weights~$|d\cT(bw^*)|$ are gauge equivalent to the weights~\eqref{eq:Ising-dimer-weights} provided that the \mbox{s-embedding} $\cS=\cS_\cX:\LI\to\C$ is constructed {by the rule
\begin{equation}
\label{eq:SX-def}
\cS_\cX(v_p^\bullet)-\cS_\cX(v_q^\circ)\ =\ (\cX(c_{pq}))^2,
\end{equation}
where a \emph{Dirac spinor}} $\cX:\Upsilon^\times\to\C$ satisfies the \emph{propagation equation}
\begin{equation}
\label{eq:propagation}
\cX(c_{pq})=\cX(c_{p,1-q})\cos\theta_e +\cX(c_{1-p, q})\sin\theta_e
\end{equation}
for any three consecutive vertices~$c_{pq}\in\Upsilon^\times$ surrounding~$z_e$; see~\cite[Definition~1.1]{chelkak-s-emb}. 

\begin{figure}
\center{\includegraphics[width=0.88\textwidth]{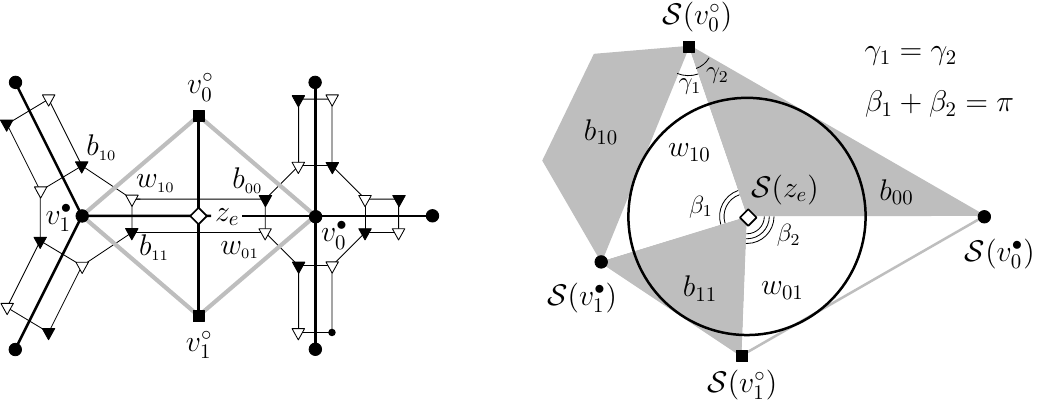}}
\caption{{An s-embedding~$\cS$ of the {`diamond graph' $\LI$ obtained from a planar graph~$\GI$ can be viewed as a t-embedding of (the dual of) the bipartite graph~$\Upsilon^\circ\cup\Upsilon^\bullet$ if} one places vertices of~$\dI=\LI^*$ at the centers of the circles inscribed into faces of~$\LI$. We draw vertices of~$\GI^*$ as black squares but now} {denote them as~$v^\circ$ in order to keep the notation consistent with~\cite{chelkak-s-emb}.}\label{s_emb}}
\end{figure}

Further, note that
\begin{equation}
\label{eq:etab=etaw=}
\eta_{b(c)}=\eta_{w(c)}\ :=\ \exp[-\tfrac{i}{2}\arg(\cS(v^\bullet(c))-\cS(v^\circ(c)))]
\end{equation}
defines an origami square root function on~$\cT$. Moreover, one can make the choice of the signs in the right-hand side canonical {(cf. Remark~\ref{rem:eta-def})} by passing from~$\Upsilon$ to its double cover~$\Upsilon^\times$ and defining a spinor~$\eta_c$, $c\in\Upsilon^\times$, as $\bar{\varsigma}\eta_c:=\eta_{b(c)}=\eta_{w(c)}$, where the global factor $\varsigma$ is such that $|\varsigma|=1$. (The particular choice of~$\varsigma$ is not important; in the planar Ising model literature it is sometimes fixed as $\varsigma=e^{i\frac\pi 4}$ following~\cite{smirnov-annals,chelkak-smirnov-universality}.) Vice versa, each choice of a section of~$\Upsilon^\times$ (i.e., a choice of the square root in the definition of~$\eta_c$ at each `corner'~$c\in\Upsilon$) defines Kasteleyn signs on~$\Upsilon^\bullet\cup\Upsilon^\circ$ according to the following rule: one always assigns the `$-$' sign to edges of the form~$(b(c)w(c))$ while the signs of all other edges~$(bw)$ of~$\Upsilon^\bullet\cup\Upsilon^\circ$ are given by the signs of~$\Re[\eta_b\overline{\eta}_w]$. (The fact that these signs satisfy the Kasteleyn condition directly follows from the branching structure of the double cover~$\Upsilon^\times$ on which the Dirac spinor~$\eta$ is defined.) Once a section of~$\Upsilon^\times$ and the corresponding Kasteleyn signs are fixed, the values~$\cF_\mathrm{gauge}^\bullet(b(c)):=\cX(c)$ and~$\cF_\mathrm{gauge}^\circ(w(c)):=\cX(c)$ are nothing but the Coulomb gauge giving rise to the t-embedding~$\cT$.

\begin{definition}[see~\cite{smirnov-annals,chelkak-smirnov-universality,chelkak-icm2018,chelkak-s-emb}]
\label{def:s-hol}
A function $F: \dI \to \mathbb{C}$ is called s-holomorphic (on an s-embedding~$\cS$) if
\begin{equation}
\label{eq:s-hol}
\Pr(F(z), \eta_c\mathbb{R})\ =\ \Pr(F(z'), \eta_c\mathbb{R})
\end{equation}
for each pair $z,z'\in\dI$ adjacent to the same edge $(v^\bullet(c) v^\circ(c))$ of $\LI$.
\end{definition}

It is easy to see that this definition essentially (up to the factor~$\varsigma$) matches the definition of t-holomorphic functions on the corresponding t-embedding~$\cT$. Indeed, let~$F_\frw$ be, say, a t-white-holomorphic function. Then, its `true' complex values at two vertices~$w_{10},w_{{01}}\in\Upsilon^\circ$ corresponding to the same edge of~$\GI$ match since
\begin{align*}
\Pr( F_\frw^{\tw}( w_{10}) ,{\eta}_{b_{00}} \mathbb{R})\ &=\ F_\frw^{\tb} (b_{00})\ =\ \Pr( F_\frw^{\tw}(w_{01}) ,{\eta}_{b_{00}} \mathbb{R})\,,\\
\Pr( F_\frw^{\tw}( w_{10}) ,{\eta}_{b_{11}} \mathbb{R})\ &=\ F_\frw^{\tb} (b_{11})\ =\ \Pr( F_\frw^{\tw}(w_{01}) ,{\eta}_{b_{11}} \mathbb{R})\,;
\end{align*}
see Fig.~\ref{s_emb} for the notation. This allows to associate these values to the quads $z_e\in\dI$ instead of~$w\in\Upsilon^\circ$ and one immediately sees that the function
\begin{equation}
\label{eq:s-hol<->t-hol}
F(z_e)\ :=\ \varsigma\cdot F_\frw^\tw(w_{10})\ =\ \varsigma\cdot F_\frw^\tw(w_{10})
\end{equation}
satisfies the s-holomorphicity condition~\eqref{eq:s-hol}; vice versa, given an s-holomorphic function on $\dI$ one can view it as a t-white-holomorphic one using the same rule. Since $\Upsilon^\bullet$ and $\Upsilon^\circ$ play fully symmetric roles,
the same discussion applies to the `true' complex values~$F_\frb^\tb$ of t-black-holomorphic functions.

The {next} result gives an interpretation of the values~$F_\frw^\tb$ of t-white-holomorphic functions on~$\Upsilon^\bullet$; a similar consideration applies to values of {functions~$F_\frb^\tw$} on~$\Upsilon^\tw$.

\begin{proposition}[{see also~\cite[Proposition~2.5]{chelkak-s-emb}}]
\label{prop:propagation} Let~$F_\frw$ be a t-white-holomorphic function on the t-embedding~$\cT:\LI\cup\dI\to\C$ associated with the s-embedding $\cS=\cS_\cX:\LI\to\C$. Then, a {real-valued} spinor $c\mapsto {F_\frw^\tb(b(c))\cX(c)}$ satisfies the propagation equation~\eqref{eq:propagation} on $\Upsilon^\times$.

Vice versa, each real-valued spinor~$X$ satisfying the propagation equation~\eqref{eq:propagation} on~$\Upsilon^\times$ gives rise to a t-white-holomorphic function whose values on~$\Upsilon^\bullet$ are defined as~$b(c)\mapsto X(c)/\cX(c)$.
\end{proposition}
\begin{proof} {Note that~$\cX(c)\cdot F_\frw^\tb(b(c))\in \cX(c)\eta_{b(c)}\R=\R$ due to~\eqref{eq:etab=etaw=} and~\eqref{eq:SX-def}. Further, recall that}
\[
\begin{array}{rcl}
\cT(v_1^\bullet)-\cT(z_e)&\!\!=\!\!&\cX(c_{10})\cX(c_{11})\cdot\cos\theta_e\,,\\[4pt]
\cT(z_e)-\cT(v_0^\circ)&\!\!=\!\!&\cX(c_{10})\cX(c_{00})\cdot\sin\theta_e\,,
\end{array}\qquad
\cT(v_1^\bullet)-\cT(v_0^\circ)\ =\ (\cX(c_{10}))^2
\]
provided that the lifts of the corners $c_{00},c_{10},c_{11}$ are neighbors on the double cover~$\Upsilon^\times$ {(e.g., see \cite[Eq. (2.5)]{chelkak-s-emb} or \cite[Eq.~(6.1)]{chelkak-icm2018}).} Therefore, the condition
\begin{samepage}
\begin{align*}
0\ &=\ \oint_{\partial w_{10}} F_\frw^\tb \,d\cT\\
&=\ \big(F_\frw^\tb(b_{10})\cX(c_{10})-F_\frw^\tb(b_{11})\cX(c_{11})\cos\theta_e-F_\frw^\tb(b_{00})\cX(c_{00})\sin\theta_e\big)\cdot \cX(c_{10})
\end{align*}
is equivalent to the propagation equation~\eqref{eq:propagation} {for the spinor~$F_\frw^\tb(b(c))\cX(c)$.}
\end{samepage}

The computation around~$w_{01}$ is similar. Vice versa, if~$X$ is a real-valued spinor satisfying {the propagation equation}~\eqref{eq:propagation} on~$\Upsilon^\times$, then $X(c)/\cX(c)\in\eta_{b(c)}\R$ and $\oint_{\partial w} (X/\cX)d\cT=0$ for all~$w\in\Upsilon^\circ$ due to the same computation.
\end{proof}

The case of s-embeddings is special in many respects as compared to general t-embeddings. In particular, in this setup the origami map~$\cO$ is essentially \emph{one-dimensional} and coincides with the function $\mathcal{Q}=\mathcal{Q}_\cX:\LI\to\mathbb{R}$ defined (up to a global additive constant) by the rule
\begin{equation}
\label{eq:cQ-def}
\mathcal{Q}_\cX(v^\bullet(c))-\mathcal{Q}_\cX(v^\circ(c))\ :=\ |\cS_\cX(v^\bullet(c))-\cS_\cX(v^\circ(c))|=|\cX(c)|^2
\end{equation}
for all $c\in\Upsilon^\times$; see~\cite[Definition 2.2]{chelkak-s-emb}. (As~$\cX$ satisfies the propagation equation~\eqref{eq:propagation}, this definition is consistent; geometrically it means that all quads of the s-embedding~$\cS=\cS_\cX$ are tangential.) Indeed, it is easy to see that the origami map~$\cO$ sends all vertices of~$\LI$ onto the same line: e.g., if one folds the plane along the edge $\cT(v_1^{\bullet})\cT(z_e)$, then the images of $\cT(v_1^\bullet), \cT(v_1^\circ)$ and $\cT(v_0^\circ)$ lie on a line since $\cT(v_1^{\bullet})\cT(z_e)$ is the bisector of the angle $\cT(v_1^{\circ})\cT(v_1^{\bullet})\cT(v_0^{\circ})$. A similar consideration applies to all other edges,
one also sees that $\cO|_{\LI}=\mathcal{Q}$ up to a rotation and a translation. Moreover, $\cO$ sends all vertices of $\dI$ into the same half-plane with respect to the line containing the image of~$\LI$.

{We conclude this section by mentioning a reformulation of Proposition~\ref{prop:integral_product} in the context of s-embeddings:
for each s-holomorphic function~$F$ defined on (a subset of)~$\dI$, the differential form
\begin{equation}
\label{eq:Im(F^2dS)-semb}
\tfrac12\Re(\overline{\varsigma}^2 F^2d\cS+|F|^2d{\mathcal{Q}})
\end{equation}
is closed (as this is nothing but the differential form~\eqref{eq:FFdT=} under the correspondence~\eqref{eq:s-hol<->t-hol}). The primitive of the differential form~\eqref{eq:Im(F^2dS)-semb} plays a very important role in the analysis of fermionic and spinor observables appearing in the planar Ising model on s-embeddings; see~\cite{chelkak-s-emb}.}

\subsection{Isoradial grids}\label{sub:appendix-isoradial}
Isoradial grids (or, equivalently, rhombic lattices) provide a nice setup for discrete complex analysis that goes beyond a more straightforward discretization on the square grid; this was first pointed out by Duffin~\cite{duffin-rhombic} in late 1960s. The interest to this setup in connection with the planar Ising and the bipartite dimer models reappeared in the work of Mercat~\cite{mercat-CMP} and Kenyon~\cite{kenyon-isorad} in early 2000s, respectively; see also~\cite{chelkak-smirnov-analysis,chelkak-smirnov-universality}. One says that~$\Gamma^\delta=\Gamma^{\bullet,\delta}$ is an isoradial grid of mesh size~$\delta$ if~$\Gamma^\delta$ is a planar graph in which each face is inscribed into a circle of a common radius~$\delta$. Suppose that all circle centers are inside the corresponding faces, then the dual graph $\Gamma^{*,\delta}=\Gamma^{\circ,\delta}$ is also isoradial with the same radius. The associated rhombic lattice~${\Lambda^\delta}$ is the graph on the union of the vertex sets {of~$\Gamma^\delta$ and~$\Gamma^{*,\delta}$} with natural incidence relations. Clearly, all faces of~$\Lambda^\delta$ are rhombi with edge length~$\delta$. Isoradial grids form a subclass of the intersection of orthodiagonal embeddings and s-embeddings and therefore there are two notions of t-holomorphicity associated with them according to Sections~\ref{sub:appendix-ort} and~\ref{sub:appendix-s-emb}.

\begin{figure}
\center{\includegraphics[width=0.88\textwidth]{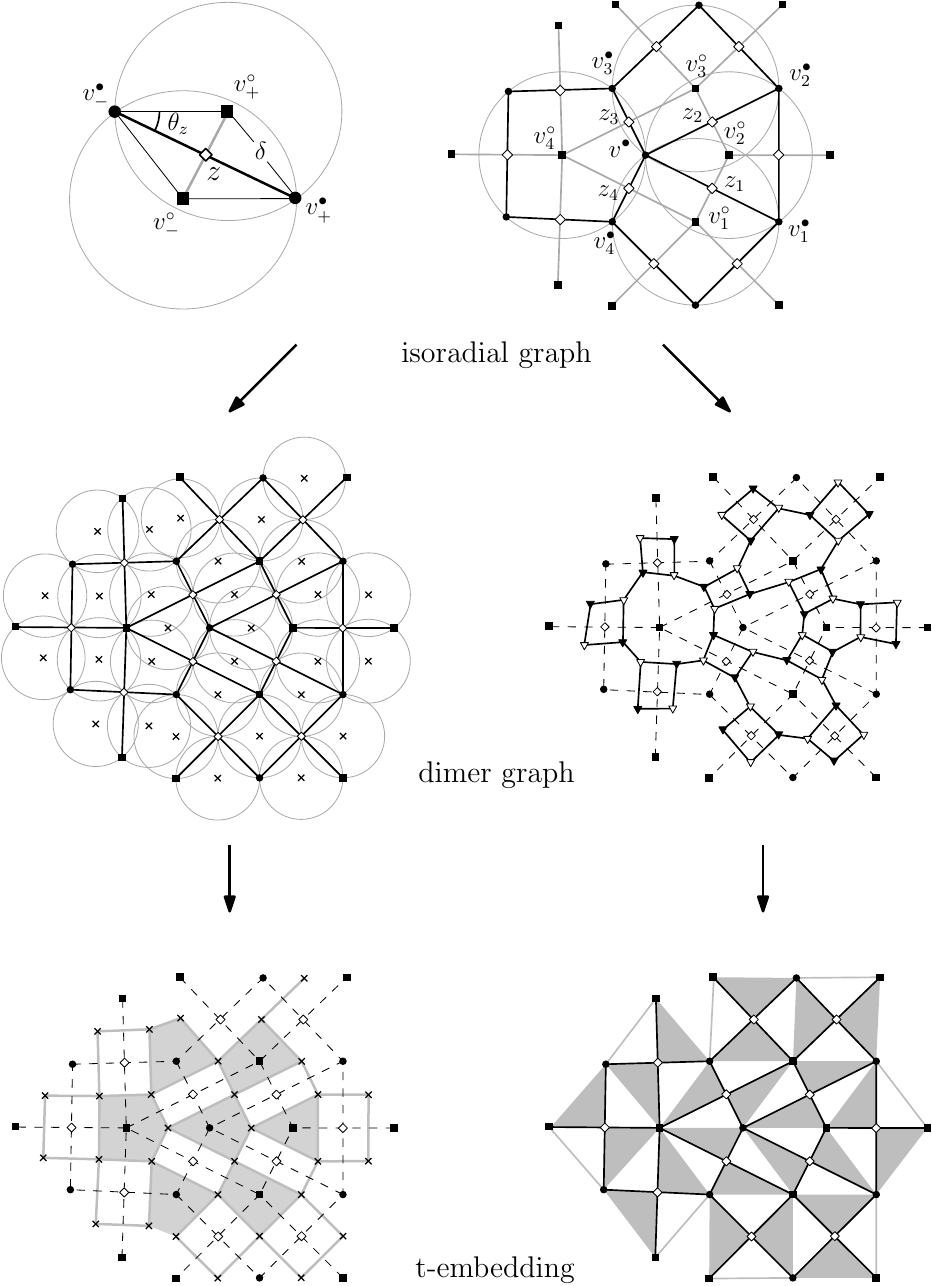}}
\bigskip
\caption{Two t-embeddings associated with an isoradial graph~$\Gamma$, its dual isoradial graph~$\Gamma^*$, and the corresponding rhombic lattice~$\Lambda$ of mesh size~$\delta$. \mbox{\textsc{Left:} a t-embedding} of the (dual of the) bipartite graph~$\Lambda\cup\diamondsuit$, where~$\diamondsuit=\Lambda^*$, is a particular case of orthodiagonal embeddings discussed in Section~\ref{sub:appendix-ort}. \mbox{\textsc{Right:} a t-embedding} of the graph~$(\Upsilon^\bullet\cup\Upsilon^\circ)^*=\Lambda\cup\diamondsuit$ is a particular case of s-embeddings discussed in Section~\ref{sub:appendix-s-emb}.}\label{fig:isorad}
\end{figure}

We begin {by} recalling the definition of discrete holomorphic functions on {the set} $\diamondsuit^\delta=(\Lambda^\delta)^*$ from Section~\ref{sub:appendix-ort} in the isoradial setup; see Fig.~\ref{fig:isorad} for the notation. Given $z\in \diamondsuit^\delta$, let $\theta_z$ denote the half-angle of the corresponding rhombus $(v^\bullet_-v^\circ_-v^\bullet_+v^\circ_+)$ along the edge~$(v^\bullet_-v^\bullet_+)$ of~$\Gamma$.  The face weights~\eqref{eq:x-mu-u},~\eqref{eq:x-mu-b} are given by
\[
\mu_\diamondsuit(z)\ =\ \delta^2\sin (2\theta_z),\ \ z\in\diamondsuit^\delta,\qquad \mu_\Lambda(v)\ =\ \textstyle \frac14\delta^2\sum_{z_k\sim v}\sin(2\theta_{z_k}),\ \ v\in\Lambda^\delta,
\]
the edge conductances~\eqref{eq:c-ortho-def} are equal to~$\tan\theta_z$, and one has
\begin{align*}
[\partial F](z)\ &=\ \frac{1}{2}\biggl(\frac{F(v^\bullet_+)-F(v^\bullet_-)}{v^\bullet_+-v^\bullet_-}+ \frac{F(v^\circ_+)-F(v^\circ_-)}{v^\circ_+-v^\circ_-}\biggr), \quad z\in\diamondsuit^\delta \\
[\partial^*G](v^\bullet)\ &=\ \frac{i}{4\mu_\Lambda(v^\bullet)}\sum_{z_k\sim v^\bullet}(v^\circ_{k+1}-v^\circ_k)\cdot G(z_k),\hskip 38pt v^\bullet\in\Gamma^\delta,\\
[\Delta H](v^\bullet)\ &=\ [-4\partial^*\partial H](v^\bullet)\ =\ \frac{1}{2\mu_\Lambda(v^\bullet)}\sum_{v^\bullet_k\sim v^\bullet}\tan\theta_k\cdot H(v^\bullet_k)
\end{align*}
and similarly for~$v^\circ\in\Gamma^{*,\delta}$; see~\eqref{eq:pa-star-ort} and~\eqref{eq:delta_ort}. Recall that a function~$G$ defined on (a subset of) $\diamondsuit^\delta$ is called discrete holomorphic at~$v^\bullet$ if~$[\partial^* G](v^\bullet)=0$ and that a function~$H$ defined on (a subset of) $\Gamma^\delta$ is called discrete harmonic at~$v^\bullet$ if~$[\Delta H](v^\bullet)=0$. All results mentioned in Section~\ref{sub:appendix-ort} apply on isoradial grids, in particular, the following holds:
\begin{itemize}
\item If~$F_\frb$ is a t-black-holomorphic function (in the sense of Section~\ref{sub:appendix-ort}), then $F_\frb^\tw$ is discrete holomorphic {(i.e., $\partial^*F^\tw_\frb=0$ on~$\Lambda^\delta$)} and satisfies the condition \mbox{$F_\frb^\tw(z)\in (\overline{v}^\circ_+-\overline{v}^\circ_-)\R$} for all~$z\in\diamondsuit^\delta$.
\item Vice versa, each discrete holomorphic function~$G$ defined on (a subset of)~$\diamondsuit^\delta$ and such that~$G(z)\in (\overline{v}^\circ_+-\overline{v}^\circ_-)\R$ is a t-black-holomorphic function in the sense of Section~\ref{sub:appendix-ort}.
\item Locally (or in simply connected domains), such discrete holomorphic functions are nothing but the discrete gradients of \emph{real-valued} harmonic functions on~$\Gamma^*=\Gamma^{\circ,\delta}$ or, equivalently, the discrete gradients of $i\R$-valued harmonic functions on~$\Gamma=\Gamma^{\bullet,\delta}$.
\end{itemize}

\begin{remark}\label{rem:involution-appendix}
It is worth noting that the condition~$G(z)\in \eta_z\R=(\overline{v}^\circ_+-\overline{v}^\circ_-)\R$ that specifies t-black-holomorphic functions~$F_\frb^\tw$ out of all discrete holomorphic functions on $\diamondsuit^\delta$ can be equivalently formulated as the requirement that~$G$ is invariant under the involution~$G(z)\mapsto \eta_z^2\overline{G(z)}$ on the set of discrete holomorphic functions on~$\diamondsuit^\delta$; see equation~\eqref{eq:x-oint-dT=pa-star} and Remark~\ref{rq:involution}.
\end{remark}

For the notion of discrete harmonicity inherited from orthodiagonal embedding, the main simplification in the isoradial setup comes from the fact that the discrete Laplacian approximates the continuous one up to the \emph{second order}: one has~$\Delta H=2(a+c)$ if~$H=ax^2+2bxy+cy^2$. This means that the covariance matrix~$\var X_t$ of the corresponding continuous time random walk~$X_t$ on~$\Gamma$ (see Section~\ref{sec:Tgraph} and Proposition~\ref{prop:Tgraph-ort}) not only satisfies the normalization~$\mathrm{Tr}\,\var X_t=t$ but is \emph{rotationally invariant}, i.e., one has $\var  X_t=\frac12 t\cdot \mathrm{Id}$; cf.~\cite[Eq.~(1.2)]{chelkak-smirnov-analysis}. In particular, we see that this random walk must converge as~$\delta\to 0$ to a standard 2D Brownian motion (as opposed to a centered diffusion with non-trivial covariance depending on the point that arises for general orthodiagonal embeddings; see Remark~\ref{rem:diffusion-theta}). 

We now move on to the second approach to discrete holomorphic functions in the isoradial setup that comes from the notion of s-holomorphicity discussed in Section~\ref{sub:appendix-s-emb}; see Fig.~\ref{fig:isorad}. Clearly, a rhombic lattice~$\Lambda^\delta$ of mesh size~$\delta$ can be viewed as an s-embedding~$\cS^\delta=\cS^\delta_\cX$, where the Dirac spinor~$\cX^\delta$ is given by
\[
\cX^\delta(c)\ =\ (v^\bullet(c)-v^\circ(c))^{1/2}\ =\ \delta^{1/2} \overline{\eta}_{b(c)}=\delta^{1/2}\overline{\eta}_{w(c)}=\delta^{1/2} \varsigma\overline{\eta}_c,\quad c\in\Upsilon^\times;
\]
see~\eqref{eq:etab=etaw=}. Moreover, in the isoradial setup the parameter~$\theta_e$ involved into the propagation equation~\eqref{eq:propagation} is nothing but the half-angle of the corresponding rhombus.

Recall that the notion of t-holomorphicity on s-embeddings boils down to Definition~\ref{def:s-hol} and that the correspondence between t-holomorphic and s-holomorphic functions is given by~\eqref{eq:s-hol<->t-hol}: one can assign the `true' complex values of t-holomorphic functions to~$z\in\diamondsuit$ and the result is an s-holomorphic function (up to a fixed global factor~$\varsigma$ such that~$|\varsigma|=1$). A major simplification in the isoradial setup as compared to general s-embeddings is that the function~$\mathcal{Q}^\delta$ defined by~\eqref{eq:cQ-def} (or, equivalently, the origami map~$\cO^\delta$ associated to the corresponding t-embedding) becomes completely explicit: one can choose an additive constant in its definition so that
\begin{equation}
\label{eq:cQ-isorad}
\mathcal{Q}^\delta(v^\bullet)=\tfrac12\delta\ \ \text{for all $v^\bullet\in\Gamma^\delta$}\quad \text{and}\quad \mathcal{Q}^\delta(v^\circ)=-\tfrac12 \delta\ \ \text{for all $v^\circ\in\Gamma^{*,\delta}$.}
\end{equation}
(Moreover, one easily sees that $|\cO^\delta(z)|=\frac\delta 2$ for all~$z\in\diamondsuit$ and that the image of~$\cO^\delta$ is contained in the intersection of the ball~$B(0,\frac\delta 2)$ and the upper half plane.)

\begin{proposition}[{see also~\cite[Lemma~3.2]{chelkak-smirnov-universality}}] Let~$F$ be an s-holomorphic function defined on (a subset of) the rhombic lattice~$\diamondsuit^\delta=(\Lambda^\delta)^*$. Then, $F$ is discrete holomorphic, i.e., $[\partial^*F](v)=0$ at all points~$v\in\Lambda^\delta=\Gamma^\delta\cup\Gamma^{*,\delta}$ where this discrete derivative is defined.
\end{proposition}
\begin{proof}
This is a trivial combination of the definition~\eqref{eq:pa-star-ort} of~$\partial^*F$, the fact that the differential form~$ Fd\cS+\varsigma^2\overline{F}d\mathcal{Q}$ is closed on~$\Lambda^\delta$ (see~\eqref{eq:s-hol<->t-hol} and Proposition~\ref{prop:integral_simple}) and of the identity~\eqref{eq:cQ-isorad}, which implies that the second term disappears if we restrict this form onto~$\Gamma^\delta$ (or~$\Gamma^{*,\delta}$) only.
\end{proof}
\begin{remark} Let~$F$ be an s-holomorphic function on an isoradial grid. Due to~\eqref{eq:cQ-isorad}, restricting the differential form~\eqref{eq:Im(F^2dS)-semb} onto~$\Gamma$ (or, {similarly}, onto~$\Gamma^*$) one obtains a consistent definition of the discrete primitive~$\frac12\int\Re[\overline{\varsigma}^2(F(z))^2dz]$, {which reads as} $\frac12\int\Im[(F(z))^2dz]$ {if~$\varsigma=e^{i\frac\pi 4}$}; cf.~\cite[Section~3.3]{chelkak-smirnov-universality}.
\end{remark}

Let us now briefly recall a characterization of s-holomorphic functions within the (strictly larger) class of discrete holomorphic functions on~$\diamondsuit^\delta$; see also~\cite[Section~3.2]{chelkak-smirnov-universality} for more details.

If~$F$ is an s-holomorphic function, then one can easily see that there {exists} a \emph{real-valued} spinor $X(c):=\Re[\overline{\eta}_c F(z)]$ satisfying the propagation equation~\eqref{eq:propagation} on~$\Upsilon^\times$ and such that
\begin{equation}
\label{eq:F=etaX+etaX}
F(z)\ =\ \eta_{c_{00}}X(c_{00})+\eta_{c_{11}}X(c_{11})=\eta_{c_{01}}X(c_{01})+\eta_{c_{10}}X(c_{10}),
\end{equation}
where~$c_{pq}\in\Upsilon^\times$ are the four `corners' of~$\Gamma$ (or, equivalently, edges of~$\Lambda$) that surround~$z$; note that the right-hand side is a product of two spinors ($X$ and $\eta$) on~$\Upsilon^\times$ and thus does not depend on the choice of the lifts of~$c_{pq}$ from~$\Upsilon$ onto~$\Upsilon^\times$. In fact, one can show that a similar representation holds for all discrete holomorphic functions~$F$ defined on simply connected subsets of~$\diamondsuit^\delta$: e.g., it follows from~\cite[Lemma~3.3]{chelkak-smirnov-universality} that each such a function~$F$ admits a decomposition~$F=F_1+iF_2$ with s-holomorphic~$F_1$ and $F_2$, which gives a \emph{complex-valued} spinor~$X=X_1+iX_2$ {solving the equation~\eqref{eq:F=etaX+etaX}}. Moreover, one can easily see that~$X$ is defined uniquely up to adding a constant multiple of~$\eta_c$.

Since the propagation equation~\eqref{eq:propagation} has real coefficients, there exists a trivial involution~$X\mapsto\overline{X}$ on the set of spinors satisfying this equation. One can now use the representation~\eqref{eq:F=etaX+etaX} in order to obtain the corresponding involution on the set of discrete holomorphic functions on~$\diamondsuit^\delta$. (Since~$X$ in~\eqref{eq:F=etaX+etaX} is defined only up to adding a constant multiple of~$\eta_c$, one should also remove from the consideration constant multiples of~$\overline{\eta}_c$, which correspond to constant discrete holomorphic functions on~$\diamondsuit^\delta$.) The class of s-holomorphic functions is the invariant subspace of discrete holomorphic functions under this involution. It is worth pointing out a similarity with Remark~\ref{rem:involution-appendix}: in the isoradial setup both real-linear notions of \mbox{t-holo}\-morphicity (inherited from orthodiagonal embeddings and from s-embeddings, respectively) can be used to obtain the \emph{same} complex-linear discrete holomorphicity relations~$[\partial^* G](v)=0$, $v\in\Lambda$. However, the conventions to select a real-linear subspace of such functions $G$ are different.

\subsubsection{Kite embeddings}\label{subsub:appendix-kites} It is easy to see that the intersection of the two setups discussed in Section~\ref{sub:appendix-ort} (orthodiagonal embeddings) and in Section~\ref{sub:appendix-s-emb} (s-embeddings) is strictly bigger than the isoradial grids discussed in this section: the two diagonals of a tangential quad~$(v^\bullet_-v^\circ_-v^\bullet_+v^\circ_+)$ are orthodiagonal if and only if this quad is a kite. It is also natural to assume that all these kites are symmetric with respect to, say, edges~$(v^\bullet_-v^\bullet_+)$ of~$\Gamma$; note that this assumption breaks the symmetry between~$\Gamma$ and~$\Gamma^*$. A convenient way to think about such embeddings is to view them as \emph{circle patterns}: each vertex of~$\Gamma$ becomes a center of a circle and vertices of~$\Gamma^*$ are their intersection points. This setup appeared in the planar Ising model context in the work of Lis~\cite{Lis} as a generalization of the class of isoradial embeddings. Clearly, all results of Sections~\ref{sub:appendix-ort},~\ref{sub:appendix-s-emb} still apply in this case and two notions of t-holomorphic functions arise. However, to the best of our knowledge there is no simple link between these two notions. Let us also emphasize that the circle patterns mentioned above are \emph{not} the ones discussed in~\cite{KLRR}, where the intersections of circles form the dual to a bipartite dimer graph; this is why one should be precise when speaking about the `discrete complex analysis on circle patterns'.

\def\BR{B_{\raisebox{-1pt}{\tiny R}}}
\def\BI{B_{\raisebox{-1pt}{\tiny I}}}

\subsection{Regular lattices}\label{sub:appendix-regular}
\subsubsection{Square lattice} \label{subsub:appendix-Z2} In this section we briefly discuss the most classical discretization of complex analysis: that on the square lattice. This notion dates back at least to 1940s, e.g., it explicitly appeared in the work of Ferrand~\cite{ferrand-44}. Consider a checkerboard tiling $\mathbb{C}^\delta$ of the complex plane with
squares, each square has side $\delta$ and is centered at a lattice point of
$\left\{\left(\delta n,\delta m\right)\, | \,n, m \in \mathbb{Z}\right\}$. Let us call~$(n,m)$ the coordinates of the corresponding square and
define the sets $\BR^\delta$ and $\BI^\delta$ of black squares and the sets $W_\lambda^\delta$ and $W_{i\lambda}^\delta$  of white squares by the following properties:
\begin{itemize}
\item $(B_{\raisebox{-1pt}{\tiny R}}^\delta)$ both coordinates are even;
\item $(B_{\raisebox{-1pt}{\tiny I}}^\delta)$ both coordinates are odd;
\item $(W_\lambda^\delta)$ the first coordinate is even and the second one is odd, where~$\lambda:=e^{i\frac\pi 4}$;
\item $(W_{i\lambda}^\delta)$ the first coordinate is odd and the second one is even;
\end{itemize}
see Fig.~\ref{fig:sqgrid}. Classically, a function $F$ defined on (a subset of) $\BR^\delta\cup \BI^\delta$ is called discrete holomorphic if it has purely real values on~$\BR^\delta$, purely imaginary values on~$\BI^\delta$, and the following identity holds for odd $n+m$:
\begin{equation}
\label{eq:disc-hol-Ferrand}
F(\delta n,\delta(m\!+\!1))-F(\delta n,\delta(m\!-\!1))\ =\ i(F(\delta(n\!+\!1),\delta m)-F(\delta(n\!-\!1),\delta m)).
\end{equation}
Note that this definition can be viewed as a particular case of t-white-holomorphic functions on orthodiagonal embeddings if we set~$\Gamma:=\BR^\delta$ and $\Gamma^*:=\BI^\delta$; see Section~\ref{subsub:appendix-thol-orthodiag}.

Clearly, the checkerboard tiling~$\C^\delta$ is an example of a t-embedding with square faces. According to Section~\ref{sec:non_triangulation}, in order to speak about the `true' complex values of, say, t-white-holomorphic functions on $\C^\delta$ one should fix a splitting of its white faces; a similar discussion applies to t-black-holomorphic functions. This splitting can be done, e.g., by drawing diagonals of white squares as shown in Fig.~\ref{fig:sqgrid}. Let us call thus obtained t-embedding~$\cT^{\circ,\delta}_\mathrm{spl}$ and denote by~$\mathcal{V}_\diamond^\delta$ the set of its vertices that are \emph{not} adjacent to these diagonals. If $\eta_b=\pm 1$ for~$b\in\BR^\delta$ and $\eta_b=\pm i$ for~$b\in \BI^\delta$, then the origami square root function~$\eta_b$ has the values~$\pm \lambda$ and~$\pm i\lambda$ on diagonals splitting white squares~$w\in W_\lambda^\delta$ and~$w\in W_{i\lambda}^\delta$, respectively. The following notion was introduced by Smirnov in his seminal work~\cite{smirnov-annals} on the conformal invariance in the critical Ising model on~$\mathbb{Z}^2$ (the name s-holomorphicity was coined in a later paper~\cite{chelkak-smirnov-universality} devoted to the isoradial setup).

\begin{figure}
\center{
\includegraphics[clip, trim=0cm 0cm 7cm 0cm, width=0.4\textwidth]{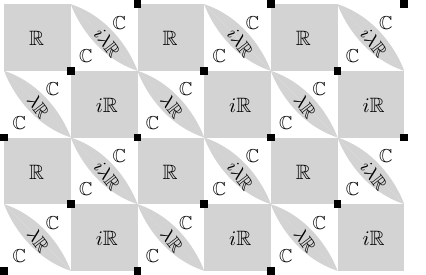}\hskip 0.1\textwidth
\includegraphics[clip, trim=11.5cm 0.96cm 0cm 0cm, width=0.18\textwidth]{holom_s-holom_t-holom_2.pdf}}
\caption{\textsc{Left:} A splitting~$\cT^{\circ,\delta}_\mathrm{spl}$ (see Section~\ref{sec:non_triangulation}) of the white squares of~$\delta\mathbb{Z}^2$ by diagonals and the types  of values ($\eta_b\R$ or $\C$) of t-white-holomorphic functions on this t-embedding. The vertices from the set~$\mathcal{V}^\delta_\diamond$ are shown as black squares.
\textsc{Right:} Local notation near a vertex~$z\in\mathcal{V}^\delta_\diamond$, note that one has~$F_\frw^\tw(z^\circ_\lambda)=F_\frw^\tw(z^\circ_{i\lambda})$.}\label{fig:sqgrid}
\end{figure}

\begin{definition}\label{def:s-hol-sqgrid}
A function $F_\diamond$ defined on (a subset of) the grid $\mathcal{V}_\diamond^\delta$ is called \mbox{s-holo}\-morphic if for each pair of vertices $z_1, z_2\in\mathcal{V}_\diamond^\delta$ of the same square $s\in B^\delta\cup W^\delta$ the following holds:
\begin{equation}
\label{eq:s-hol-sqgrid}
\Pr(F_\diamond(z_1), \eta(s)\mathbb{R})\ =\ \Pr(F_\diamond(z_2), \eta(s)\mathbb{R}),
\end{equation}
where  $\eta(s)$ is $1$, $i$, $\lambda$ or $i\lambda$ if $s$ {has} the type
$B_{\raisebox{-1pt}{\tiny R}}^\delta$,
$B_{\raisebox{-1pt}{\tiny I}}^\delta$,
$W_\lambda^\delta$ or
$W_{i\lambda}^\delta$, respectively.
\end{definition}
Let~$F_\frw$ be a t-white-holomorphic function on~$\cT^{\circ,\delta}_\mathrm{spl}$. By definition, its values at two white triangles~$z^\circ_\lambda$, $z^\circ_{i\lambda}$ adjacent to a vertex~$z\in \mathcal{V}_\diamond^\delta$ have matching real parts as well as matching imaginary parts; see Fig.~\ref{fig:sqgrid}. Thus, one can define a function
\begin{equation}
\label{eq:sqgrid:s-hol<->tw-hol}
F_\diamond(z)\ :=\ F_\frw^\tw(z_{\lambda}^\circ)=F_\frw^\tw(z_{i\lambda}^\circ), \quad z\in \mathcal{V}_\diamond^\delta.
\end{equation}
It is easy to see that this function is s-holomorphic (provided that~$F_\frw$ is t-white-holomorphic) and that, vice versa, starting with an s-holomorphic function~$F_\diamond$ one can define a t-white-holomorphic function~$F_\frw$ by the same rule. It is also easy to see that Definition~\ref{def:s-hol-sqgrid} actually coincides with the definition of s-holomorphic functions on isoradial grids or, more generally, on s-embeddings (see Definition~\ref{def:s-hol} and Section~\ref{sub:appendix-isoradial}) if we set~$\dI:=\mathcal{V}_\diamond^\delta$; note that the mesh size of thus obtained rhombic lattice is~$\sqrt{2}\delta$.

Finally, it is known (e.g., see~\cite[Remark~3.3]{smirnov-annals}) that the notion of s-holomorphic functions on~$\mathcal{V}^\delta_\diamond$ is actually \emph{equivalent} to the classical definition~\eqref{eq:disc-hol-Ferrand} of discrete holomorphic functions on~$\BR^\delta\cup\BI^\delta$. More precisely, the following holds (see also~\cite{russkikh-h}):
\begin{itemize}
\item Let~$F_\frw^\tw$ be a t-white-holomorphic function defined on (a subset of) $\cT^{\circ,\delta}_\mathrm{spl}$. Then, the function~$F_\frw^\bullet$ is discrete holomorphic on~$\BR^\delta\cup\BI^\delta$ since the equation~\eqref{eq:disc-hol-Ferrand} can be equivalently written as the condition~$\oint_{\partial w}F_\frw^\tb d\cT=0$, where the contour integral is computed around the white square~$w$ (split into two triangles) centered at the point~$(\delta n,\delta m)$.
\item Vice versa, let a function~$F$ be defined on (a subset of) $\BR^\delta\cup\BI^\delta$ so that it has purely real values on~$\BR^\delta$ and purely imaginary ones on~$\BI^\delta$. If this function satisfies the discrete holomorphicity condition~\eqref{eq:disc-hol-Ferrand}, then the function~$F_\diamond(z):=F(b_R)+F(b_I)$ is s-holomorphic and hence can be also viewed as a t-white-holomorphic function on~$\cT^{\circ,\delta}_\mathrm{spl}$ via~\eqref{eq:sqgrid:s-hol<->tw-hol}. Indeed, the identities~\eqref{eq:s-hol-sqgrid} with~$\eta(s)=1$ or~$\eta(s)=i$ are tautological while the similar identities with~$\eta(s)=\lambda$ or~$\eta(s)=i\lambda$ turn out to be equivalent to the equation~\eqref{eq:disc-hol-Ferrand}.
\end{itemize}
Let us emphasize that the equivalence of the two notions of discrete holomorphicity discussed in this section heavily relies upon the special structure of the square grid and does \emph{not} hold, e.g., in a  more general setup of isoradial grids.

\subsubsection{Special definitions on the triangular/honeycomb lattices} We conclude this paper by mentioning two special notions of discrete holomorphic functions on triangular/honeycomb lattices: the first was suggested by Dynnikov and Novikov~\cite{dynnikov-novikov} in the discrete Hodge theory context, the second appeared in a recent work~\cite{chelkak-glazman-smirnov} on the discrete stress-energy tensor in 2D lattice models.

\begin{figure}
\center{\includegraphics[width=0.36\textwidth]{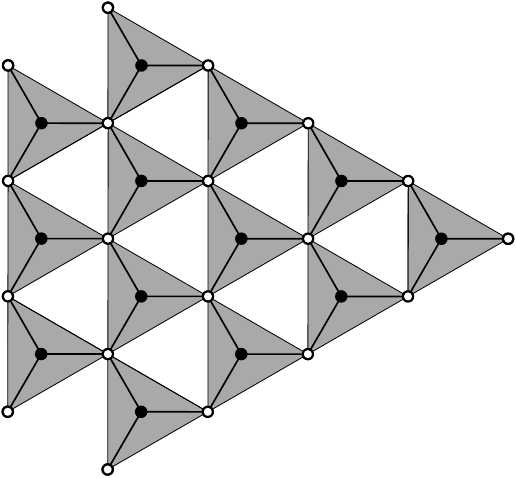}\hskip 0.1\textwidth
\includegraphics[width=0.36\textwidth]{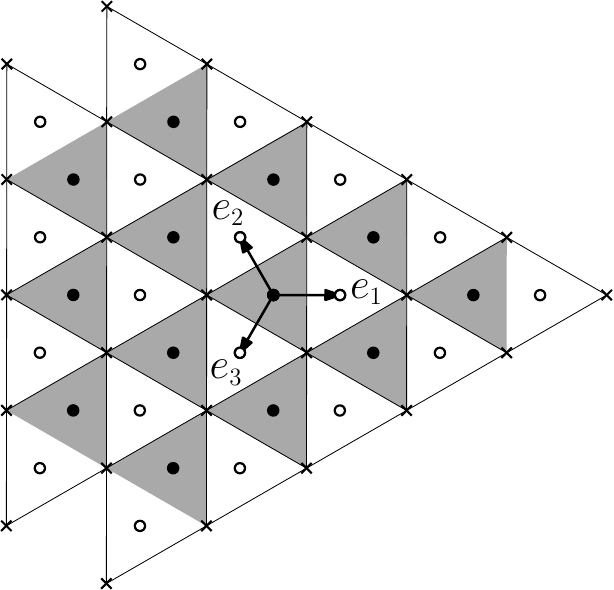}}
\caption{\textsc{Left:} in the notation of~\cite[Section~4]{dynnikov-novikov}, a real-valued function defined on the triangular grid~$W$ (white nodes) is called discrete holomorphic if the sum of its values at the three vertices of each black triangle vanishes. \textsc{Right:} a regular t-embedding $\cT:(B\cup W)^*\to\C$ of the (dual of the) honeycomb lattice~$B\cup W$; vectors~$e_1,e_2,e_3$ have directions~$1,e^{2\pi i/3},e^{4\pi i/3}$, respectively.}\label{fig:hex}
\end{figure}

\begin{definition} A real-valued function~$G$ defined on \emph{vertices} of a regular triangular lattice $W$ is called discrete holomorphic in the sense of~\cite[Section~4]{dynnikov-novikov} if the equation
\begin{equation}
\label{eq:disc-hol-DN}
G(w_1)+G(w_2)+G(w_3)=0,\quad w_{1,2,3}\sim b
\end{equation}
holds for the three vertices~$w_{1,2,3}$ of each black triangle of~$W$; see Fig.~\ref{fig:hex}.
\end{definition}

Denote by~$B$ the set of centers of these black triangles and note that the set $B\cup W$ forms a honeycomb lattice. Let us now consider a regular t-embedding~$\cT$ of {this set}; see Fig.~\ref{fig:hex}. Since all faces of the honeycomb lattice have degree~$6\in 2+4\mathbb{Z}$, one can consistently define the square root origami function~$\eta_w$ (and not only~$\eta_w^2$; cf. Remark~\ref{rem:eta-def}) so that~$\eta_w\in\{1,e^{2\pi i/3},e^{4\pi i/3}\}$ for all~$w\in W$; each of these three values appears exactly once around each black triangle of $\cT$.

Now let~$F_\frb$ be a t-black-holomorphic function on~$\cT$. By definition, one has $F_\frb^\tw(w)\in\eta_w\R$ for all~$w\in W$ and
\[
\oint_{\partial b}F_\frb^\tw d\cT\ =\ \delta\overline{\eta}_b\cdot(F_\frb^\tw(w_1)\overline{\eta}_{w_1}+ F_\frb^\tw(w_2)\overline{\eta}_{w_2}+F_\frb^\tw(w_3)\overline{\eta}_{w_3}) =\ 0,
\]
where~$w_{1,2,3}\in W$ denote the three neighbors of~$b\in B$ and~$\delta$ is the mesh size of~$(B\cup W)^*$ (or, equivalently, the mesh size of~$W$). These conditions can be equivalently formulated as follows: the function~$G(w):=\overline{\eta}_w F_\frb^\tw(w)$, $w\in W$, is real-valued and satisfies the equation~\eqref{eq:disc-hol-DN} around each~$b\in B$. In other words, the definition of discrete holomorphic functions on a regular triangular grid discussed in~\cite[Section~4]{dynnikov-novikov} trivially fits the t-holomorphicity framework.

\begin{definition}[see~\cite{chelkak-glazman-smirnov}] A real-valued function~$G$ defined on \emph{edges} of a regular honeycomb lattice~$\mathrm{H}$ is called discrete holomorphic if it satisfies local identities
\[
G(e_4)-G(e_1)\ =\ G(e_2)-G(e_5)\ =\ G(e_6)-G(e_3),
\]
\[
G(e_1)+G(e_2)+G(e_3)+G(e_4)+G(e_5)+G(e_6)=0
\]
on edges~$e_1,\ldots,e_6$ (listed counterclockwise) of each face of~$\mathrm{H}$; see Fig.~\ref{fig:hex-ii}.
\label{def:CGS-hol}
\end{definition}

\begin{figure}
\center{\includegraphics[width=0.36\textwidth]{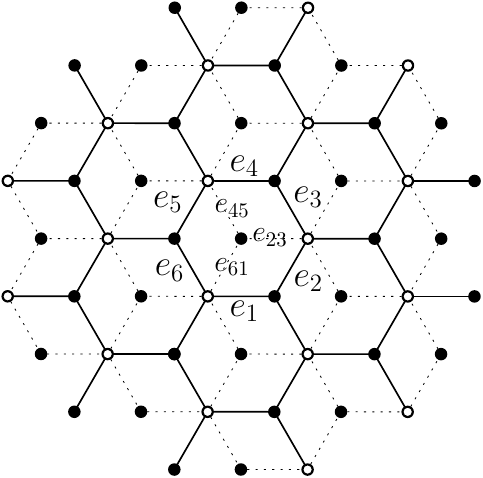}\hskip 0.1\textwidth
\includegraphics[width=0.36\textwidth]{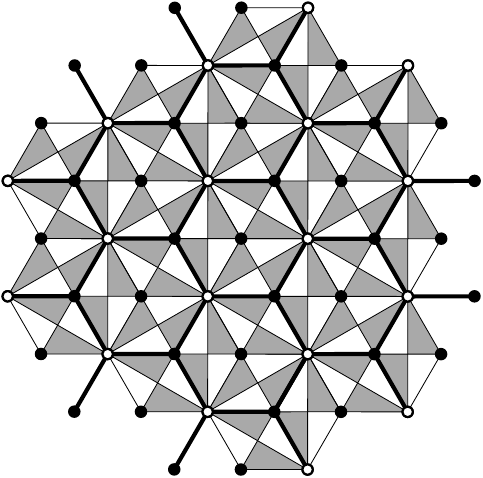}}
\caption{A discrete holomorphic function~$G$ defined on \emph{edges} of a regular honeycomb grid~$\mathrm{H}$ (see Definition~\ref{def:CGS-hol} and~\cite{chelkak-glazman-smirnov}) can be viewed as an \mbox{s-holomorphic} function on the rhombic lattice obtained from~$\mathrm{H}$ (left picture)
or, equivalently, as a t-holomorphic function on the corresponding t-embedding (right picture).}\label{fig:hex-ii}
\end{figure}

{Let $\Gamma$ be} the triangular grid formed by half of the vertices of~$\mathrm{H}$ and~$\Gamma^*$ be its dual honeycomb grid; see Fig.~\ref{fig:hex-ii}. We claim that Definition~\ref{def:CGS-hol} can be equivalently reformulated as the \mbox{s-holo}\-morphicity property for functions defined on the graph~$\diamondsuit=\Lambda^*$, {where, as usual, $\Lambda$ stands for the `diamond graph' of~$\Gamma$.} To this end, recall that s-holomorphic functions on~$\diamondsuit$ can be equivalently viewed as real-valued spinors~$X$ satisfying the propagation equation~\eqref{eq:propagation} on the double cover~$\Upsilon^\times$ of the medial graph of~$\Lambda$. (Note that in our setup one has~$\theta_e=\frac\pi6$ for all rhombi of~$\Lambda$.) Let us fix a section of~$\Upsilon^\times$ by choosing the values
\[
\eta_{b(c)}=\eta_{w(c)}\in\{1,e^{{2\pi i}/{3}},e^{{4\pi i}/{3}},i,ie^{{2\pi i}/{3}},ie^{{4\pi i}/{3}}\}
\]
in the definition~\eqref{eq:etab=etaw=}. If we only know the values of~$X$ at the edges~$e_1,...,e_6$ surrounding a given face~$f$ of~$\mathrm{H}$, then the consistency relations required to define the values~$X(e_{23})$, $X(e_{45})$ and~$X(e_{61})$ so that the propagation equation~\eqref{eq:propagation} holds for the three rhombi inside~$f$ read as
\begin{align*}
\tfrac12 X(e_1)+X(e_2)\ =\ -\tfrac{\sqrt{3}}{2}X(e_{23})\ =\ -(X(e_3)+\tfrac12 X(e_4))\,,\\
\tfrac12 X(e_3)+X(e_4)\ =\ -\tfrac{\sqrt{3}}{2}X(e_{45})\ =\ -(X(e_5)+\tfrac12 X(e_6))\,,\\
\tfrac12 X(e_5)+X(e_6)\ =\ -\tfrac{\sqrt{3}}{2}X(e_{61})\ =\ -(X(e_1)+\tfrac12 X(e_2))\,;
\end{align*}
see Fig.~\ref{fig:hex-ii} for the notation. A simple algebraic manipulation shows that this is equivalent to saying that the values~$X(e_1),\ldots,X(e_6)$ satisfy the equations given in Definition~\ref{def:CGS-hol}. In other words, this definition also fits the general t-holomorphicity framework developed in our paper.


\begin{thebibliography}{10}

\bibitem{affolter}
Niklas~C. Affolter.
\newblock Miquel dynamics, {Clifford} lattices and the dimer model.
\newblock {\em Lett. Math. Phys.}, 111(3):23, 2021.
\newblock Id/No 61.

\bibitem{aggarwal}
Amol {Aggarwal}.
\newblock {Universality for lozenge tiling local statistics}.
\newblock {\em arXiv e-prints}, page arXiv:1907.09991, Jul 2019.

\bibitem{BLRnote}
Nathana{\"e}l {Berestycki}, Beno\^{\i}t {Laslier}, and Gourab {Ray}.
\newblock {A note on dimers and T-graphs}.
\newblock {\em arXiv e-prints}, page arXiv:1610.07994, Oct 2016.

\bibitem{BLR1}
Nathana\"el {Berestycki}, Beno\^{\i}t {Laslier}, and Gourab {Ray}.
\newblock {Dimers and imaginary geometry}.
\newblock {\em {Ann. Probab.}}, 48(1):1--52, 2020.

\bibitem{BdTR-via-dimers}
C\'{e}dric Boutillier, B\'{e}atrice de~Tili\`ere, and Kilian Raschel.
\newblock The {$Z$}-invariant {I}sing model via dimers.
\newblock {\em Probab. Theory Related Fields}, 174(1-2):235--305, 2019.

\bibitem{BSST}
R.~L. {Brooks}, C.~A.~B. {Smith}, A.~H. {Stone}, and W.~T. {Tutte}.
\newblock {The dissection of rectangles into squares}.
\newblock {\em {Duke Math. J.}}, 7:312--340, 1940.

\bibitem{bufetov-gorin}
Alexey Bufetov and Vadim Gorin.
\newblock Fluctuations of particle systems determined by {S}chur generating
  functions.
\newblock {\em Adv. Math.}, 338:702--781, 2018.

\bibitem{chelkak-toolbox}
Dmitry Chelkak.
\newblock Robust discrete complex analysis: a toolbox.
\newblock {\em Ann. Probab.}, 44(1):628--683, 2016.

\bibitem{chelkak-icm2018}
Dmitry Chelkak.
\newblock Planar {I}sing model at criticality: state-of-the-art and
  perspectives.
\newblock In {\em Proceedings of the {I}nternational {C}ongress of
  {M}athematicians---{R}io de {J}aneiro 2018. {V}ol. {IV}. {I}nvited lectures},
  pages 2801--2828. World Sci. Publ., Hackensack, NJ, 2018.

\bibitem{chelkak-s-emb}
Dmitry {Chelkak}.
\newblock {Ising model and s-embeddings of planar graphs}.
\newblock {\em arXiv e-prints}, page arXiv:2006.14559, June 2020.

\bibitem{chelkak-glazman-smirnov}
Dmitry {Chelkak}, Alexander {Glazman}, and Stanislav {Smirnov}.
\newblock {Discrete stress-energy tensor in the loop O(n) model}.
\newblock {\em arXiv e-prints}, page arXiv:1604.06339, April 2016.

\bibitem{CLR2}
Dmitry {Chelkak}, Beno{\^\i}t {Laslier}, and Marianna {Russkikh}.
\newblock {Bipartite dimer model: perfect t-embeddings and Lorentz-minimal
  surfaces}.
\newblock {\em arXiv e-prints}, page arXiv:2109.06272, September 2021.

\bibitem{chelkak-ramassamy}
Dmitry {Chelkak} and Sanjay {Ramassamy}.
\newblock {Fluctuations in the Aztec diamonds via a Lorentz-minimal surface}.
\newblock {\em arXiv e-prints}, page arXiv:2002.07540, February 2020.

\bibitem{chelkak-smirnov-analysis}
Dmitry Chelkak and Stanislav Smirnov.
\newblock Discrete complex analysis on isoradial graphs.
\newblock {\em Adv. Math.}, 228(3):1590--1630, 2011.

\bibitem{chelkak-smirnov-universality}
Dmitry Chelkak and Stanislav Smirnov.
\newblock Universality in the 2{D} {I}sing model and conformal invariance of
  fermionic observables.
\newblock {\em Invent. Math.}, 189(3):515--580, 2012.

\bibitem{cohn-kenyon-propp}
Henry Cohn, Richard Kenyon, and James Propp.
\newblock A variational principle for domino tilings.
\newblock {\em J. Amer. Math. Soc.}, 14(2):297--346, 2001.

\bibitem{colomo-sportiello}
F.~Colomo and A.~Sportiello.
\newblock Arctic curves of the six-vertex model on generic domains: the tangent
  method.
\newblock {\em J. Stat. Phys.}, 164(6):1488--1523, 2016.

\bibitem{yellow-book}
Philippe {Di Francesco}, Pierre {Mathieu}, and David {S\'en\'echal}.
\newblock {\em {Conformal field theory}}.
\newblock New York, NY: Springer, 1997.

\bibitem{dubedat-bosonization}
Julien {Dub{\'e}dat}.
\newblock {Exact bosonization of the Ising model}.
\newblock {\em arXiv e-prints}, page arXiv:1112.4399, Dec 2011.

\bibitem{dubedat-SL2}
Julien Dub\'{e}dat.
\newblock Double dimers, conformal loop ensembles and isomonodromic
  deformations.
\newblock {\em J. Eur. Math. Soc. (JEMS)}, 21(1):1--54, 2019.

\bibitem{duffin-rhombic}
R.~J. {Duffin}.
\newblock {Potential theory on a rhombic lattice}.
\newblock {\em {J. Comb. Theory}}, 5:258--272, 1968.

\bibitem{dynnikov-novikov}
I.~A. {Dynnikov} and S.~P. {Novikov}.
\newblock {Geometry of the triangle equation on two-manifolds}.
\newblock {\em {Mosc. Math. J.}}, 3(2):419--438, 2003.

\bibitem{ferrand-44}
Jacqueline {Ferrand}.
\newblock {Fonctions pr\'eharmoniques et fonctions pr\'eholomorphes}.
\newblock {\em {Bull. Sci. Math., II. S\'er.}}, 68:152--180, 1944.

\bibitem{giuliani-mastropietro-toninelli}
Alessandro {Giuliani}, Vieri {Mastropietro}, and Fabio~Lucio {Toninelli}.
\newblock {Non-integrable dimers: universal fluctuations of tilted height
  profiles}.
\newblock {\em {Commun. Math. Phys.}}, 377(3):1883--1959, 2020.

\bibitem{gorin-lectures}
Vadim {Gorin}.
\newblock {\em {Lectures on random lozenge tilings}}, volume 193.
\newblock Cambridge: Cambridge University Press, 2021.

\bibitem{gurel-jerison-nachmias}
Ori {Gurel-Gurevich}, Daniel~C. {Jerison}, and Asaf {Nachmias}.
\newblock {The Dirichlet problem for orthodiagonal maps}.
\newblock {\em {Adv. Math.}}, 374:53, 2020.
\newblock Id/No 107379.

\bibitem{hull-origami}
Thomas Hull, editor.
\newblock {\em Origami{{\(^3\)}}. {Proceedings} of the third international
  meeting of origami science, mathematics, and education (3OSME) {Asilomar},
  {CA}, {USA}, 2001}.
\newblock Natick, MA: A. K. Peters, 2002.

\bibitem{kenyon-isorad}
R.~{Kenyon}.
\newblock {The Laplacian and Dirac operators on critical planar graphs}.
\newblock {\em {Invent. Math.}}, 150(2):409--439, 2002.

\bibitem{kenyon-gff-a}
Richard Kenyon.
\newblock Conformal invariance of domino tiling.
\newblock {\em Ann. Probab.}, 28(2):759--795, 2000.

\bibitem{kenyon-gff-b}
Richard Kenyon.
\newblock Dominos and the {G}aussian free field.
\newblock {\em Ann. Probab.}, 29(3):1128--1137, 2001.

\bibitem{kenyon-honeycomb}
Richard Kenyon.
\newblock Height fluctuations in the honeycomb dimer model.
\newblock {\em Comm. Math. Phys.}, 281(3):675--709, 2008.

\bibitem{kenyon-lectures}
Richard Kenyon.
\newblock Lectures on dimers.
\newblock In {\em Statistical mechanics}, volume~16 of {\em IAS/Park City Math.
  Ser.}, pages 191--230. Amer. Math. Soc., Providence, RI, 2009.

\bibitem{KLRR}
Richard {Kenyon}, Wai~Yeung {Lam}, Sanjay {Ramassamy}, and Marianna {Russkikh}.
\newblock {Dimers and circle patterns}.
\newblock {\em arXiv e-prints}, page arXiv:1810.05616, Oct 2018.

\bibitem{kenyon-okounkov}
Richard Kenyon and Andrei Okounkov.
\newblock Limit shapes and the complex {B}urgers equation.
\newblock {\em Acta Math.}, 199(2):263--302, 2007.

\bibitem{kenyon-sheffield}
Richard~W. Kenyon and Scott Sheffield.
\newblock Dimers, tilings and trees.
\newblock {\em J. Combin. Theory Ser. B}, 92(2):295--317, 2004.

\bibitem{laslier-21}
Beno\^{\i}t {Laslier}.
\newblock {Central limit theorem for lozenge tilings with curved limit shape}.
\newblock {\em arXiv e-prints}, page arXiv:2102.05544, February 2021.

\bibitem{zhongyang-li}
Zhongyang Li.
\newblock Conformal invariance of dimer heights on isoradial double graphs.
\newblock {\em Ann. Inst. Henri Poincar\'{e} D}, 4(3):273--307, 2017.

\bibitem{Lis}
Marcin {Lis}.
\newblock {Circle patterns and critical Ising models}.
\newblock {\em {Commun. Math. Phys.}}, 370(2):507--530, 2019.

\bibitem{mercat-CMP}
Christian {Mercat}.
\newblock {Discrete Riemann surfaces and the Ising model.}
\newblock {\em {Commun. Math. Phys.}}, 218(1):177--216, 2001.

\bibitem{romaskevich}
Olga {Paris-Romaskevich}.
\newblock {Trees and flowers on a billiard table}.
\newblock {\em arXiv e-prints}, page arXiv:1907.01178, Jul 2019.

\bibitem{petrov}
Leonid Petrov.
\newblock Asymptotics of uniformly random lozenge tilings of polygons.
  {G}aussian free field.
\newblock {\em Ann. Probab.}, 43(1):1--43, 2015.

\bibitem{russkikh-t}
Marianna Russkikh.
\newblock Dimers in piecewise {T}emperleyan domains.
\newblock {\em Comm. Math. Phys.}, 359(1):189--222, 2018.

\bibitem{russkikh-h}
Marianna {Russkikh}.
\newblock {Dominos in hedgehog domains}.
\newblock {\em {Ann. Inst. Henri Poincar\'e D, Comb. Phys. Interact. (AIHPD)}},
  8(1):1--33, 2021.

\bibitem{sheffield-GFF}
Scott Sheffield.
\newblock Gaussian free fields for mathematicians.
\newblock {\em Probab. Theory Relat. Fields}, 139(3-4):521--541, 2007.

\bibitem{smirnov-annals}
Stanislav Smirnov.
\newblock Conformal invariance in random cluster models. {I}. {H}olomorphic
  fermions in the {I}sing model.
\newblock {\em Ann. of Math. (2)}, 172(2):1435--1467, 2010.

\bibitem{smirnov-icm2010}
Stanislav Smirnov.
\newblock Discrete complex analysis and probability.
\newblock In {\em Proceedings of the {I}nternational {C}ongress of
  {M}athematicians. {V}olume {I}}, pages 595--621. Hindustan Book Agency, New
  Delhi, 2010.

\bibitem{smirnov-boykiy-private}
Stanislav Smirnov and Roman Boykiy.
\newblock Private communication, 2014.

\bibitem{thurston-height}
William~P. Thurston.
\newblock Conway's tiling groups.
\newblock {\em Amer. Math. Monthly}, 97(8):757--773, 1990.

\end{thebibliography}

\end{document}